\documentclass{amsart}

\usepackage[all]{xy}
\usepackage{amsmath}
\usepackage{hyperref}
\usepackage{amsfonts,graphics,amsthm,amscd,latexsym}
\usepackage{epsfig}
\usepackage{flafter}
\usepackage{mathtools}
\usepackage{comment}
\usepackage{stmaryrd}
\usepackage{tabularx}
\usepackage{pst-node}
\usepackage{enumitem}
\usepackage{tikz-cd}
\hypersetup{
	colorlinks=true,    
	linkcolor=blue,          
	citecolor=blue,      
	filecolor=blue,      
	urlcolor=blue           
}
\usepackage{placeins}

\usepackage{pgfplots}
\pgfplotsset{compat=1.15}
\usepackage{mathrsfs}
\usepackage{tikz}
\usetikzlibrary{graphs,positioning,arrows,shapes.misc,decorations.pathmorphing}

\tikzset{
	>=stealth,
	every picture/.style={thick},
	graphs/every graph/.style={empty nodes},
}

\tikzstyle{vertex}=[
draw,
circle,
fill=black,
inner sep=1pt,
minimum width=5pt,
]
\usepackage[position=top]{subfig}
\usepackage{amssymb}
\usepackage{color}

\calclayout

\usetikzlibrary{decorations.pathmorphing}
\tikzstyle{printersafe}=[decoration={snake,amplitude=0pt}]

\newcommand{\Aut}{\operatorname{Aut}}
\newcommand{\Bir}{\operatorname{Bir}}

\newcommand{\PDiv}{\operatorname{PDiv}}
\newcommand{\WDiv}{\operatorname{WDiv}}
\newcommand{\Pic}{\operatorname{Pic}}
\newcommand{\Nef}{\operatorname{Nef}}

\newcommand{\Cl}{\operatorname{Cl}}

\newcommand{\codim}{\operatorname{codim}}

\newcommand{\Mov}{\operatorname{Mov}}

\newcommand{\Spec}{\operatorname{Spec}}

\newcommand{\Hom}{\operatorname{Hom}}

\newcommand{\relint}{\operatorname{relint}}

\newcommand{\divv}{\operatorname{div}}

\newcommand{\Cox}{\operatorname{Cox}}

\newcommand{\PsAut}{\operatorname{PsAut}}

\newcommand{\MovE}{\overline{\operatorname{Mov}}^e}

\newcommand{\Amp}{\operatorname{Amp}}
\newcommand{\Eff}{\operatorname{Eff}}

\newcommand{\pp}{\mathbb{P}}

\newcommand{\qq}{\mathbb{Q}}
\newcommand{\zz}{\mathbb{Z}}

\newcommand{\rr}{\mathbb{R}}

\newcommand{\kk}{\mathbb{K}}

\usepackage{tikz}
\usetikzlibrary{matrix,arrows,decorations.pathmorphing}
\usetikzlibrary{arrows}

\definecolor{uuuuuu}{rgb}{0.26666666666666666,0.26666666666666666,0.26666666666666666}

\newtheorem{introthm}{Theorem}

\newtheorem{theorem}{Theorem}[section]
\newtheorem{lemma}[theorem]{Lemma}
\newtheorem{proposition}[theorem]{Proposition}
\newtheorem{corollary}[theorem]{Corollary}

\theoremstyle{definition}

\newtheorem{definition}[theorem]{Definition}

\theoremstyle{remark}
\newtheorem{remark}[theorem]{Remark}

\newtheorem{example}[theorem]{Example}

\numberwithin{equation}{section}



\begin{document}
	
	\title[Computing Cox rings via the cone conjecture]{Computing Cox rings via the cone conjecture}
	
	\author[T.~Oda]{Tomoki Oda}
	\address{UCLA Mathematics Department, Box 951555, Los Angeles, CA 90095-1555, USA
	}
	\email{tomokioda0723@math.ucla.edu}
	
	\author[J.I.~Y\'a\~nez]{Jos\'e Ignacio Y\'a\~nez}
	\address{Departamento de Matem\'atica, Universidad T\'ecnica Federico Santa Mar\'ia, Avenida Espa\~na 1680, Valpara\'iso, Chile}
	\email{jose.yaneze@usm.cl}
	
	\author[J.P.~Z\'u\~niga]{Juan Pablo Z\'u\~niga}
	\address{Facultad de Matem\'aticas, Pontificia Universidad Cat\'olica de Chile, Santiago, Chile.
	}
	\email{jpzuniga3@uc.cl}
	
	\begin{abstract}
		We initiate a program to study the Cox ring of Calabi--Yau varieties, employing the notion of Morrison--Kawamata dream spaces. In this setting, we establish an analogue of the Hu--Keel GIT constructions for Mori dream spaces. More precisely, for a Morrison--Kawamata dream space $X$, we establish a correspondence between the small $\mathbb{Q}$-factorial modifications of $X$ and the GIT quotients of $\operatorname{Spec}\operatorname{Cox}(X)$. We further show that the Cox ring of a Morrison--Kawamata dream space is a filtered direct limit of subalgebras, each of which is an inverse limit of finitely generated $\Cl(X)$-graded $\mathbb{K}$-algebras. As an application, we give an explicit presentation of the Cox ring of a very general hypersurface of multidegree $(2,\dots,2,n+1)$ in $(\mathbb{P}^1)^m\times \mathbb{P}^n$. Furthermore, we prove that the Cox ring of such a hypersurface is of dense $F$-pure type.
		
	\end{abstract}
	\maketitle
	
	\setcounter{tocdepth}{1}
	\tableofcontents

	\section{Introduction}
	The Cox ring is a fundamental invariant of an algebraic variety $X$ over a field $\mathbb{K}$: it encodes the global sections of all line bundles over $X$ into a single graded ring $\Cox(X)$. In general, it is not a finitely generated $\mathbb{K}$-algebra. When $\Cox(X)$ is finitely generated, the variety $X$ is called a \emph{Mori dream space}. This notion was introduced by Hu and Keel in \cite{HK00}. When $X$ is a Mori dream space, it can be recovered as a quotient of $\Spec \Cox(X)$ by an action of the quasi-torus $\operatorname{Spec} \mathbb{K}[\Cl(X)]$ (see~\cite[Proposition~2.9]{HK00}). In this paper, we develop an analogue of their construction for certain non-Mori dream spaces, replacing the finite generation of the Cox ring by a structural hypothesis motivated by the Morrison--Kawamata cone conjecture.
	
	Since the early 2000s, a series of works clarified striking links between the singularities of the Cox ring and the global geometry of Mori dream spaces. For instance, if $X$ is of Fano type, then $\Cox(X)$ is finitely generated by the results of Birkar, Cascini, Hacon, and McKernan~\cite{BCHM10}. Furthermore, $X$ is of Fano type if and only if $\Spec\Cox(X)$ is klt at the origin, and also if and only if $\Spec\Cox(X)$ is Gorenstein and canonical everywhere by \cite[Theorem~1.1]{GOST15} and \cite[Theorem~2]{Braun}. A Mori dream space $X$ is toric if and only if $\Cox(X)$ is a polynomial ring \cite[Corollary~2.10]{HK00}. A Mori dream space $X$ is of Calabi--Yau type if and only if  $\Spec \Cox(X)$ is log canonical at the origin \cite[Theorem~1.1]{Kawamata-Okawa}. More recently, Enwright, Francone, Moraga, and Spink proved that a $\mathbb Q$-factorial Fano variety is of cluster type if and only if its Cox ring is a $\Cl(X)$-graded mutation semigroup algebra \cite[Theorem~1.5]{EFMS25}. On the other hand, comparatively little is known about Cox rings beyond the Mori dream space setting. A major source of examples comes from the work of Gross, Hacking, and Keel~\cite{GHK15} and Mandel~\cite{Mandel19}, who showed that many cluster algebras arise as Cox rings of varieties obtained from toric pairs via non-toric blow-ups.

	\medskip
	We initiate a program to investigate the Cox ring of Calabi--Yau varieties. Building on the Morrison--Kawamata cone conjecture~\cite{K88,Morrison93}, and Totaro’s klt version~\cite{T10}, we expect the Morrison–Kawamata property to hold for klt Calabi--Yau pairs. These conjectures serve as a substitute for the rational polyhedrality phenomenon of Mori dream spaces. Recently, Choi, Li, Li, and Zhou~\cite{CRX25} proposed a class of varieties called \emph{Morrison--Kawamata dream spaces} (Definition~\ref{def:KMMDspace}), giving an axiomatic framework for varieties satisfying the Morrison--Kawamata cone conjecture, under the existence of good minimal models.
	
	\medskip
	Our first main result shows that the description as a GIT quotient in~\cite[Proposition~2.9]{HK00} extends to Morrison--Kawamata dream spaces, even when the scheme $\Spec \Cox(X)$ is non-Noetherian.
	Throughout the paper a small $\qq$-factorial modification is abbreviated as SQM, and we denote the movable effective cone of $X$ as $\MovE(X):=\overline{\Mov}(X)\cap \Eff(X)$.
	
	\begin{introthm}\label{intro:KMMDS}
		Let $X$ be a Morrison--Kawamata dream space with finitely generated class group $\Cl(X)$. Let $H:=\Hom(\Cl(X),\mathbb G_m)$ be the quasi-torus of $X$, and let $\overline{X}:=\Spec \Cox(X)$ be the total coordinate space. Then the following properties hold:
		\begin{enumerate}
			
			\item For every ample class $D$ and its associated character $\chi_D$, the GIT quotient recovers $X$:
			\[
			X \;\cong\; X_D \;:=\; \overline X^{ss}(\chi_D) \sslash H .
			\]
			Moreover, for every SQM $\varphi\colon X \dashrightarrow X'$, there exists a class $D' \in \MovE(X)$
			such that
			\[
			X'\cong X_{D'}=\overline{X}^{ss} (\chi_{D'})\sslash H.
			\]
			
			\item  Let $Y$ be an SQM of $X$ and let $\sigma \preceq \Nef(Y)$ be a face.
			If $D,D'\in \relint(\sigma)$ are $\mathbb Q$-Cartier divisor classes, then $X_D \cong X_{D'}$.
			
			\item If $D$ and $D'$ lie in the interiors of two adjacent chambers of $\MovE(X)$ separated by a wall contained in $\operatorname{int}\MovE(X)$, then $D$ and $D'$ are big and there exists a small birational map
			\[
			X_D \dashrightarrow X_{D'} \quad \text{factoring through a common contraction} \quad X_D \longrightarrow Z \longleftarrow X_{D'}.
			\]
			Furthermore, for any $D_1,D_2$ lying in the interiors of chambers of $\MovE(X)$, the induced birational map $X_{D_1}\dashrightarrow X_{D_2}$ between the corresponding models is an isomorphism in codimension one.
			\item  $\Cox(X)$ is a filtered direct limit of subalgebras, each of which is an inverse limit of finitely generated $\Cl(X)$-graded $\mathbb{K}$-algebras.
		\end{enumerate}
	\end{introthm}
	
	\medskip
	As a first step towards understanding the Cox ring of Morrison--Kawamata dream spaces, 
	we apply Theorem~\ref{intro:KMMDS} to compute the Cox ring of very general smooth 
	hypersurfaces of multihomogeneous degree $(2,\dots,2,n+1)$ 
	in the product of projective spaces $(\mathbb{P}^1)^m \times \mathbb{P}^n$. 
	We refer to such hypersurfaces as \emph{Wehler type hypersurfaces}. 
	These varieties have been extensively studied from the perspectives  of the dynamics of birational automorphisms~\cite{CO15},  the dynamics of rational points~\cite{Benjamin}, and complex dynamics~\cite{Oguiso, McMullen}. 
	In particular, Cantat and Oguiso~\cite{CO15} and the second author~\cite{Yáñ22} 
	established that Wehler type hypersurfaces are Morrison--Kawamata dream spaces,
	with polyhedral fundamental domain $\Pi = \Nef(X)$.
	
	In \cite{Ot15}, Ottem computes the Cox ring of hypersurfaces in $\pp^1\times \pp^n$ that are Mori dream spaces. This result was later generalized by \cite{HLU24} and \cite{PIS25}. 
	In a concurrent work, Artebani, Laface, and Ugaglia~\cite{ALU26} study the Cox
	rings of smooth anticanonical Calabi--Yau hypersurfaces in smooth toric Fano varieties, using the combinatorics of primitive pairs of the ambient Fano polytope. They obtain explicit presentations of $\Cox(X)$ in the cases where
	$X$ is a Mori dream space, and they characterize combinatorial configurations that force $\Bir(X)$ to be infinite.
	
	Throughout the paper, we set $\mathbb P := (\mathbb P^1)^m \times \mathbb P^n$, where either $m,n\ge 2$ or $n=0$ and $m\ge 4$, with homogeneous coordinates $\bigl([x_{10}:x_{11}],\ldots,[x_{m0}:x_{m1}],[y_0:\cdots:y_n]\bigr)$. For each $k\in\{1,\dots,m\}$, the defining equation $F$ of $X$ is quadratic in the $k$-th $\mathbb{P}^1$ factor:
	\[
	F \;:=\; R_{(k,0)}\,x_{k0}^2 \;+\; R_{(k,1)}\,x_{k0}x_{k1} \;+\; R_{(k,2)}\,x_{k1}^2,
	\]
	where $R_{(k,i)}$ are multi-homogeneous polynomials of multidegree $(2,\dots,0_k,\dots,2,n+1)$. The birational automorphism group $\Bir(X)$ is generated by the covering involutions $\iota_k\colon X\dashrightarrow X$ induced by the projection $\pi_k\colon X\to(\mathbb{P}^1)^{m-1}\times\mathbb{P}^n$ forgetting the $k$-th factor (see \cite{CO15,Yáñ22} and Theorem~\ref{thm:Wehler-bir}):
	\[
	\operatorname{Bir}(X) = \langle \iota_1,\ldots,\iota_m\rangle \cong (\mathbb{Z}/2\mathbb{Z})^{*m}.
	\]
	
	In this situation, the induced action on the cone $\MovE(X)$ is defined by reflections (see Figure~\ref{figureone}).
	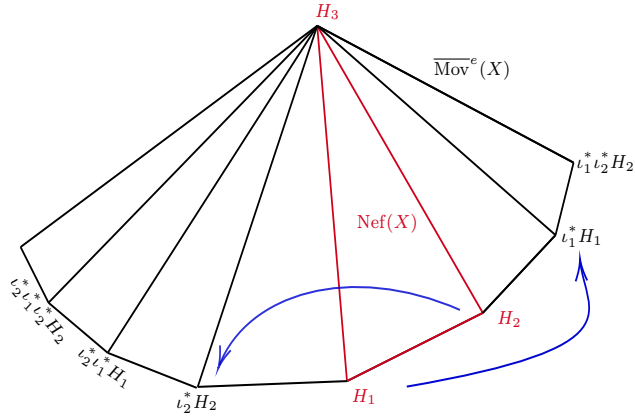
\begin{figure}[htb]
		\centering
		\resizebox{8.6cm}{5.5cm}{%
			\tikzset{every picture/.style={line width=0.75pt}}
			\begin{tikzpicture}[x=0.75pt,y=0.75pt,yscale=-1,xscale=1]
				
				\draw [color={rgb, 255:red, 208; green, 2; blue, 27 }  ,draw opacity=1 ]   (335.26,24.5) -- (352.03,199.96) ;
				\draw [color={rgb, 255:red, 208; green, 2; blue, 27 }  ,draw opacity=1 ]   (335.26,24.5) -- (400,122.86) -- (428.57,166.27) ;
				\draw    (352.03,199.96) -- (428.57,166.27) ;
				\draw    (335.26,24.5) -- (469.54,127.85) ;
				\draw    (335.26,24.5) -- (479.7,92.07) ;
				\draw    (469.54,127.85) -- (428.57,166.27) ;
				\draw    (469.54,127.85) -- (479.7,92.07) ;
				\draw [color={rgb, 255:red, 208; green, 2; blue, 27 }  ,draw opacity=1 ]   (352.03,199.96) -- (428.57,166.27) ;
				\draw    (335.26,24.5) -- (479.7,92.07) ;
				\draw    (469.54,127.85) -- (428.57,166.27) ;
				\draw    (335.26,24.5) -- (267.91,202.81) ;
				\draw    (267.91,202.81) -- (352.03,199.96) ;
				\draw    (217.19,185.98) -- (267.38,203.28) ;
				\draw    (183.85,161.24) -- (217.19,185.98) ;
				\draw [fill={rgb, 255:red, 155; green, 155; blue, 155 }  ,fill opacity=1 ]   (335.26,24.5) -- (217.19,185.98) ;
				\draw    (335.26,24.5) -- (183.85,161.24) ;
				\draw    (335.26,24.5) -- (168.11,133.75) ;
				\draw    (168.11,133.75) -- (183.85,161.24) ;
				\draw [color={rgb, 255:red, 2; green, 6; blue, 208 }  ,draw opacity=1 ]   (385.71,202.68) .. controls (515.6,183.09) and (486.77,166.08) .. (483.53,141.47) ;
				\draw [shift={(483.34,139.56)}, rotate = 85.87] [color={rgb, 255:red, 2; green, 6; blue, 208 }  ,draw opacity=1 ][line width=0.75]    (10.93,-3.29) .. controls (6.95,-1.4) and (3.31,-0.3) .. (0,0) .. controls (3.31,0.3) and (6.95,1.4) .. (10.93,3.29)   ;
				\draw [color={rgb, 255:red, 2; green, 6; blue, 208 }  ,draw opacity=0.83 ]   (415.68,164.81) .. controls (353.82,141.12) and (295.72,156.79) .. (280.14,192.06) ;
				\draw [shift={(279.45,193.69)}, rotate = 291.7] [color={rgb, 255:red, 2; green, 6; blue, 208 }  ,draw opacity=0.83 ][line width=0.75]    (10.93,-3.29) .. controls (6.95,-1.4) and (3.31,-0.3) .. (0,0) .. controls (3.31,0.3) and (6.95,1.4) .. (10.93,3.29)   ;
				
				\draw (356.05,115.8) node [anchor=north west][inner sep=0.75pt]  [font=\scriptsize,color={rgb, 255:red, 208; green, 2; blue, 27 }  ,opacity=1 ] [align=left] {$\displaystyle \operatorname{Nef}( X)$};
				\draw (332.63,12.13) node [anchor=north west][inner sep=0.75pt]  [font=\scriptsize,color={rgb, 255:red, 208; green, 2; blue, 27 }  ,opacity=1 ] [align=left] {$\displaystyle H_{3}$};
				\draw (435.12,163.91) node [anchor=north west][inner sep=0.75pt]  [font=\scriptsize,color={rgb, 255:red, 208; green, 2; blue, 27 }  ,opacity=1 ] [align=left] {$\displaystyle H_{2}$};
				\draw (353.29,202.25) node [anchor=north west][inner sep=0.75pt]  [font=\scriptsize,color={rgb, 255:red, 208; green, 2; blue, 27 }  ,opacity=1 ] [align=left] {$\displaystyle H_{1}$};
				\draw (254.36,203.8) node [anchor=north west][inner sep=0.75pt]  [font=\scriptsize,color={rgb, 255:red, 0; green, 0; blue, 0 }  ,opacity=1 ] [align=left] {$\displaystyle \iota _{2}^{*} H_{2}$};
				\draw (471.19,121.92) node [anchor=north west][inner sep=0.75pt]  [font=\scriptsize,color={rgb, 255:red, 0; green, 0; blue, 0 }  ,opacity=1 ] [align=left] {$\displaystyle \iota _{1}^{*} H_{1}$};
				\draw (399.38,38.73) node [anchor=north west][inner sep=0.75pt]   [align=left] {\textsuperscript{$\displaystyle \operatorname{\overline{Mov}}^{e}( X)$}};
				\draw (203.25,178.18) node [anchor=north west][inner sep=0.75pt]  [font=\scriptsize,color={rgb, 255:red, 0; green, 0; blue, 0 }  ,opacity=1 ,rotate=-30] [align=left] {$\displaystyle \iota _{2}^{*} \iota _{1}^{*} H_{1}$};
				\draw (167.77,143.6) node [anchor=north west][inner sep=0.75pt]  [font=\scriptsize,color={rgb, 255:red, 0; green, 0; blue, 0 }  ,opacity=1 ,rotate=-46.74] [align=left] {$\displaystyle \iota _{2}^{*} \iota _{1}^{*} \iota _{2}^{*} H_{2}$};
				\draw (479.84,85.39) node [anchor=north west][inner sep=0.75pt]  [font=\scriptsize,color={rgb, 255:red, 0; green, 0; blue, 0 }  ,opacity=1 ] [align=left] {$\displaystyle \iota _{1}^{*} \iota _{2}^{*} H_{2}$};
				
			\end{tikzpicture}
		}%
		\caption{The involution $\iota_k^*$ induces an action on $\MovE(X)$ via pseudo-reflections. Shown here is the case of multidegree $(2,2,3)$ in $\pp^1\times \pp^1\times \pp^2$.}
		\label{figureone}
	\end{figure}
	
	To translate the cone decomposition into an algebraic description of the Cox ring, we analyze the induced action on the ring generators. We regard the coordinate $x_{ki}$ as a global section of the line bundle $\mathcal{O}_{\mathbb{P}}(0,\dots,1_k,\dots,0)$ on the ambient space $\mathbb{P}$. We denote the restriction of $x_{ki}$ to $X$ by the same symbol and define $H_k$ to be the restriction of the line bundle to $X$.
	The covering involution $\iota_k$ induces a pullback map $\iota_k^*$ on global sections. Specifically, $x_{ki}$ and its pullback $\iota_k^*(x_{ki})$ satisfy a \emph{Vieta-type relation} (see Proposition~\ref{multiplication-involution}):
	\[
	\iota_k^*(x_{ki})x_{ki} - R_{(k,2-2i)} = 0.
	\]
	These relations are fundamental to the structure of the Cox ring of Wehler type hypersurfaces.

	Our construction begins with the $\Nef(X)$-graded section ring
	\[
	\mathcal R_0=\bigoplus_{D\in\Nef(X)\cap\Cl(X)}H^0(X,\mathcal O_X(D)),
	\]
	whose presentation is described in Theorem~\ref{thm:Nef-coordinate}. We then extend this description across the chamber decomposition by iteratively applying pullbacks under the covering involutions. For each $s\ge 0$, let $W_s$ be the set of reduced words of length at most $s$ in the generators $\iota_k$. Set
	\[
	\widetilde{\mathcal R}_s
	:=
	\mathbb K\bigl[\,w^{*}(x_{ki}),\,y_0,\dots,y_n\bigm| w\in W_s,\; 1\le k\le m,\; i\in\{0,1\}\,\bigr].
	\]
	For $s\ge 1$, define the ideal
	\[
	I_s
	:=
	\Bigl\langle\, w^{*}\!\bigl(x_{ki}\,\iota_k^{*}(x_{ki})-R_{(k,2-2i)}\bigr),\; w^{*}(F)
	\;\Bigm|\; w\in W_{s-1},\; 1\le k\le m,\; i\in\{0,1\}\,\Bigr\rangle,
	\]
	the monomial $m_s:=\prod_{w\in W_{s-1},\,k,i} w^{*}(x_{ki})$, and $J_s:=\langle m_s\rangle\subseteq\widetilde{\mathcal R}_s$. With the base-case convention $I_0:=\langle F\rangle$, $m_0:=1$, $J_0:=\langle 1\rangle$. Set $\overline I_s:=(I_s:J_s^{\,\infty})$ and
	\begin{equation}\label{eq:filtered-ring}
		\mathcal{R}'_s := \widetilde{\mathcal{R}}_s / \overline{I}_s.
	\end{equation}
	\begin{introthm}\label{thm:cox-limit}
		Let $X$ be a Wehler type hypersurface. Then $\mathcal R'_s$ is a $\mathbb K$-subalgebra of $\Cox(X)$ for every $s\ge 0$. Furthermore, the Cox ring of $X$ is isomorphic to the filtered direct limit
		\[
		\Cox(X)\;\cong\;\varinjlim_s \mathcal R'_s.
		\]
	\end{introthm}
	
	\begin{remark}
		The ideal $I_1$ contains all Vieta type relations. However, $I_1$ does not account for all relations that appear in $\mathcal{R}_0$. Passing to the ideal quotient $\overline{I}_1:=(I_1\,:\,J_1^\infty)$, we obtain all necessary relations to describe $\mathcal{R}_0$ (see Theorem \ref{thm:Nef-coordinate}). For $s\geq 0$, the saturated ideal $\bar{I}_s$ is the kernel of the evaluation homomorphism $\widetilde{\mathcal{R}}_s\to \Cox(X)$ sending each formal symbol $w^{*}(x_{ki})$ to the corresponding pullback section and $y_l$ to itself.
	\end{remark}

	The direct limit in Theorem~\ref{thm:cox-limit} already relies on the chamber decomposition of $\MovE(X)$ given by Theorem~\ref{Kawamata-Morrison}. The same
	structure underlies the next main result that provides an alternative presentation of $\Cox(X)$. More precisely, it replaces the limit construction by an explicit presentation, bypassing the ideal saturations entirely. To achieve this, in Theorem~\ref{thm:Nef-no-saturation} we revisit the nef graded section ring using a different generating set. In addition to the coordinate sections $x_{ki}$ and $y_j$, we introduce the additional sections
	\[
	\zeta_k:=x_{k0}\iota_k^{*}(x_{k1})\qquad k=1,\dots,m.
	\]
	These sections do not arise as restrictions of ambient sections, and they are essential for obtaining a presentation of $\mathcal R_0$ that does not rely on ideal saturation.
	For each pair $w,w'\in\Bir(X)$ and indices $k,l,i,j$ such that
	$w^{*}(x_{ki})\,w'^{*}(x_{lj})$ has degree lying in $\Nef(X)$, its image in $\Cox(X)$ lies in the nef-graded section subring. We fix a polynomial $H_{w,w',ki,lj}$ in the
	nef-chamber generators representing it, so that
	\[
	w^{*}(x_{ki})\,w'^{*}(x_{lj})=H_{w,w',ki,lj}\quad\text{in }\Cox(X).
	\]
	In particular, for each $k<l$ there is a polynomial $H_{kl}$ in the nef-chamber
	generators such that $\zeta_k\zeta_l=x_{k0}x_{l0}H_{kl}$ in $\Cox(X)$. The following presentation is established in Theorem \ref{thm:cox-pres}.
	
	\begin{introthm}\label{thm:cox-intro}
		Let $X$ be a Wehler type hypersurface. The natural map sending each formal generator
		to the corresponding section induces an isomorphism of $\Cl(X)$--graded
		$\mathbb K$--algebras
		\[
		\Cox(X)\;\cong\;
		\mathbb K\bigl[\,w^{*}(x_{ki}),\,w^{*}(\zeta_k),\,y_j
		\,\bigm|\, w\in\Bir(X),\,1\le k\le m,\,i\in\{0,1\},\,0\le j\le n\,\bigr]\big/I,
		\]
		where the ideal $I$ is generated by:
		\begin{enumerate}[label=\textup{(\roman*)}]
			\item For each $w\in\Bir(X)$ and each $1\le k,l\le m$ with $k<l$:
			\begin{enumerate}[label=\textup{(\alph*)}]
				\item $w^{*}\!\bigl(\zeta_k - x_{k0}\,\iota_k^{*}(x_{k1})\bigr)$,
				\item $w^{*}\!\bigl(\zeta_k^{2}+R_{(k,1)}\zeta_k+R_{(k,0)}R_{(k,2)}\bigr)$,
				\item $w^{*}\!\bigl(x_{k1}\zeta_k - R_{(k,0)}x_{k0}\bigr)$
				and $w^{*}\!\bigl(x_{k0}\zeta_k+R_{(k,1)}x_{k0}+R_{(k,2)}x_{k1}\bigr)$,
				\item $w^{*}\!\bigl(\zeta_k\zeta_l - x_{k0}x_{l0}H_{kl}\bigr)$;
			\end{enumerate}
			\item For each $u,w,w'\in\Bir(X)$ and $k,l,i,j$ as above: the elements
			$u^{*}\!\bigl(w^{*}(x_{ki})\,w'^{*}(x_{lj})-H_{w,w',ki,lj}\bigr)$.
		\end{enumerate}
	\end{introthm}
	
	As in the Mori dream space setting, one expects to extract global information about Morrison--Kawamata dream spaces from the local structure of their Cox rings. Here, we show that the Cox rings of Wehler type hypersurfaces are of dense $F$-pure type. $F$-purity is often viewed as a characteristic-$p$ analog of log canonical singularities. In characteristic $0$, the Cox rings of Calabi--Yau type Mori dream spaces are known to have log canonical singularities by the work of Kawamata and Okawa~\cite{Kawamata-Okawa}. 
	
	\begin{introthm}\label{F-puretype}
		Let $X$ be a Wehler type hypersurface defined over $\mathbb{C}$. Then the Cox ring of $X$ is of dense $F$-pure type.
	\end{introthm}

	\subsection*{Acknowledgements}
	We are grateful to Joaquín Moraga for proposing the question and his comments about the manuscript. We also thank Burt Totaro for his numerous comments on this project.
	This project was initiated during the \href{https://www.math.ucla.edu/~jmoraga/JAGW2025}{Junior Algebraic Geometry Workshop 2025}, held at UCLA. The first author was partially supported by NSF research grant DMS-2443425. The third author was supported by the ANID National Doctoral Scholarship 2022–21221224.
	
	\section{Preliminaries}
	Throughout the paper, unless stated otherwise, $\mathbb{K}$ denotes an algebraically closed field of characteristic $0$.
	\begin{definition}(\cite[Construction~1.4.2.1]{ADHL15})\label{def:cox}
		Let $X$ be a normal projective variety with $\Pic^0(X)=0$ and a finitely generated class group $\Cl(X)$.
		Let $\mathcal M\subset \WDiv(X)$ be a finitely generated submonoid with $0\in\mathcal M$.
		The \emph{divisorial algebra} associated to $(X,\mathcal M)$ is
		\[
		S(X,\mathcal M)\;:=\;\bigoplus_{D\in\mathcal M} H^0\!\bigl(X,\mathcal O_X(D)\bigr),
		\]
		with multiplication induced by the natural maps
		\[
		H^0\!\bigl(X,\mathcal O_X(D)\bigr)\otimes H^0\!\bigl(X,\mathcal O_X(D')\bigr)
		\longrightarrow
		H^0\!\bigl(X,\mathcal O_X(D+D')\bigr),\qquad s\otimes s'\longmapsto ss'.
		\]
		Let $\PDiv(X)\subset \WDiv(X)$ be the subgroup of principal divisors and set
		\[
		\mathcal M^0 \;:=\; \mathcal M\cap \PDiv(X).
		\]
		For each $E=\divv(f)\in \mathcal M^0$, multiplication by $f$ induces an isomorphism
		$\mathcal O_X \xrightarrow{\sim} \mathcal O_X(E)$ and hence identifies
		$1\in H^0(X,\mathcal O_X)$ with a canonical section $1_E\in H^0\!\bigl(X,\mathcal O_X(E)\bigr)$.
		Let $I\subset S(X,\mathcal M)$ be the ideal generated by the elements $1_E-1$ for all $E\in\mathcal M^0$.
		We define the associated \emph{$\mathcal M$-graded section ring} to be the quotient
		\[
		\mathcal R(X,\mathcal M)\;:=\; S(X,\mathcal M)/I.
		\]
		
		Moreover, if $\mathcal M$ is a finitely generated subgroup of $\WDiv(X)$ mapping surjectively onto $\Cl(X)$,
		then $\mathcal R(X,\mathcal M)$ is naturally graded by $\mathcal M/\mathcal M^0 \cong \Cl(X)$.
		In this case, we call $\mathcal R(X,\mathcal M)$ the \emph{Cox ring} of $X$. 
	\end{definition}
	
	\begin{remark}\label{rem:choiceK}
		
		If $\mathcal M,\mathcal M'\subset \WDiv(X)$ are finitely generated subgroups whose images in
		$\Cl(X)$ agree, then the associated graded section rings $\mathcal R(X,\mathcal M)$ and
		$\mathcal R(X,\mathcal M')$ are (non-canonically) isomorphic as $\Cl(X)$-graded $\mathbb{K}$-algebras;
		see \cite[\S1.4.3]{ADHL15}. Hence, after fixing such a choice once and for all, we omit $\mathcal M$
		from the notation. In particular, by abuse of notation, we regard the grading monoids as submonoids of
		$\Cl(X)$. Hence for a submonoid $\mathcal M\subset \Cl(X)$ we write
		$\mathcal R(\mathcal M)$.  We call it an $\mathcal{M}$-graded section ring. Moreover, we denote the Cox ring of $X$ simply by $\operatorname{Cox}(X)$, without reference to the
		auxiliary choice of $\mathcal M$. If $\mathcal{M}$ is torsion-free, the description simplifies:
		$$\mathcal{R}(\mathcal{M})\cong \bigoplus_{D\in \mathcal{M}} H^0(X,\mathcal{O}_X(D)),$$
		as detailed in \cite[Construction~I.4.1.1]{ADHL15}.
	\end{remark}
	
	\begin{definition}\label{def:movable-linear-system}
		Let $X$ be a normal $\mathbb{Q}$-factorial projective variety.
		A $\mathbb{Q}$-divisor $D$ on $X$ is \emph{movable} if there exists $m>0$ such that $mD$ is Cartier, $|mD|\neq \emptyset$, and the linear system $|mD|$ has no fixed divisorial components. The movable cone $\operatorname{Mov}(X)\subset N^1(X)_\mathbb{R}$ is the cone generated by the classes
		of movable divisors. Similarly, we denote by $\MovE(X)$ the closure of the effective movable cone, i.e., $\MovE(X)=\overline{\operatorname{Mov}}(X)\cap\operatorname{Eff}(X)$.
	\end{definition}
	
	\begin{remark}
		Let $X$ be a normal projective $\mathbb Q$-factorial variety with finitely generated class group. This implies $\Pic^0(X)=0$. Since $X$ is $\mathbb Q$-factorial, there is a well-defined homomorphism
		\[
		\nu\colon \Cl(X)\longrightarrow N^1(X)_{\mathbb R}
		\]
		sending a class $[D]$ to $\frac{1}{m}[mD]_{\text{num}}$ for any $m>0$ with $mD$ Cartier.
		Its kernel consists of classes whose Cartier multiples are numerically trivial; equivalently,
		$\ker(\nu)=\Pic^\tau(X)$.
		In particular, if $\Pic^0(X)=0$, then $\Pic^\tau(X)$ is torsion and hence $\ker(\nu)=\Cl(X)_{\text{tor}}$. Since $\Pic^0(X)=0$, we also have $\Pic(X)\cong \operatorname{NS}(X)$, hence
		\[
		\Cl(X)_{\rr}\;\cong\;\Pic(X)_{\rr}\;\cong\;\text{NS}(X)_{\rr}\;=\;N^1(X)_{\rr}.
		\]
		
		For a convex cone $\mathcal K\subset N^1(X)_{\mathbb R}$, set
		\[
		\mathcal K\cap \Cl(X):=\{\,w\in \Cl(X)\mid \nu(w)\in \mathcal K\,\}.
		\]
		We define the associated $\mathcal K$-graded section ring by
		\[
		\mathcal R(\mathcal K):=\bigoplus_{w\in {\mathcal{K}}\cap \Cl(X)} \operatorname{Cox}(X)_w
		\ \subset\ \operatorname{Cox}(X),
		\]
		where $\operatorname{Cox}(X)_w$ denotes the $w$-graded piece of $\operatorname{Cox}(X)$. In particular, $\mathcal{R}(\Nef(X))$ will be called the \emph{$\Nef(X)$-graded section ring}.
	\end{remark}
	
	\begin{remark}
		Cox rings need not be finitely generated $\mathbb{K}$-algebras. Nonetheless, Lemma~\ref{domain} shows that one can still establish ring-theoretic properties such as being UFD or normal. Moreover, as in Theorem~\ref{F-puretype}, the notion of $F$-purity is also well-defined.
	\end{remark}

	\begin{lemma}\emph{(\cite[Proposition~I.4.1.5, Theorem~I.5.1.1]{ADHL15})}\label{domain} Let $X$ be a smooth projective variety with a free, finitely generated class group. Then $\operatorname{Cox}(X)$ is a normal UFD.
	\end{lemma}

	We recall the definition and basic properties of Mori dream spaces.

	\begin{definition}\label{definition-MDS}
		Let $X$ be a projective normal $\mathbb{Q}$-factorial variety with finitely generated class group. Then $X$ is called a \emph{Mori dream space} if $\operatorname{Cox}(X)$ is finitely generated. 
	\end{definition}
	\begin{proposition}\label{MDS,property}\emph{(\cite[Proposition~2.9]{HK00})}
		Let $X$ be a normal projective $\mathbb{Q}$-factorial variety with $\Pic^0(X)=0$.
		Then the following are equivalent:
		\begin{enumerate}
			\item $X$ is a Mori dream space.
			\item The following conditions hold:
			\begin{enumerate}
				\item The class group $\Cl(X)$ is finitely generated.
				\item The nef cone $\Nef(X)$ is a rational polyhedral cone, and every nef divisor on $X$ is semiample.
				\item There exist finitely many small $\mathbb{Q}$-factorial modifications $f_i\colon X \dashrightarrow X_i$ such that
				\[\operatorname{Mov}(X)=\bigcup_i f_i^*\Nef(X_i).
				\]
			\end{enumerate}
		\end{enumerate}
	\end{proposition}
	
	A key feature of Mori dream spaces is that they can be constructed as GIT quotients of their Cox rings by a quasi-torus~\cite{HK00}.

	\begin{definition}\label{def:semi-stability}
		Let $X$ be a smooth projective variety and set $\bar X:=\Spec\Cox(X)$. Let
		\[
		H:=\Spec\, \mathbb{K}[\Cl(X)]
		\]
		be an algebraic group called a \emph{quasi-torus}; it acts on $\bar X$ via the $\Cl(X)$-grading. Using the canonical identification
		\[\Hom(H,\mathbb{G}_m) \;\cong\; \Cl(X),
		\]
		we may regard any class $D\in \Cl(X)$ as a character of $H$, i.e., $\chi_D\in X^*(H)$.
		
		Fix an integral class $D\in \text{Mov}(X)$. Define the associated $\mathbb{N}$-graded subring
		\[
		\mathcal{R}(D):=\bigoplus_{n\ge 0} H^0(X,\mathcal{O}_X(nD))\subset \Cox(X).
		\]
		The \emph{semi-stable locus with respect to $D$} is the open subset
		\[
		\overline X^{ss}(\chi_D)
		:=\left\{x\in \bar X \ \middle|\ \text{there exist }\, n>0,\ \, f\in {\mathcal{R}(D)}_n\ \text{such that}\ f(x)\neq 0\right\}.
		\]
	\end{definition}
	
	Over a field of characteristic $0$, when $\Cl(X)$ is a finitely generated abelian group, a quasi-torus is isomorphic to the direct product of an algebraic torus $\mathbb{G}_m^n$ and a finite torsion group $T$ (\cite[Construction~I.2.1.3]{ADHL15}).
	
	\begin{theorem}\label{Cox-MDS} \emph{(\cite[Theorem~2.3, Corollary~2.4, Proposition~2.9]{HK00})}
		Let \(X\) be a Mori dream space. Then:
		
		\begin{enumerate}
			\item  
			For any ample class $D$, we have
			\[
			X \;\cong\; \overline{X}^{ss}(\chi_D)\sslash H.
			\] Moreover, given any SQM $\varphi:X\dashrightarrow X'$, there exists an effective movable divisor $D'\in\Mov(X)$ such that $$X'\cong \overline{X}^{ss}(\chi_{D'})\sslash H.$$
			
			\item
			There exists a finite rational polyhedral fan  in \(\Cl(X)_\mathbb R\) whose support equals the movable cone, such that:
			if $D$ and $D'$ lie in the relative interior of the same cone, then
			\[
			\overline X^{ss}(\chi_D)=\overline X^{ss}(\chi_{D'})
			\quad\text{and hence}\quad
			X_D \cong X_{D'}.
			\]
			
			\item
			If \(D\) and \(D'\) lie in the relative interior of adjacent chambers of $\Mov(X)$, then there exists a small birational map
			\[
			X_D \dashrightarrow X_{D'} \qquad\text{factoring through a common contraction}\qquad
			X_D \to Z\leftarrow X_{D'}.
			\]
			Moreover, when \(D\) ranges in the relative interior of the movable cone, the induced birational maps are isomorphisms in codimension one.
			
		\end{enumerate}
	\end{theorem}

	\subsection{Wehler type hypersurfaces}
	Fix integers \(m\ge 2\) and \(n\ge 2\),
	or \(m= 0\) and \(n\ge 4\). Set
	\(\mathbb{P}:=(\mathbb{P}^1)^m\times \mathbb{P}^n\).
	We write a point of \(\mathbb{P}\) in homogeneous coordinates as
	\[
	\bigl([x_{10}:x_{11}],\dots,[x_{m0}:x_{m1}],[y_0:\dots:y_n]\bigr).
	\]
	Let \(X\subset \mathbb{P}\) be a closed variety. For \(1\le k\le m\), define the open subsets
	\[
	U_{k0,X}:=\{\,p\in X \mid x_{k0}(p)\neq 0\,\},
	\qquad
	U_{k1,X}:=\{\,p\in X \mid x_{k1}(p)\neq 0\,\}.
	\]
	When \(X\) is fixed, we unambiguously write \(U_{k0}\) and \(U_{k1}\) for these open sets.
	
	Let
	\[
	\pi_i:\mathbb{P}\longrightarrow
	\begin{cases}
		\mathbb{P}^1 & \text{for } 1\le i\le m,\\
		\mathbb{P}^n & \text{for } i=m+1
	\end{cases}
	\]
	denote the natural projections. Let $H_i$ be a divisor \(\mathcal{O}_{\mathbb{P}}(H_i):=\pi_i^{*}\mathcal{O}(1)\).
	By abuse of notation, we also write \(H_i\) for its restriction to a closed subvariety \(X\subset \mathbb{P}\) when no confusion can arise.
	Each \(H_i\) is semiample, since it is the pullback of the ample line bundle \(\mathcal{O}(1)\) on a projective space. 
	\begin{definition}
		Let $n\geq2$ and $m\geq2$ or $n\geq 4$ and $m=0$. A \emph{Wehler type hypersurface} $X \subset (\mathbb{P}^1)^m \times \mathbb{P}^n$ is a very general smooth hypersurface of multidegree
		\[
		(2,\dots,2,n+1).
		\]
		For each $k\in\{1,\dots,m\}$, the defining equation of $X$ can be written as a quadratic polynomial in the coordinates of the $k$-th $\mathbb{P}^1$ factor:
		\[
		F \;=\; R_{(k,0)}\,x_{k0}^2 \;+\; R_{(k,1)}\,x_{k0}x_{k1} \;+\; R_{(k,2)}\,x_{k1}^2,
		\]
		where $R_{(k,j)}$ are multi-homogeneous polynomials in $\mathbb{K}[\,x_{ki},\,y_0,\dots,y_{n} \mid 1\leq k\leq m , i\in\{0,1\}\,]$ of multidegree $(2,\dots,0_k,\dots,2,n+1)$, i.e., degree $0$ at $k$-th degree.
	\end{definition}

	By Bertini’s theorem, \(X\) is smooth and irreducible. Moreover, the adjunction formula gives
	\[
	K_X \sim (K_{\mathbb{P}}+X)|_X \sim 0.
	\]
	One can compute that \(H^i(X,\mathcal{O}_X)=0\) for \(1\le i<\dim X\). Hence \(X\) is a Calabi--Yau variety.
	
	The next two lemmas are standard for Calabi--Yau varieties; we record them here in the present setting.

	\begin{lemma}\emph{(\cite[Lemma~3.3]{Yáñ22})}
		Let \(X\subset \mathbb{P}\) be a Wehler type hypersurface. Then the restriction map
		\[
		H^0\!\bigl(\mathbb{P},\mathcal{O}_{\mathbb{P}}(H_i)\bigr)\longrightarrow
		H^0\!\bigl(X,\mathcal{O}_{X}(H_i)\bigr)
		\]
		is an isomorphism.
	\end{lemma}
	
	\begin{lemma}\label{Neron-Severi}
		Let \(X\) be a Wehler type hypersurface. Then the class group $\Cl(X)$ is generated by the divisor classes \(H_i\).
	\end{lemma}

	\begin{proof}
		By the Grothendieck-Lefschetz Hyperplane Theorem, the restriction map $\Pic(\pp)\to\Pic(X)$ is an isomorphism.
		Since $\Pic(\pp)$ is generated by the hyperplane classes $H_i$, then $\Cl(X)\cong\Pic(X)$ is generated by the restrictions of the pullback hyperplane classes $H_i$.
	\end{proof}
	
	\begin{theorem}\emph{(\cite[Lemma~4.5 and Proposition~4.7]{Yáñ22})}\label{Kawamata-Morrison}
		Let $X$ be a Wehler type hypersurface. Then, any integral effective divisor is the pull back of a nef divisor via a birational automorphism. Furthermore, the Morrison--Kawamata conjecture holds for Wehler type hypersurfaces where the fundamental chamber $C$ is the rational polyhedral cone $\operatorname{Nef}(X)$. 
	\end{theorem}

	\section{GIT for Morrison--Kawamata dream spaces}
	
	Let $X$ be a normal projective $\mathbb{Q}$-factorial variety. We denote by $\PsAut(X)$ the group of birational self-maps of $X$ that are isomorphisms in codimension one. Any $g\in \PsAut(X)$ induces a linear action on numerical divisor classes
	\[
	g^*\colon N^1(X)_{\mathbb{R}} \longrightarrow N^1(X)_{\mathbb{R}}.
	\]
	If $C\subset N^1(X)_{\mathbb{R}}$ is a cone, we write $g^*C := \{\, g^*\alpha\mid \alpha\in C\,\}$.
	
	\medskip
	We now state the definition of Morrison--Kawamata dream space established in \cite{CRX25}.
	\begin{definition}
		\label{def:KMMDspace}
		(\emph{\cite[Theorem~1.3]{CRX25}})
		Let $X$ be a normal projective variety. Assume that any effective $\mathbb{R}$-Cartier divisor admits a minimal model. Then $X$ is a \emph{Morrison--Kawamata dream space} if:
		\begin{enumerate}
			\item $X$ is $\mathbb{Q}$-factorial;
			\item there exists a rational polyhedral cone $\Pi \subset \MovE(X)$ such that $\PsAut(X) \cdot \Pi =
			\MovE(X)$;
			\item there is a finite collection of small $\mathbb{Q}$-factorial modifications $f_i : X \dashrightarrow X_i,1 \leq i \leq l$
			such that $\Pi \subset \cup_{i=1}^l  f_i^*(\Nef(X_i))$ and each $\Pi\cap f_i^*(\Nef(X_i))$ is a rational polyhedral cone,
			\item $f_{i*}D$ is semi-ample for each effective $\mathbb R$-Cartier divisor $D$ with $[D]\in\Pi \cap
			f_i^*(\Nef (X_i ))$.
			
		\end{enumerate}
	\end{definition}
	
	Condition (1) is satisfied by every Mori dream space \cite[Proposition~1.11~(1)]{HK00}. Condition (2) is an analogue to Proposition~\ref{MDS,property}. Combined with $\Pic^0(X)=0$, it follows that there are only countably many chambers for the movable cone. As a consequence of \cite{CRX25}, $X$ always admits a good minimal model. In particular, for any effective $\mathbb{Q}$-divisor $D$, the ring 
	\[\mathcal{R}(D):=\bigoplus_{m\ge 0} H^0\!\bigl(X,\mathcal O_X(\lfloor mD\rfloor )\bigr)
	\]
	is a finitely generated $\mathbb K$-algebra. 
	Note that there exists a Veronese subalgebra of $\mathcal R(D)$ which is also a subalgebra of the Cox ring. Passing to a Veronese subalgebra does not change the structure of the associated projective variety (up to a regrading). Hence, without loss of generality, we may assume for the following statement that $D$ is a $\mathbb Z$-divisor. In Theorem \ref{thm:GIT-Proj}, we prove a GIT construction for Morrison--Kawamata dream spaces.
	
	\begin{theorem}\label{thm:GIT-Proj}
		Let $X$ be a Morrison--Kawamata dream space with finitely generated class group. Let $\overline{X}:=\Spec\Cox(X).$
		Let $H:=\Spec \mathbb{K}[\Cl(X)]$ act on $\bar X$ via the grading.
		Let $D$ be an effective integral divisor. 
		Then:
		\begin{enumerate}
			\item $\overline{X}^{ss}(\chi_D)$ is an $H$-invariant open subscheme of $\overline{X}$;
			\item there exists a categorical quotient 
			\[
			\pi_D:\ \overline X^{ss}(\chi_D)\longrightarrow \overline{X}^{ss}(\chi_D)\sslash H;
			\]
			\item there is a canonical isomorphism of projective varieties
			\[
			\overline{X}^{ss}(\chi_D)\sslash H \ \cong\ \operatorname{Proj}\, \mathcal{R}(D).
			\]
			In particular, if $D\in \Amp(X)$, then $X$ is recovered as a GIT quotient. 
		\end{enumerate}
	\end{theorem}
	
	\begin{proof}
		By definition,
		\[
		\overline X^{ss}(\chi_D)=\bigcup_{n>0}\ \bigcup_{f\in \mathcal{R}(D)_n} D(f),
		\]
		hence it is an open subset of $\overline{X}$. Since $f$ is homogeneous of degree $nD$, $D(f)$ is $H$-stable. Let
		$\mathcal{R}(D)_+$ denote the positive graded part of $\mathcal{R}\,(D)$. Note that $\mathcal{R}\,(D)$ is finitely generated, thus $\operatorname{Proj}\, \mathcal{R}\,(D)$ is a projective $\mathbb K$-variety. Choose homogeneous elements
		$f_1,\dots,f_m\in \mathcal{R}(D)_+$ whose images generate $\mathcal{R}(D)_+$ up to radical. Then the standard opens
		\[
		D_+(f_i)\subset \text{Proj}\, \mathcal{R}(D)
		\]
		form a finite affine open cover. Now fix a homogeneous element $f\in H^0(X,\mathcal{O}_X(nD))\subset \mathcal{R}(D)$ with $n>0$. Then $D(f)=\Spec(\Cox(X)_f)$ is an affine $H$-stable open subset of $\overline{X}$.
		Thus there exists a uniform categorical quotient \cite[Ch1\,\S2\, Theorem 1.1]{MFK94}
		\[
		q_f:\ D(f)=\Spec\,(\Cox(X)_f)\longrightarrow \Spec\,\bigl((\Cox(X)_f)^H\bigr).
		\]
		
		\medskip
		
		The $H$-action on $\Cox(X)$ is induced by the $\Cl(X)$-grading of Definition~\ref{def:cox}. For any $\Cl(X)$-graded ring, the invariants under the corresponding quasi-torus action are precisely the degree $0$ part. Moreover, localization at a homogeneous element preserves the induced $\Cl(X)$-grading. Hence
		\[
		(\Cox(X)_f)^H=(\Cox(X)_f)_0.
		\]
		The standard affine chart of $\text{Proj}\, \mathcal{R}(D)$ defined by $f$ is
		\[
		D_+(f)=\Spec\,\bigl((\mathcal{R}(D))_{f,0}\bigr),
		\] 
		Note that $\mathcal{R}(D)$ has a standard $\mathbb{Z}$-grading. We claim there is a canonical ring isomorphism
		\begin{equation}\label{eq:deg0-identification}
			(\mathcal{R}(D)_f)_{0}\ \xrightarrow{\ \sim\ }\ (\Cox(X)_f)^H.
		\end{equation}
		We have a natural inclusion
		$(\mathcal{R}(D)_f)_0\hookrightarrow (\Cox(X)_f)^H$. Conversely, any homogeneous element of $(\Cox(X)_f)_0$ is represented by a fraction $g/f^m$ with $g\in \Cox(X)$ homogeneous, satisfying
		$$\deg_{\Cl(X)}(g)=m\deg_{\Cl(X)}(f)=mnD.$$ 
		Hence, $g\in H^0(X,\mathcal{O}_X(mnD))=\mathcal{R}(D)_{mn}\subset \mathcal{R}(D)$ and consequently, since $g$ and $f$ both lie in $\mathcal{R}(D)$, we have $g/f^m\in (\mathcal{R}(D)_f)_0$. Therefore, the inclusion is bijective, proving \eqref{eq:deg0-identification}. We obtain canonical identifications of affine schemes
		\[
		\Spec\bigl((\Cox(X)_f)^H\bigr)
		=
		\Spec\bigl((\Cox(X)_f)_0\bigr)
		\cong
		\Spec\bigl(\mathcal{R}(D)_f\bigr)
		=
		D_+(f)\subset \text{Proj}\, \mathcal{R}(D).
		\]
		These local morphisms $q_f:D(f)\to \Spec((\Cox(X)_f)^H)$ glue together to a morphism
		\[
		\pi_D:\ \overline X^{ss}(\chi_D)\longrightarrow \text{Proj}\, \mathcal{R}(D).
		\]
		By construction, $\pi_D|_{D(f)}=q_f$ for every $f\in \mathcal{R}(D)_+$, hence $\pi_D$ is a categorical quotient. Therefore, we obtain the following isomorphism
		\[
		\overline{X}^{ss}(\chi_D)\sslash H \ \cong\ \text{Proj}\, \mathcal{R}(D).
		\]\end{proof}
	\begin{lemma}\label{lem:transport-sections}
		Let $\phi\colon X\dashrightarrow Y$ be an SQM between normal projective varieties,
		and let $D$ be an effective $\mathbb Q$-Weil divisor on $Y$.
		Let $D_X$ be the strict transform of $D$ on $X$. Then, for every $m\ge 0$, there is a natural identification
		\[
		H^0\!\,\bigl(Y,\mathcal O_Y(\lfloor mD\rfloor)\bigr)\ \cong\ H^0\!\,\bigl(X,\mathcal O_X(\lfloor mD_X\rfloor)\bigr).
		\]
		In particular, the $D$-graded section rings are naturally isomorphic:
		\[
		\mathcal{R}\,(Y,D)\  \cong\
		\mathcal R(X,D_X).
		\]
		
		Moreover, if $g\in\operatorname{PsAut}(X)$, then $\mathcal R(X,rD)\cong \mathcal R(X,g^*rD)$.
	\end{lemma}

	Now we prove analogues of conditions (1), (2) and (3) of Theorem~\ref{Cox-MDS}. Assume that $X$ is a Morrison--Kawamata dream space. By hypothesis, the movable cone
	$\MovE(X)\subset N^1(X)_{\mathbb R}$ is covered by the fundamental domain $\Pi$ and its pullbacks under $\operatorname{PsAut}(X)$.
	For a SQM $\phi\colon X \dashrightarrow Y$, we identify $N^1(Y)_{\mathbb R}$ with $N^1(X)_{\mathbb R}$
	via $\phi^*$, and regard $\Nef(Y)$ as a cone in $N^1(X)_{\mathbb R}$.
	We say that two SQM models $Y_+$ and $Y_-$ are \emph{adjacent} if
	$\Nef(Y_+)\cap \Nef(Y_-)$ is a codimension-one face of both cones; in this case we write
	\[
	F \;:=\; \Nef(Y_+)\cap \Nef(Y_-)\subset N^1(X)_\mathbb R
	\]
	for their common facet.

	\begin{proposition}\label{prop:KM-wallcrossing}
		Let $X$ be a Morrison--Kawamata dream space.
		Assume that $Y_+$ and $Y_-$ are adjacent SQM models of $X$ with common facet $F$ and let $D\in \relint(F)$ be an effective $\mathbb Q$-Cartier class.
		Then:
		\begin{enumerate}
			\item Let $Y$ be any SQM model of $X$, let $\sigma\preceq \Nef(Y)$ be a face, and let $D,D'\in \relint(\sigma)$ be semiample $\mathbb Q$-Cartier classes.
			Then the associated projective models coincide: $X_D\cong X_{D'}$.
			\item For any sufficiently divisible $m>0$ such that $mD$ is Cartier on both $Y_+$ and $Y_-$, the divisor $mD$ defines projective morphisms
			\[
			f_\pm: Y_\pm \to Z_\pm,\qquad
			Z_\pm \cong \operatorname{Proj} \bigoplus_{n\ge 0} H^0\!\left(Y_\pm,\mathcal O_{Y_\pm}(nmD)\right),
			\]
			and the targets $Z_+$ and $Z_-$ are canonically isomorphic.
			\item The induced birational map $\psi:Y_+\dashrightarrow Y_-$ factors through $Z$:
			\[
			Y_+\xrightarrow{f_+} Z \xleftarrow{f_-} Y_- .
			\]
		\end{enumerate}
	\end{proposition}
	
	\begin{proof}
		For (1), let $Y$ be an SQM model of $X$ and let $\sigma\preceq \Nef(Y)$ be a face.
		Let $D$ and $D'$ be semiample $\mathbb Q$-Cartier classes belonging to $\relint(\sigma)$,
		and choose $m>0$ sufficiently divisible such that $mD$, $mD'$ are Cartier and base point free.
		Since $D$ and $D'$ lie in $\relint(\sigma)$, they vanish on the same curve classes,
		namely those in $\sigma^\perp\cap \overline{\operatorname{NE}}(Y)$.
		Because $D$ and $D'$ are semiample, the morphisms defined by $|mD|$ and $|mD'|$
		contract precisely these curves, hence they have the same Stein factorization.
		In particular, there exists a projective morphism $f\colon Y\to Z$ and ample
		$\mathbb Q$-divisors $A,A'$ on $Z$ such that
		\[
		D\sim_{\mathbb Q} f^*A,\qquad D'\sim_{\mathbb Q} f^*A'.
		\]
		After replacing $m$ by a common multiple, we may assume $mA$ and $mA'$ are Cartier,
		and then for all $n\ge 0$,
		\[
		H^0\!\left(Y,\mathcal O_Y(nmD)\right)\cong H^0\!\left(Z,\mathcal O_Z(nmA)\right),
		\qquad
		H^0\!\left(Y,\mathcal O_Y(nmD')\right)\cong H^0\!\left(Z,\mathcal O_Z(nmA')\right).
		\]
		Thus,
		$$\operatorname{Proj}\,\mathcal R(Y,D)\cong\operatorname{Proj}\,\mathcal R(Y,D')\cong Z.$$
		By Lemma~\ref{lem:transport-sections}, the SQM $X\dashrightarrow Y$ induces
		$\mathcal R(X,mD)\cong\mathcal R(Y,mD)$.
		Since $mD$ is an integral divisor, Theorem~\ref{thm:GIT-Proj} gives
		$X_{mD}\cong\operatorname{Proj}\mathcal R(X,mD)$,
		and $X_D=X_{mD}$ because $\overline X^{ss}(\chi_D)$ depends only on the
		ray $\mathbb R_{>0}D$.
		The same holds for $D'$, and consequently
		\[
		X_D
		\;\cong\;
		\operatorname{Proj}\mathcal R(X,mD)
		\;\cong\;
		\operatorname{Proj}\mathcal R(X,mD')
		\;\cong\;
		X_{D'}.
		\]

		For (2), let $\psi\colon Y_+\dashrightarrow Y_-$ be the birational map induced by the SQM identifications with $X$.
		Choose open subsets $U_\pm\subset Y_\pm$ with $\codim(Y_\pm\setminus U_\pm)\ge 2$ such that
		$\psi|_{U_+}\colon U_+\xrightarrow{\sim}U_-$. Fix $n\ge 0$ and choose $m>0$ sufficiently divisible such that $mD$ is Cartier on both $Y_\pm$.
		Then the sheaves $\mathcal{O}_{Y_\pm}(nmD)$ are line bundles and restriction over $U_{\pm}$ induces isomorphisms
		\[
		H^0\!\bigl(Y_\pm,\mathcal{O}_{Y_\pm}(nmD)\bigr)\ \xrightarrow{\ \sim\ }\ H^0\!\bigl(U_\pm,\mathcal{O}_{U_\pm}(nmD)\bigr),
		\]
		because $Y_\pm$ are normal and the complements have codimension at least $2$.
		Using $U_+\cong U_-$, we obtain canonical identifications
		\[
		H^0\!\,\bigl(Y_+,\mathcal{O}_{Y_+}(nmD)\bigr)\ \cong\ H^0\!\,\bigl(Y_-,\mathcal{O}_{Y_-}(nmD)\bigr)
		\quad\text{for all }n\ge 0,
		\]
		hence an isomorphism of graded rings
		\[
		\bigoplus_{n\ge 0} H^0\!\,\bigl(Y_+,\mathcal{O}_{Y_+}(nmD)\bigr)
		\ \cong\
		\bigoplus_{n\ge 0} H^0\!\,\bigl(Y_-,\mathcal{O}_{Y_-}(nmD)\bigr).
		\]
		Denote the common $\text{Proj}$ by
		\[
		Z:=\text{Proj}\!\,\Big(\bigoplus_{n\ge 0} H^0\!\,\bigl(Y_+,\mathcal{O}_{Y_+}(nmD)\bigr)\Big)
		\ \cong\
		\text{Proj}\!\,\Big(\bigoplus_{n\ge 0} H^0\!\,\bigl(Y_-,\mathcal{O}_{Y_-}(nmD)\bigr)\Big).
		\]
		Since $D$ is effective, by Lemma \ref{lem:transport-sections}
		\[
		\bigoplus_{n\ge0}H^0\bigl(Y_\pm,\mathcal O_{Y_\pm}(nmD)\bigr)\cong \mathcal{R}(X,mD).
		\]
		Since $\mathcal{R}(X,mD)$ is finitely generated (as $X$ admits a good minimal model for $D$), it follows that $Z$ is a projective variety.
		
		Since $Y_\pm$ are SQMs of $X$, they are again Morrison--Kawamata dream spaces by \cite[Lemma 3.16 (2)]{CRX25}.
		Moreover, by \cite[Theorem 1.3]{CRX25}, every effective $\mathbb R$-Cartier divisor on a Morrison--Kawamata dream space admits a good minimal model.
		Since $D\in\relint(F)\subset\Nef(Y_\pm)$, the class $D$ is nef on each $Y_\pm$ and it is effective by hypothesis. Hence $Y_\pm$ is itself a good minimal model of $D$, so $D$ is semiample on $Y_\pm$.
		Consequently for $m$ sufficiently divisible, $|mD|$ is base point free on $Y_\pm$ and $mD$ defines the projective morphisms $f_\pm\colon Y_\pm\to Z_\pm\cong Z$.

		\medskip
		For (3), the maps $f_+|_{U_+}$ and $f_-\circ(\psi|_{U_+})$ are induced by the same graded ring of sections under the identifications above, hence coincide. Consequently, $\psi$ factors over $Z$:
		\[
		Y_+\xrightarrow{\,f_+\,} Z \xleftarrow{\,f_-\,} Y_-.
		\]
		Moreover, $\psi$ is an SQM, and therefore an isomorphism in codimension one; it follows that $\psi$ remains an isomorphism in codimension one over $Z$.
	\end{proof}
	\begin{corollary}\label{intro:KMMDS3}
		Let $X$ be a Morrison--Kawamata dream space with finitely generated class group, let $Y_+,Y_-$ be adjacent SQM models with common facet $F$  and let $D,D'$ lie in the relative interiors of $\Nef(Y_+)$ and $\Nef(Y_-)$ respectively. Then $X_D\cong Y_+$ and $X_{D'}\cong Y_-$ are big models, and there is a small birational map
		\[
		X_D\dashrightarrow X_{D'}\qquad\text{factoring through a common contraction}\qquad X_D\xrightarrow{f_+} Z\xleftarrow{f_-} X_{D'}.
		\]
		Moreover, for any $D_1,D_2$ in the interior of chambers of $\MovE(X)$, the induced birational map $X_{D_1}\dashrightarrow X_{D_2}$ between the corresponding models is an isomorphism in codimension one.
	\end{corollary}
	
	\begin{proof}
		Since $D$ lies in $\relint\Nef(Y_+)$, it is ample on $Y_+$. Choose $m>0$ such that $mD$ is integral; then $X_D=X_{mD}$, since $\overline X^{ss}(\chi_D)$ depends only on the ray $\mathbb R_{>0}D$. By Lemma~\ref{lem:transport-sections}, $\mathcal R(X,mD)\cong\mathcal R(Y_+,mD)$, so by Theorem~\ref{thm:GIT-Proj},
		\[
		X_D\cong\operatorname{Proj}\mathcal R(X,mD)\cong\operatorname{Proj}\mathcal R(Y_+,mD)\cong Y_+,
		\]
		the last isomorphism because $D$ is ample on $Y_+$.
		Likewise, $X_{D'}\cong Y_-$. Since $F$ is a common facet of two chambers of the decomposition of $\MovE(X)$, we have $\relint(F)\subseteq\operatorname{int}\Mov(X)\subseteq\operatorname{Big}(X)$. Fix $D_0\in\relint(F)$ and $m>0$ sufficiently divisible such that $mD_0$ is a Cartier divisor on both $Y_+$ and $Y_-$.
		
		By Proposition \ref{prop:KM-wallcrossing} (2), we obtain projective morphisms $f_\pm\colon Y_\pm\to Z$ and part~(3) yields the factorization $Y_+\xrightarrow{f_+}Z\xleftarrow{f_-}Y_-$ with $\psi: Y_+\dashrightarrow Y_-$ an isomorphism in codimension one. Bigness of $D_0$ gives $\dim Z=\dim Y_\pm$, so $f_\pm$ are birational. since $D_0\in\relint(F)\subseteq\operatorname{int}\Mov(X)$ is movable, its ample model contracts
		no divisor (a divisorial contraction gives $D_0\in\partial\Mov(X)$); hence $f_\pm$
		are small and $\psi=f_-^{-1}\circ f_+$ is small. Under $X_D\cong Y_+$ and $X_{D'}\cong Y_-$ this is the asserted small factorization.
		
		For the last claim, let $\Gamma$ be the graph whose vertices are the chambers of the fan covering $\MovE(X)$, with two vertices joined by an edge whenever the corresponding chambers share a wall contained in $\operatorname{int}\Mov(X)$. 
		The union of the cones of codimension at least $2$ does not disconnect $\operatorname{int}\MovE(X)$. Since adjacent chambers meet along their codimension-one walls, which lie in this complement, the complement is connected. Therefore any two chambers can be joined by a finite sequence of chambers $E_0,E_1,\dots,E_N$, such that $\overline{E_{k-1}}\cap\overline{E_k}$ is a wall contained in $\operatorname{int}\Mov(X)$ for every $k$. Equivalently, the graph $\Gamma$ is connected.
		
		By part~(1) of Proposition~\ref{prop:KM-wallcrossing}, $X_D$ is constant on $\relint(E_k)$, equal to the SQM model with nef chamber $E_k$. By part~(3), the induced map between the models of $E_{k-1}$ and $E_k$ is an isomorphism in codimension one. Composing the $N$ maps shows $X_{D_1}\dashrightarrow X_{D_2}$ is an isomorphism in codimension one.
	\end{proof}
	
	\medskip
	Finally, we show that for a Morrison--Kawamata dream space $X$, the Cox ring
	$\Cox(X)$ is a filtered colimit of $\Cl(X)$-graded subalgebras, each
	approximated by finitely generated $\mathbb K$-algebras $\mathcal{R}_{s,\epsilon}$.

	\medskip
	Note that the pseudo-automorphisms act linearly on $\operatorname{NS}(X)$.
	Hence there is a representation
	\[
	\rho:\PsAut(X)\to GL(\operatorname{NS}(X),\mathbb{Z}).
	\]
	Since $GL(\operatorname{NS}(X),\mathbb{Z})$ is countable as a group, the image
	$\rho(\PsAut(X))$ is countable.
	Thus, we can enumerate it as
	$\rho(\PsAut(X)):=\{g_1,\dots, g_s,\dots \}$.
	We set
	\[
	\mathcal U_0 \;:=\; \Pi, \qquad
	\mathcal U_s \;:=\; \Pi + g_1^*\Pi + \cdots + g_s^*\Pi \;\subset\; N^1(X)_{\mathbb R}
	\quad (s\ge 1),
	\]
	where the sum denotes the convex cone generated by the union.
	Each $\mathcal U_s$ is a rational polyhedral cone, and $\mathcal U_s\subseteq
	\mathcal U_{s'}$ whenever $s\le s'$.
	
	\medskip
	Since $\Eff(X)$ has non-empty interior and is covered by the translates of
	$\Pi$, after replacing $\Pi$ by some $g^*\Pi$ if necessary we may assume that
	$\Pi$ contains an ample class.
	Fix an ample effective integral class
	\[
	A \;\in\; \Amp(X)\cap \mathcal U_0.
	\]
	
	\medskip
	By Gordan's lemma \cite[Proposition~1.1.2.1]{ADHL15}, the semigroup
	$\mathcal U_s\cap \Cl(X)$ is finitely generated.
	Denote its Hilbert basis by
	\[
	S_{1,s},\,\dots,\,S_{r_s,s} \;\in\; \mathcal U_s\cap\Cl(X),
	\]
	where each $S_{i,s}$ is the class of an effective integral Weil divisor.
	By picking $\epsilon \in \mathbb{Q}$ with $0<\epsilon\ll 1$, we set
	\[
	D_{i,s,\epsilon} \;:=\; S_{i,s} + \epsilon A,
	\qquad
	\mathcal K_{s,\epsilon} \;:=\; \operatorname{Cone}(D_{1,s,\epsilon},\dots,D_{r_s,s,\epsilon})
	\;\subset\; N^1(X)_{\mathbb R},
	\]
	and $\mathcal K_s:=\mathcal K_{s,0}$.
	Since $A$ is ample and $S_{i,s}$ is effective, each $D_{i,s,\epsilon}$ is a big effective $\mathbb Q$-Cartier
	class.
	
	\medskip
	Let $\Cl(X)_{\mathrm{tors}}=\{T_1,\dots,T_t\}$ denote the torsion subgroup of $\Cl(X)$.
	The multiplication on the
	$\mathbb K$-vector space
	\begin{equation}\label{eq:Rseps}
		\mathcal R_{s,\epsilon}
		\;:=\;
		\bigoplus_{j=1}^{t}
		\bigoplus_{(d_1,\dots,d_{r_s})\in\mathbb Z_{\ge 0}^{r_s}}
		H^0\!\Bigl(X,\,
		\mathcal O_X\!\Bigl(
		\Bigl\lfloor \sum_{k=1}^{r_s} d_k D_{k,s,\epsilon} \Bigr\rfloor + T_j
		\Bigr)\Bigr)
	\end{equation}
	is defined by multiplication of rational functions: if
	\[
	\sigma \in H^0\!\Bigl(X,\mathcal O_X\!\Bigl(\Bigl\lfloor \sum_{k=1}^{r_s} d_k D_{k,s,\epsilon} \Bigr\rfloor + T_i\Bigr)\Bigr)
	\quad\text{and}\quad
	\tau \in H^0\!\Bigl(X,\mathcal O_X\!\Bigl(\Bigl\lfloor \sum_{k=1}^{r_s} d'_k D_{k,s,\epsilon} \Bigr\rfloor + T_j\Bigr)\Bigr),
	\]
	then $\sigma\tau$ is regarded as a section of
	$\mathcal O_X\!\bigl(\bigl\lfloor \sum_{k=1}^{r_s} (d_k+d'_k) D_{k,s,\epsilon} \bigr\rfloor + T_i + T_j\bigr)$
	using the natural inclusion from
	\[
	\Bigl\lfloor \sum_{k=1}^{r_s} d_k D_{k,s,\epsilon} \Bigr\rfloor
	+\Bigl\lfloor \sum_{k=1}^{r_s} d'_k D_{k,s,\epsilon} \Bigr\rfloor
	\;\le\;
	\Bigl\lfloor \sum_{k=1}^{r_s} (d_k+d'_k) D_{k,s,\epsilon} \Bigr\rfloor.
	\]
	This makes $\mathcal R_{s,\epsilon}$ a $\Cl(X)$-graded $\mathbb K$-algebra.
	Equivalently, $\mathcal R_{s,\epsilon}$ is the $\Cl(X)$-graded section ring
	associated to the finitely generated submonoid of $\WDiv(X)$ generated by
	effective representatives of
	$$\{D_{1,s,\epsilon},\dots,D_{r_s,s,\epsilon},T_1,\dots,T_t\},$$
	in the sense of Definition~\ref{def:cox}.
	Since each $D_{i,s,\epsilon}$ is big, the algebra $\mathcal R_{s,\epsilon}$ is finitely generated by
	\cite[Corollary~1.1.9]{BCHM10}.
	We write $\mathcal R_s := \mathcal R_{s,0}$
	for the $\mathcal K_s$-graded section ring.
	
	\begin{theorem}\label{thm:direct-limit-mov}
		Let $X$ be a Morrison--Kawamata dream space with finitely generated class
		group.
		For $s\ge 0$, let $\mathcal R_s$ denote the $\mathcal K_s$-graded section
		ring and for $\epsilon>0$ let $\mathcal R_{s,\epsilon}$ denote the
		$\mathcal K_{s,\epsilon}$-graded section ring.
		Then\textup{:}
		\begin{enumerate}
			\item The canonical map
			\[
			\varinjlim_s \mathcal R_s \;\xrightarrow{\ \sim\ }\; \Cox(X)
			\]
			is an isomorphism of $\Cl(X)$-graded $\mathbb K$-algebras.
			
			\item For each $s$ and $0<\epsilon'\le\epsilon$, there is a natural inclusion of
			finitely generated $\mathbb K$-algebras
			\[
			\mathcal R_{s,\epsilon'}\;\hookrightarrow\;\mathcal R_{s,\epsilon},
			\]
			preserving the grading of \eqref{eq:Rseps}; these make
			$\{\mathcal R_{s,\epsilon}\}_{\epsilon>0}$ a directed system.
			
			\item For each $s$, the canonical map
			\[
			\mathcal R_s \;\xrightarrow{\ \sim\ }\; \varprojlim_{\epsilon\to 0^+}\mathcal R_{s,\epsilon}
			\]
			is an isomorphism of $\mathbb K$-algebras graded as in \eqref{eq:Rseps}.
		\end{enumerate}
	\end{theorem}
	
	\begin{proof}
		\emph{(1)} For $s \le s'$ we have $\mathcal U_s \subseteq \mathcal U_{s'}$, so each 
		Hilbert basis element of $\mathcal U_s \cap \Cl(X)$ is a non-negative integer combination 
		of those of $\mathcal U_{s'} \cap \Cl(X)$, inducing the inclusion $\mathcal R_s \hookrightarrow \mathcal R_{s'}$.
		
		By the Morrison--Kawamata cone conjecture \cite[Theorem~1.6]{CRX25}, 
		$\Eff(X) = \bigcup_s \mathcal U_s$.
		For any nonzero $f \in \Cox(X)$ with degree 
		$[D] = [D_0] + T_j$ (where $[D_0] \in \Eff(X) \cap \Cl(X)$ and 
		$T_j \in \Cl(X)_{\mathrm{tors}}$), Gordan's lemma gives $[D_0] = \sum_k d_k S_{k,s}$ 
		for some $s$ with $[D_0] \in \mathcal U_s \cap \Cl(X)$.
		Thus $f \in \mathcal R_s$, 
		so $\Cox(X) = \bigcup_s \mathcal R_s$.
		Since the transition maps are inclusions, 
		the canonical map $\varinjlim_s \mathcal R_s \to \Cox(X)$ is an isomorphism of 
		$\Cl(X)$-graded $\mathbb K$-algebras.
		
		\smallskip
		\emph{(2)} Fix $s$ and $0<\epsilon'\le\epsilon$.
		Since $A$ is effective,
		$\epsilon' A\le\epsilon A$ as $\mathbb Q$-divisors, so for every
		$(d_1,\dots,d_{r_s})\in\mathbb Z_{\ge0}^{r_s}$,
		\[
		\Bigl\lfloor \sum_{k=1}^{r_s} d_k D_{k,s,\epsilon'} \Bigr\rfloor
		\;\le\;
		\Bigl\lfloor \sum_{k=1}^{r_s} d_k D_{k,s,\epsilon} \Bigr\rfloor.
		\]
		Viewing sections inside the function field $\mathbb K(X)$, this inequality of
		divisors gives in each multidegree,
		\[
		H^0\!\Bigl(X,\mathcal O_X\!\Bigl(\Bigl\lfloor \sum_k d_k D_{k,s,\epsilon'} \Bigr\rfloor + T_j\Bigr)\Bigr)
		\;\subseteq\;
		H^0\!\Bigl(X,\mathcal O_X\!\Bigl(\Bigl\lfloor \sum_k d_k D_{k,s,\epsilon} \Bigr\rfloor + T_j\Bigr)\Bigr),
		\]
		and these inclusions are compatible with the multiplication above.
		Hence there exists a natural inclusion $\mathcal R_{s,\epsilon'}\hookrightarrow\mathcal R_{s,\epsilon}$
		of finitely generated $\mathbb K$-algebras preserving the grading of
		\eqref{eq:Rseps}, and $\{\mathcal R_{s,\epsilon}\}_{\epsilon>0}$ is a directed
		system.
		
		\smallskip
		\emph{(3)} Fix $s$, and consider $A = \sum_\alpha a_\alpha \Gamma_\alpha$ in prime components. Set $a := \max_\alpha a_\alpha$.
		For each multi-index $d = (d_1,\dots,d_{r_s})$, let $E_d := \sum_{k=1}^{r_s} d_k S_{k,s}$ and $|d| := \sum_{k=1}^{r_s} d_k$.
		
		Since $E_d$ is an integral divisor,
		\[
		\Bigl\lfloor \sum_{k=1}^{r_s} d_k D_{k,s,\epsilon} \Bigr\rfloor
		= \bigl\lfloor E_d + \epsilon|d|A \bigr\rfloor
		= E_d + \sum_\alpha \bigl\lfloor \epsilon|d|a_\alpha \bigr\rfloor \Gamma_\alpha.
		\]
		For $\epsilon < (a|d|)^{-1}$, all terms $\lfloor \epsilon|d|a_\alpha \rfloor$ vanish, so  $(\mathcal R_{s,\epsilon})_{(d,j)} = (\mathcal R_s)_{(d,j)}$.
		Since $E_d \le \lfloor \sum_k d_k D_{k,s,\epsilon}\rfloor$ for all $\epsilon \ge 0$, we have 
		$(\mathcal R_s)_{(d,j)} \subseteq (\mathcal R_{s,\epsilon})_{(d,j)}$ for every $\epsilon > 0$.
		Thus
		\[
		(\mathcal R_s)_{(d,j)} \;=\; \bigcap_{\epsilon>0}(\mathcal R_{s,\epsilon})_{(d,j)}.
		\]
		The inverse limit of a system of graded rings commutes with direct sums over fixed multidegrees, so the componentwise equality gives the isomorphism
		\[
		\mathcal R_s \;\xrightarrow{\,\sim\,}\; \varprojlim_{\epsilon\to0^+}\mathcal R_{s,\epsilon}
		\]
		of $\mathbb K$-algebras graded as in \eqref{eq:Rseps}.
	\end{proof}
	
	\begin{remark}
		Combining Theorem~\ref{thm:direct-limit-mov}(1) and~(2) presents $\Cox(X)$ as
		the iterated filtered colimit
		\[
		\Cox(X) \;\cong\;
		\varinjlim_{s}\,\Big(\varprojlim_{\epsilon\to 0^+}\,\mathcal R_{s,\epsilon}\Big),
		\]
		each $\mathcal R_{s,\epsilon}$ being a finitely generated $\Cl(X)$-graded $\mathbb K$-algebra.
		This is the substitute, for Morrison--Kawamata dream spaces, of the finite generation of $\Cox(X)$ that characterizes Mori dream spaces.
	\end{remark}
	
	In the subsequent sections, as an application of this theory, we compute the Cox ring of a very general Calabi--Yau hypersurface of multidegree $(2,\dots,2,n+1)$ in $(\mathbb{P}^1)^m\times \mathbb{P}^n.$
	
	\section{Dynamics of Wehler type hypersurfaces}
	We study the birational automorphisms of Wehler type hypersurfaces and their action on the $\mathbb{K}$-vector space $H^0(X,\mathcal{O}(H_i))$. These automorphisms have been studied in detail in \cite{CO15} and \cite{Yáñ22}. The following theorem provides a precise description of $\Bir(X)$.
	
	\begin{theorem}\label{thm:Wehler-bir}
		Let $X$ be a Wehler type hypersurface. Then
		\begin{enumerate}
			\item its automorphism group is trivial.
			\item its birational automorphism group is isomorphic to the free product of $m$ copies of $\mathbb{Z}/2\mathbb{Z}$:
			\[
			\Bir(X)\ \cong\ 
			\underbrace{\mathbb{Z}/2\mathbb{Z} * \cdots * \mathbb{Z}/2\mathbb{Z}}_{m-\text{factors}}.
			\]
		\end{enumerate}
	\end{theorem}
	\begin{proof}
		By \cite[Theorem~1.3]{Yáñ22}, (2) is a consequence of (1). We therefore prove that the automorphism group is trivial. 
		Let us denote 
		\[V:=H^0(\mathbb{P},\mathcal{O}_{\mathbb{P}}(2,\dots,2,n+1)),\quad
		A:=\operatorname{Sym}^2(\kk^2), \quad W:=\operatorname{Sym}^{n+1}(\kk^{n+1}) \quad\text{and} \quad
		V \cong A^{\otimes m}\otimes W.
		\]
		The automorphism group of $X$ is isomorphic to a subgroup of $G:=((\operatorname{PGL}_{2})^m\rtimes S_m)\times \operatorname{PGL}_{n+1}$, by \cite[Theorem~1.3]{Yáñ22}. Write $G^0$ for the identity component of $G$. Since the finite group $S_m$ does not affect dimension estimates, it is enough
		to prove the statement for $G^0\cong (\operatorname{PGL}_2)^m\times \operatorname{PGL}_{n+1}$. Consider the closed incidence subset
		\[
		\mathcal I
		:=\Big\{(g,[s])\in (G^0\setminus\{1\})\times \pp(V)\ \Big|\ g^*s=\lambda s
		\ \text{for some }\lambda\in \mathbb{G}_m\Big\}.
		\]
		Let $\pi:\mathcal I\to \pp(V)$ be the projection. Then $\pi(\mathcal I)$ is precisely the locus
		of $[s]$ having a nontrivial stabilizer in $G^0$. Hence it suffices to show
		\[
		\dim \mathcal I \ <\ \dim \pp(V)=\dim V-1.
		\]
		Indeed, then $\pi(\mathcal I)$ is a proper closed subset of $\pp(V)$ and its complement is a
		nonempty Zariski open set $U$ with $\operatorname{Stab}_{G^0}([s])=\{1\}$ for all $[s]\in U$. Fix $g\in G^0\setminus\{1\}$. The fiber of the projection
		$\mathcal I\to (G^0\setminus\{1\})$ over $g$ is the union of the projective eigenspaces
		\[
		\bigcup_{\lambda}\pp(V_{\lambda}(g))\subset \pp(V),
		\quad \text{where}\quad 
		V_{\lambda}(g):=\{v\in V\mid g\cdot v=\lambda v\}.
		\]
		Therefore,
		\[
		\dim \mathcal I
		\ \le\
		\dim G^0\ +\ \max_{\substack{g\ne 1\\ \lambda}} \dim \pp(V_{\lambda}(g))
		\ =\
		\dim G^0\ +\ \max_{\substack{g\ne 1\\ \lambda}}(\dim V_{\lambda}(g)-1).
		\]
		Consequently, it suffices to prove the uniform codimension bound
		\begin{equation}\label{eq:OC-codim-bound}
			\min_{\lambda}\codim_V(V_{\lambda}(g))
			=\min_{\lambda}\bigl(\dim V-\dim V_{\lambda}(g)\bigr)
			\ >\ \dim G^0,
			\quad\text{for all }g\in G^0\setminus\{1\}.
		\end{equation}
		
		\medskip
		We have
		\[
		\dim V
		= \dim(\operatorname{Sym}^2\kk^2)^m\cdot \dim(\operatorname{Sym}^{n+1}\kk^{n+1})
		=3^{m}\binom{2n+1}{n},
		\]
		and
		\[
		\dim G^0
		= \dim\big((\operatorname{PGL}_2)^m\big)+\dim(\operatorname{PGL}_{n+1})
		=3m+\big((n+1)^2-1\big).
		\]
		
		\medskip
		We bound $\dim V_{\lambda}(g)$ by cases. Write $g=(g_1,\dots,g_m,h)$ with $g_i\in \operatorname{PGL}_2$ and $h\in \operatorname{PGL}_{n+1}$.
		
		\medskip
		\noindent\emph{Case 1: $g$ acts nontrivially on some $\operatorname{PGL}_2$-factor.}
		
		After reindexing assume $g_1\neq 1$, and write $V\cong A\otimes R$, where
		$R:=A^{\otimes(m-1)}\otimes W$, so that $g$ acts as $S\otimes T$, where
		$S:=\operatorname{Sym}^2(G_1)$ and $T$ is the induced action on $R$. Since $g_1\neq 1$, the lift $G_1$ is non-scalar.
		We first show that every eigenspace of $S$ has dimension at most $2$.
		If $G_1$ is diagonalizable with eigenvalues $\alpha,\beta$, then $S$ has eigenvalues
		$\alpha^2$, $\alpha\beta$, and $\beta^2$.
		These are pairwise distinct unless $\alpha=-\beta$, in which case they become
		$\alpha^2$, $-\alpha^2$, and $\alpha^2$, so the largest eigenspace has dimension $2$.
		The case $\alpha=\beta$ is excluded because $G_1$ is non-scalar.
		If $G_1$ is not diagonalizable, then $\operatorname{Sym}^2(G_1)$ consists of a single Jordan block of size $3$, hence every eigenspace is one-dimensional.
		Therefore, $\max_{\mu}\dim\ker(S-\mu\,\mathrm{id}_A)\le 2$.
		Fix an eigenvalue $\lambda$ of $g$.
		Choose a basis of $R$ in which $T$ is upper triangular, with diagonal entries
		$\nu_1,\dots,\nu_{\dim R}$.
		With respect to the induced decomposition
		$A\otimes R\cong\bigoplus_{i=1}^{\dim R}A$,
		the operator $S\otimes T-\lambda\,\mathrm{id}_{A\otimes R}$ is block upper triangular with diagonal blocks
		$\nu_iS-\lambda\,\mathrm{id}_A$.
		Hence,
		$\dim V_\lambda(g)
		=
		\dim\ker(S\otimes T-\lambda I)
		\le
		\sum_{i=1}^{\dim R}\dim\ker(\nu_iS-\lambda I)$.
		Since each $\nu_i\neq0$, we have
		$\ker(\nu_iS-\lambda I)=\ker(S-(\lambda/\nu_i)I)$,
		so every summand is bounded by $2$.
		Therefore,
		$\dim V_\lambda(g)\le 2\dim R=\frac{2}{3}\dim V$.
		
		Consequently,
		\[
		\min_{\lambda}
		\operatorname{codim}_V\!\big(V_\lambda(g)\big)
		\ge
		\frac{1}{3}\dim V
		=
		3^{m-1}\binom{2n+1}{n}.
		\]

		\medskip
		\noindent\emph{Case 2: $g$ is trivial on all $\operatorname{PGL}_2$--factors and nontrivial on $\operatorname{PGL}_{n+1}$.} This case relies on the following lemma.
		
		\begin{lemma}\label{lem:sym-eigenspace-bound}
			Assume that \(\kk\) is algebraically closed of characteristic 0. Let \(N=n+1\), and let \(h\in \operatorname{PGL}_{N}\) be nontrivial. Then, for every eigenvalue \(\lambda\) of the induced action of \(h\) on \(W:=\operatorname{Sym}^{N}(\kk^N)\), one has \(\operatorname{codim}_{W} W_\lambda \ge \binom{2n-1}{n-1}\).
		\end{lemma}
		
		\begin{proof}
			Choose a lift \(H = H_s H_u \in \operatorname{GL}_{N}(\kk)\) with its Jordan decomposition. Since \(H_s\) and \(H_u\) commute, the \(H\)-eigenspace \(W_\lambda(H)\) is contained in the \(H_s\)-weight space \(W_\lambda(H_s)\). We consider two cases based on whether \(H_s\) is scalar.
			
			\emph{Case 1: \(H_s\) is non-scalar.} Choose an eigenbasis \(x_0, \dots, x_n\) for \(\kk^N\). We may assume \(x_0\) and \(x_1\) have distinct \(H_s\)-weights. Let \(U := \langle x_0, x_2, \dots, x_n \rangle\). The \(\binom{2n-1}{n-1}\) monomials of degree \(n\) in \(U\) form a basis for \(\operatorname{Sym}^n U\). For each such monomial \(M\in \operatorname{Sym}^n(U)\), the degree \(n+1\) monomials \(x_0M\) and \(x_1M\) have distinct \(H_s\)-weights. Thus, at least one avoids the \(\lambda\)-weight space. Since the pairs \(\{x_0M, x_1M\}\) are pairwise disjoint as \(M\) varies, the codimension of \(W_\lambda(H_s)\) (and therefore of \(W_\lambda(H)\)) is at least \(\binom{2n-1}{n-1}\).
			
			\noindent\emph{Case 2: \(H_s\) is scalar.} Since \(h\) is nontrivial, its unipotent part \(H_u\) must also be nontrivial. After rescaling \(H\), we may assume \(H = H_u = \exp(D)\) for some nonzero nilpotent operator \(D\). The only eigenvalue of \(H\) is 1, and the space of fixed points is exactly \(W^H = \ker(D|_W)\). 
			
			To bound the codimension of \(W^H\) from below, we must maximize the dimension of \(\dim \ker(D|_W)\), over all nonzero nilpotent operators \(D\). To make the kernel as large as possible, \(D\) must have the minimal possible rank without being entirely zero. In terms of Jordan normal form, this minimum rank is achieved when \(D\) has exactly one \(2 \times 2\) block and all other blocks are \(1 \times 1\) (Jordan type \((2, 1, \dots, 1)\)).
			
			We therefore restrict our attention to this minimal case. Let us identify \(W = \operatorname{Sym}^{n+1}(\kk^{n+1})\) with the space of homogeneous polynomials of degree \(n+1\) in variables \(e_0, \dots, e_n\). We may choose a basis for \(\kk^{n+1}\) such that \(D(e_1) = e_0\) and \(D(e_i) = 0\) for all \(i \neq 1\). When we apply \(D\) to polynomials, it follows the standard product rule. Because it effectively replaces occurrences of \(e_1\) with \(e_0\), this action is equivalent to taking the partial derivative with respect to \(e_1\) and multiplying the result by \(e_0\). Consequently, \(\ker(D|_W)\) consists of all polynomials in \(W\) that vanish under this partial derivative namely, those that do not depend on \(e_1\). This yields:
			\[
			\ker(D|_W) = \operatorname{Sym}^{n+1}\langle e_0, e_2, \dots, e_n \rangle.
			\]
			The codimension of \(W^H\) in \(W\) is then simply the difference in dimensions:
			\[
			\dim \operatorname{Sym}^{n+1}(\kk^{n+1}) - \dim \operatorname{Sym}^{n+1}(\kk^n) = \binom{2n+1}{n} - \binom{2n}{n-1} = \binom{2n}{n}.
			\]
			Since \(\binom{2n}{n} \geq \binom{2n-1}{n-1}\), the desired codimension bound holds in the unipotent case as well.
		\end{proof}
		By Lemma~\ref{lem:sym-eigenspace-bound}, in Case~2 we obtain
		\[
		\min_\lambda\codim_V\!\big(V_\lambda(g)\big)
		=3^m\min_\lambda\codim_W\!\big(W_\lambda(h)\big)
		\ \ge\ 3^m\binom{2n-1}{n-1}.
		\]
		
		\medskip
		Combining the two cases, for every $g\neq1$,
		\[
		\min_\lambda \codim_V\!\big(V_\lambda(g)\big)
		\ \ge\
		\min\Big\{\,3^{\,m-1}\binom{2n+1}{n},\ \ 3^{m}\binom{2n-1}{n-1}\,\Big\}.
		\]
		Thus \eqref{eq:OC-codim-bound} holds provided
		\begin{equation}\label{eq:final-ineq}
			3^{\,m-1}\binom{2n+1}{n}\ >\ 3m+(n+1)^2-1
			\qquad\text{and}\qquad
			3^{m}\binom{2n-1}{n-1}\ >\ 3m+(n+1)^2-1,
		\end{equation}
		both of which hold whenever $m\ge2$ and $m+n\ge4$. Under these hypotheses
		$\dim\mathcal I<\dim V-1$, so $\pi(\mathcal I)$ is a proper closed subset of $\pp(V)$ and there is a
		nonempty Zariski open $U\subset\pp(V)$ with $\operatorname{Stab}_{G^0}([s])=\{1\}$ for all $[s]\in U$. 
		Since $\Aut(X)\cong\operatorname{Stab}_G([s])$ by \cite[Theorem~1.3]{Yáñ22}, the hypersurface $X=V(s)$ has trivial
		automorphism group for every $[s]\in U$, proving (1).
	\end{proof}

	The generators of $\Bir(X)$ are described explicitly.
	For each $k\in\{1,\dots,m\}$, let
	\[
	\pi^\prime_k : X \to (\mathbb{P}^1)^{m-1}\times \mathbb{P}^n
	\]
	be the morphism induced by the projection forgetting the $k$-th $\mathbb{P}^1$-coordinate.
	Since $X$ has degree $2$ in the $k$-th $\mathbb{P}^1$-coordinate, the map $\pi^\prime_k$ is generically a double cover onto its image.
	Exchanging the two sheets defines a covering involution
	\[
	\iota_k:X\dashrightarrow X,
	\]
	which fixes the ramification divisor of $\pi^\prime_k$ (where the quadratic equation in the $k$-th coordinate acquires a double root).
	These involutions generate $\Bir(X)$ and give a concrete description of the induced action on $N^1(X)_{\mathbb{R}}$ and on the movable cone.
	
	\begin{theorem}\emph{(\cite[Proposition~3.6]{Yáñ22})}\label{involution-action}
		For each $k\in\{1,\dots,m\}$, the induced action of the covering involution
		$\iota_k\in\Bir(X)$ on $\operatorname{NS}(X)$ is given by
		\[
		\iota_k^*(H_j)=H_j \quad (j\neq k),
		\qquad\text{and}\qquad
		\operatorname{deg}(\iota_k^*(H_k))=(2, \dots,-1_{k},\dots, 2, n+1).\footnote{In a $n$-tuple $(a_1,\dots,a_n)\in\zz^n$ such that $a_k=i$, we write $i_k$ for the $k$-th entry.}\]
		In particular, there is a natural linear isomorphism of global sections
		\[
		H^0\!\bigl(X,\mathcal{O}_X(H_k)\bigr)
		\ \xrightarrow{\ \sim\ }\
		H^0\!\bigl(X,\mathcal{O}_X(\iota_k^*H_k)\bigr),
		\]
		induced by pullback along $\iota_k$.
	\end{theorem}
	
	\begin{proposition}\label{multiplication-involution} Let $X$ be a Wehler type hypersurface and let $\iota_k$ be the $k$-th covering involution of $X$. Then, inside the Cox ring, we have the relation
		\[
		\iota_k^*(x_{ki}) x_{ki} - R_{(k,2-2i)} = 0.
		\]
	\end{proposition}
	\begin{proof}
		
		Let $X$ be defined in $\mathbb{P}$ by the polynomial
		$$F:=R_{(k,0)}x _{k0}^2+R_{(k,1)}x_{k0}x_{k1}+R_{(k,2)}x_{k1}^2 \quad \text{where }R_{(k,i)}\in S_{(2,\dots,2,0_k,2,\dots,2,n+1)}.$$
		
		The birational automorphism $\iota_k$ is the involution associated to the projection forgetting the $k$-th $\mathbb{P}^1$-coordinate; it is generically a double cover onto its image. Explicitly, it is defined by the following actions. Suppose $x_{k1}\neq0$, and we set $r_k:=\frac{x_{k0}}{x_{k1}}$. Then, the involution $\iota_k^*$ is an action that exchanges the roots of the equation of $$R_{(k,0)}r_k^2+R_{(k,1)}r_k+R_{(k,2)}=0.$$
		
		In other words, let $r'_k=\frac{\iota_k^*(x_{k0})}{\iota_k^*(x_{k1})}$, then:
		\begin{equation}\label{Vieta-relation} r_k+r'_k=-\frac{R_{(k,1)}}{R_{(k,0)}},\quad r_kr'_k=\frac{R_{(k,2)}}{R_{(k,0)}}\quad \text{with } \iota_k^*(r_k)=r'_k,\quad \iota_k^*(r'_k)=r_k.\end{equation}
		Note that $\iota_k^*$ fixes all other variables. Here by Proposition~\ref{involution-action}, the degree of $\iota_k^*({x_{ki}})$ becomes $$(2,\dots,2,-1_k,2,\dots,2,n+1),$$ where $k$-th entry is $-1$. Hence on the open set $U_{k0}\cap U_{k1}$, the involution acts as $$\iota_{k}^*\colon [x_{k0}:x_{k1}]=\left[\frac{x_{k0}}{x_{k1}}:1\right]\mapsto \left[\frac{\iota_k^*(x_{k0})}{\iota_k^*(x_{k1})}:1\right]=\left[\frac{R_{(k,2)}x_{k1}}{R_{(k,0)}x_{k0}}:1\right]=\left[\frac{R_{(k,2)}}{x_{k0}}:\frac{R_{(k,0)}}{x_{k1}}\right]=\left[\iota_k^*(x_{k0}):\iota_k^*(x_{k1})\right]. $$ Thus, $H^0(X,\mathcal{O}(H_k))=\operatorname{span}\{x_0,x_1
		\}$ and consequently, $H^0(X,\mathcal{O}(\iota_k^*H_k))=\operatorname{span}\{\iota_k^*(x_{k0}),\iota_k^*(x_{k1})\}$. 
		By the previous degree computation, on the intersection 
		$U_{k0}\cap U_{k1}$ we have the identity
		\begin{equation}\label{eq:rel}
			\iota_k^*(x_{ki})\, x_{ki} - R_{(k,2-2i)} = 0.
		\end{equation} 
		The open set $U_{k0}\cap U_{k1}$ is nonempty and lies inside the smooth locus where $\iota_k$ is regular. Since both sides of \eqref{eq:rel} are homogeneous elements of $\operatorname{Cox}(X)$, an equality of homogeneous rational functions that holds on a dense open subset must hold in $\Cox(X)$ by Lemma~\ref{domain}.  
	\end{proof}

	\section{The {Nef(X)}-graded section ring}
	For a Wehler type hypersurface, the class group $\Cl(X)$ is torsion-free, so we do not need to keep track of torsion
	classes. Accordingly, we set
	\[
	\mathcal{R}_0 \;:=\; \bigoplus_{[D]\in \Nef(X)\cap \Cl(X)} H^0\!\bigl(X,\mathcal O_X(D)\bigr)
	\]
	as the $\Nef(X)$-graded section ring. 
	We determine the generators of this ring and identify all its relations. Specifically, we first show that every global section in ample multi-degrees is the restriction of a polynomial on the ambient space (Propositions~\ref{surjectivity}--\ref{0part}). 
	This breaks down precisely when one of the degrees is zero. In that case, there exists an additional section that does not arise from restriction. We package this section into a generator $\zeta_k$ (Definition~\ref{def:zeta}) and determine the structure of the section ring in degrees with exactly one vanishing coordinate (Proposition~\ref{zeta-relation}).
	With these ingredients we give two presentations of $\mathcal{R}_0$. One of them corresponds to the $s=1$ graded part of Theorem \ref{thm:cox-limit} in the introduction which involves ideal saturation. The second replaces saturation with an explicit list of relations, see Theorem~\ref{thm:Nef-no-saturation}.
	
	\medskip
	The following proposition is a consequence of Lemma~\ref{Neron-Severi} and the semiampleness of $H_k$. 
	
	\begin{proposition}
		The $\Nef(X)$-graded section ring $\mathcal{R}_0$  of a Wehler type hypersurface $X$ is a finitely generated $\mathbb{K}$-algebra.
	\end{proposition}
	
	In the next proposition, we compute the $\Amp(X)$-graded section ring, which we denote by $\mathcal{R}_0'$. We denote by \(S_{(a_1,\dots,a_m,b)}\) the space of multihomogeneous polynomials
	of multidegree \((a_1,\dots,a_m,b)\) 
	\[
	S_{(a_1,\dots,a_m,b)}
	\;:=\;
	H^0\bigl(\mathbb{P},\mathcal{O}_{\mathbb{P}}(a_1,\dots,a_m,b)\bigr).
	\]
	\begin{proposition}\label{surjectivity} Let $X$ be a Wehler type hypersurface defined by an equation $F$ in $\mathbb{P}$. Let multi-indices be $(a_1,\dots,a_m,b),(c_1,\dots,c_m,d)\in\mathbb{Z}^{m+1}_{>0}$.
		\begin{enumerate}
			\item  The multiplication map $$H^0(X,\mathcal{O}(a_1,\dots ,a_m, b))\otimes H^0(X,\mathcal{O}(c_1,\dots, c_m, d))\rightarrow H^0(X,\mathcal{O}(a_1+c_1,\dots ,a_m+c_m,b+d ) )$$ is surjective.
			\item 
			The $\Amp(X)$-graded section ring $\mathcal{R}_0'$ is isomorphic to
			\[
			\mathcal{R}'_{0}
			:= \mathbb{K}\oplus
			\bigoplus_{\substack{a_1,\dots,a_m,b>0}}
			H^0\!\left(X,\mathcal O_X(a_1,\dots,a_m,b)\right)\cong\mathbb{K}\oplus\bigoplus_{a_1,\dots,a_m,b>0} S_{(a_1,\dots ,a_m,b)}/\langle F \rangle.
			\]    
		\end{enumerate}
	\end{proposition}
	\begin{proof}
		Consider the ideal sheaf exact sequence
		\begin{equation*}
			\begin{tikzcd}[column sep=small, font=\small]
				0 &
				{S_{(a_1-2,\dots,a_m-2 ,b-n-1)}} &
				{S_{(a_1,\dots ,b)}} &
				{H^0(X,\mathcal{O}_X(a_1,\dots, a_m ,b))} &
				{\textstyle H^1\big(\mathbb{P},\mathcal{O}_{\mathbb{P}}(a_1-2,\dots,a_m-2 ,b-n-1)\big).}
				\arrow[from=1-1, to=1-2]
				\arrow["{\cdot F}", from=1-2, to=1-3]
				\arrow[from=1-3, to=1-4]
				\arrow[from=1-4, to=1-5]
			\end{tikzcd}
		\end{equation*}
		
		The map $S_{(a_1,\dots, a_m,b)}\to H^0(X,\mathcal{O}_X(a_1,\dots, a_m ,b))$ is surjective exactly when $$H^1(\mathbb{P},\mathcal{O}(a_1-2,\dots,a_m-2 ,b-n-1))=0.$$ By Serre duality we have $H^1(\mathbb{P}^1,\mathcal{O}(a))\cong H^0(\mathbb{P}^1,\mathcal{O}(-a-2))^{\vee}$. In particular $H^1(\mathbb{P}^1,\mathcal{O}(a))=0$ for $a> -2$. Using the Künneth formula, we have 
		\begin{multline*}
			H^1\big(\mathbb{P}, \mathcal{O}(a_1-2, \dots, a_m-2, b-n-1)\big) \cong \\
			\bigoplus_{\sum i_k = 1} 
			H^{i_1}(\mathbb{P}^{1}, \mathcal{O}(a_1-2)) \otimes \dots \otimes 
			H^{i_m}(\mathbb{P}^{1}, \mathcal{O}(a_m-2)) \otimes 
			H^{i_{m+1}}(\mathbb{P}^{n}, \mathcal{O}(b-n-1)).
		\end{multline*}
		This first cohomology group is nonzero if and only if for a fixed $k$ we have $a_k=0$, $a_{i}\geq 2$ for $i\neq k$, and $b\geq n+1$. Indeed, we observe that when $a_k=0$
		\begin{equation}\label{eq:H^1}
			\begin{aligned}
				H^1(\mathbb{P},\mathcal{O}(a_1-2,\dots, -2_k ,\dots,a_m-2,b-n-1))&\cong
				H^0(\mathbb{P}^1,\mathcal{O}(a_1-2))\otimes\cdots
				\otimes H^1(\pp^1,\mathcal{O}(-2)) \otimes \cdots\\
				&\quad \otimes H^0(\mathbb{P}^1,\mathcal{O}(a_m-2))\otimes H^0(\mathbb{P}^n,\mathcal{O}(b-n-1)).
			\end{aligned}
		\end{equation}  
		This isomorphism shows precisely the vanishing condition of this first cohomology group. In particular, this implies that when $a_1,\dots, a_m,b>0$ every element of $H^0(X,\mathcal{O}(a_1,\dots, a_m,b))$ is given by the restriction of regular sections of the ambient space $\mathbb{P}$. 
		By the commutativity of the following diagram:

		\[\begin{tikzcd}
			{H^0(\mathbb{P},\mathcal{O}(a_1,\dots ,b))\otimes H^0(\mathbb{P},\mathcal{O}(c_1,\dots ,d))} & {H^0(\mathbb{P},\mathcal{O}(a_1+c_1,\dots ,b+d))} \\
			{H^0(X,\mathcal{O}(a_1,\dots ,b))\otimes H^0(X,\mathcal{O}(c_1,\dots ,d))} & {H^0(X,\mathcal{O}(a_1+c_1,\dots ,b+d)),}
			\arrow[two heads, from=1-1, to=1-2]
			\arrow["{\text{res}\otimes\text{res}}", from=1-1, to=2-1]
			\arrow[from=1-2, to=2-2,"{\text{res}}"]
			\arrow[from=2-1, to=2-2]
		\end{tikzcd}\]
		we obtain the desired surjection in (1).
		Furthermore, we also observe that 
		$$ \bigoplus_{(a_1,\dots, a_m,b)\in \mathbb{Z}^{m+1}_{>0}} S_{(a_1,\dots, a_m,b)}
		\to \bigoplus_{(a_1,\dots,a_m,b)\in\mathbb{Z}_{>0}^{m+1}} H^0(X,\mathcal{O}_X(a_1,\dots, a_m ,b))$$ is surjective. As $X$ is a hypersurface, the kernel of this map on each homogeneous component is generated by $F$. This shows part $(2)$.
	\end{proof}
	We observe that if all $a_i$ and $b$ are positive, then every global section is obtained by 
	restricting sections of the ambient space $\mathbb{P}$. 
	However, if some $a_i=0$, the next proposition shows that not every section in  $H^0\!\left(X,\mathcal{O}(a_1,\dots,a_m,b)\right)$ arises as the restriction of a section from $\mathbb{P}$. 
	\begin{proposition}\label{0part}
		Let $X$ be a Wehler type hypersurface. The following statements hold:
		\begin{enumerate}
			\item $h^0(X,\mathcal{O}_X(2,\dots,2,0_k,2,\dots,2,n+1))=1+\dim_{\mathbb{K}}S_{(2,\dots,2,0_k,2,\dots,2,n+1)}$.
			\item $H^0(X,\mathcal{O}_X(2,\dots,2,0_k,2,\dots,2,n+1))/ S_{(2,\dots,2,0_k,2,\dots,2,n+1)}$ is generated by $x_{k0}\iota_k^*(x_{k1})$. 
		\end{enumerate}
		
	\end{proposition}

	\begin{proof}
		For (1), by applying the exact sequence of Proposition~\ref{surjectivity}, we obtain that
		\[
		H^0\bigl(X,\mathcal{O}(2,\dots,2,0_k,2,\dots,2,n+1)\bigr)
		\cong S_{(2,\dots,2,0_k,2,\dots,2,n+1)} \oplus 
		H^1\bigl(\mathbb{P},\mathcal{O}_{\mathbb{P}}(0,\dots,0,-2_k,0,\dots,0)\bigr).
		\]
		By the Künneth formula,
		\[
		H^1\bigl(\mathbb{P},\mathcal{O}_{\mathbb{P}}(0,\dots,0,-2_k,0,\dots,0,0)\bigr)
		\cong H^1\bigl(\mathbb{P}^1,\mathcal{O}_{\mathbb{P}^1}(-2)\bigr)
		\cong H^0\bigl(\mathbb{P}^1,\mathcal{O}_{\mathbb{P}^1}\bigr).
		\]
		
		For (2), by the grading computation of Proposition~\ref{involution-action}, the section
		$x_{k0}\iota_k^*(x_{k1})$ lies in
		$$H^0\!\left(X,\mathcal O_X(2,\ldots,2,0_k,2,\ldots,2,n+1)\right).$$
		We claim that $x_{k0}\iota_k^*(x_{k1})$ does not lie in the image of
		$S_{(2,\ldots,2,0_k,2,\ldots,2,n+1)}$, i.e.,\ it is not the restriction of a global
		section on the ambient space $\mathbb P$. Suppose for contradiction that $x_{k0}\iota_k^*(x_{k1})$ is the restriction of some
		$h\in H^0\!\left(\mathbb P,\mathcal O_{\mathbb{P}}(2,\ldots,2,0_k,2,\ldots,2,n+1)\right)$. Then on $X$ we may write
		\[
		x_{k0}\iota_k^*(x_{k1}) = h|_X .
		\]
		Multiplying both sides by $x_{k1}$ and using the defining quadratic relation on $X$ in the
		$(x_{k0},x_{k1})$-coordinates, we obtain in $\Cox(X)$ 
		\[
		R_{(k,0)}x_{k0} = x_{k1}\cdot (h|_X).
		\]
		Hence $R_{(k,0)}x_{k0}$ vanishes along the divisor
		$D_{k1}:=\{x_{k1}=0\}\cap X$ i.e., $$(R_{(k,0)}x_{k0})|_{D_{k1}} \equiv 0.$$
		Since $x_{k0}$ does not vanish identically on $D_{k1}$, this forces
		$R_{(k,0)}|_{D_{k1}}\equiv 0$. However, for a very general choice of the coefficient $R_{(k,0)}$, this is a contradiction. Therefore, $x_{k0}\iota_k^*(x_{k1})$ is not the
		restriction of a regular section on $\mathbb P$, and the corresponding quotient
		space is one-dimensional, generated by the class of this section.
	\end{proof}
	\begin{definition} \label{def:zeta}
		For each $k \in \{1, \dots, m\}$, we define the section 
		\[
		\zeta_k := x_{k0} \iota_k^*(x_{k1}) \in H^0(X, \mathcal{O}_X(2, \dots, 2, 0_k,2,\dots, 2, n+1)).
		\]
	\end{definition}
	By Proposition~\ref{0part}, the class of $\zeta_k$ generates the quotient space 
	\[
	H^0(X, \mathcal{O}_X(2, \dots,2,0_k,2, \dots,2, n+1)) \big/ S_{(2, \dots, 0_k,2, \dots,2, n+1)}.
	\]
	In particular, $\zeta_k$ is a section that does not arise from the restriction of a section on the ambient space $\mathbb{P}$.
	
	\begin{proposition}\label{zeta-relation}
		The following statements hold:
		\begin{enumerate}
			\item The section $\zeta_k$ satisfies the quadratic relation
			\[
			\zeta_k^2 + R_{(k,1)} \zeta_k + R_{(k,0)} R_{(k,2)} = 0.
			\]
			
			\item As a graded $\mathbb{K}$-algebra, we have
			\[
			\mathbb{K}\oplus\bigoplus_{\substack{a_i,b>0\\ a_k=0}}
			H^0\bigl(X,\mathcal{O}_X(a_1,\dots,a_m,b)\bigr)
			\;\cong\; \mathbb{K}\oplus\bigoplus_{\substack{a_i,b>0\\ a_k=0}}
			S_{(a_1,\dots a_m,b)}[\zeta_k]\big/\bigl\langle \zeta_k^2+R_{(k,1)}\zeta_k+R_{(k,0)}R_{(k,2)}\bigr\rangle.
			\]
		\end{enumerate}
	\end{proposition}
	
	\begin{proof}
		For (1), on the open set $U_{k0}\cap U_{k1}$, we have
		\[
		\zeta_k = R_{(k,0)} \frac{x_{k0}}{x_{k1}},
		\]
		by Proposition~\ref{multiplication-involution}. Using the defining relation on $U_{k0}\cap U_{k1}$, it follows that
		
		\[
		R_{(k,0)}\Bigl(\frac{x_{k0}}{x_{k1}}\Bigr)^2 = -R_{(k,2)} - R_{(k,1)} \frac{x_{k0}}{x_{k1}},
		\]
		and consequently that
		\[
		\zeta_k^2
		= \Bigl(R_{(k,0)} \frac{x_{k0}}{x_{k1}}\Bigr)^2
		= R_{(k,0)} \Bigl(R_{(k,0)} \frac{x_{k0}^2}{x_{k1}^2}\Bigr)
		= R_{(k,0)}\Bigl(-R_{(k,2)} - R_{(k,1)}\frac{x_{k0}}{x_{k1}}\Bigr)
		= -R_{(k,0)} R_{(k,2)} - R_{(k,1)} \zeta_k.
		\]
		Thus, on $U_{k0}\cap U_{k1}$ we obtain
		\[
		\zeta_k^2 + R_{(k,1)} \zeta_k + R_{(k,0)} R_{(k,2)} = 0
		\]
		as a section of $H^0(X,\mathcal{O}(4,\dots,0_k,\dots,4,2n+2))$.
		
		Since $X$ is irreducible and $U_{k0}\cap U_{k1}$ is a dense open subset, a global section of a line bundle which vanishes on $U_{k0}\cap U_{k1}$ must vanish everywhere. As the Cox ring is a domain by Theorem~\ref{domain}, we have
		\[
		\zeta_k^2 + R_{(k,1)} \zeta_k + R_{(k,0)} R_{(k,2)} = 0
		\]
		holds globally on $X$.
		
		\medskip
		
		For (2), fix $(a_1,\ldots,a_m,b)$ with $a_i,b>0$ and $a_k=0$, and set $L=\mathcal O_X(a_1,\ldots,a_m,b)$. We have a short exact sequence of vector spaces
		\[
		0 \longrightarrow S_{(a_1,\ldots,a_m,b)}
		\longrightarrow H^0(X,L)
		\overset{\pi}{\longrightarrow}
		H^1\!\left(\mathbb P,\mathcal O(a_1-2,\ldots,a_m-2,b-n-1)\right)
		\longrightarrow 0.
		\]
		By K\"unneth and Serre duality, the target space is naturally identified with the
		vector space
		\[
		V:=S_{(a_1-2,\ldots,a_{k-1}-2,0,a_{k+1}-2,\ldots,a_m-2,b-n-1)}.
		\]
		Now, define a canonical linear map
		\[
		\overline m_{\zeta_k}\colon V\to H^0(X,L)/S_{(a_1,\ldots,a_m,b)},
		\qquad
		f \longmapsto [\zeta_k f].
		\]
		We show $\overline{m}_{\zeta_k}$ is injective. Suppose $\zeta_k f \in S_{(a_1,\ldots,a_m,b)}$ for some $f \in V$, 
		so that $\zeta_k f = h|_X$ for some $h \in S_{(a_1,\ldots,0_k,\ldots,a_m,b)}$.
		Multiplying both sides by $x_{k1}$ and applying 
		Proposition~\ref{multiplication-involution} gives 
		$x_{k1}\zeta_k = R_{(k,0)}x_{k0}$ in $\operatorname{Cox}(X)$. This gives
		\[
		R_{(k,0)}\,x_{k0}\,f\big|_X = x_{k1}\,h\big|_X.
		\]
		Hence $x_{k1}$ divides $R_{(k,0)}\,x_{k0}\,f$ in $\operatorname{Cox}(X)$.
		Since $\operatorname{Cox}(X)$ is a UFD by Theorem~\ref{domain}, 
		and $x_{k1}$ divides neither $R_{(k,0)}$ nor $x_{k0}$ for very general $F$, 
		we conclude $x_{k1} \mid f$. 
		But $f \in S_{(a_1-2,\ldots,0_k,\ldots,a_m-2,b-n-1)}$ 
		has degree zero in $x_{k0}$ and $x_{k1}$, so this forces $f = 0$.
		
		By the dimension computation above,
		\[
		\dim _{\mathbb{K}} V=\dim_{\mathbb{K}}\bigl(H^0(X,L)/S_{(a_1,\ldots,a_m,b)}\bigr),
		\]
		so $\overline m_{\zeta_k}$ is an isomorphism. Therefore, every class in the
		quotient has a unique representative of the form $\zeta_k f$ with
		$f\in V$, and hence
		\[
		H^0(X,L)=S_{(a_1,\ldots,a_m,b)} \;\oplus\; \zeta_k\cdot
		S_{(a_1-2,\ldots,a_{k-1}-2,0,a_{k+1}-2,\ldots,a_m-2,b-n-1)}.
		\]
		Rewriting this in all degrees and using the quadratic relation from (1) yields the graded algebra presentation in (2).
	\end{proof}
	
	With the previous results in place, we now determine generators for $\mathcal{R}_0$ and construct an ideal encoding the relations among them.
	In the ring
	\[
	\widetilde{\mathcal{R}}_1:=\mathbb{K}[\, x_{ki},\, \iota_k^*(x_{ki}),y_0,\dots,y_{n} \mid 1 \le k \le m,\,i\in\{0,1\}\,],
	\]
	denote as \(I_1\) the ideal generated by the hypersurface equation and its transforms under the quadratic involutions:
	\[
	I_1
	:=
	\big\langle
	F,\;
	x_{ki}\, \iota_k^{\ast}(x_{ki}) - R_{(k,2-2i)}
	\;\big|\;
	1 \le k \le m,\; i \in \{0,1\}
	\big\rangle.
	\]
	Let
	\[
	m_1\;:=\;\prod_{1\le k\le m,\ i\in\{0,1\}} x_{ki}\;\in\;\widetilde{\mathcal{R}}_1,
	\qquad
	J_1\;:=\;\langle m_1\rangle\;\subseteq\;\widetilde{\mathcal{R}}_1,
	\]
	Let $\overline{I}_1$ denote the saturation of $I_1$ with respect to $J_1$
	\[
	\overline{I}_1 := (I_1 : J_1^{\infty})=\bigcup_{s\geq 1}(I_1:J_1^s).
	\]
	
	\begin{proposition}\label{norelation}
		Let us fix distinct indices $k,l\in\{1,\dots, m\}$ and write the defining equation of the Wehler type hypersurface as $$F=\sum_{a+b=2,c+d=2}R_{(kl,b,d)}x_{k0}^{a}x_{k1}^bx_{l0}^{c}x_{l1}^d.$$ Let  $\mathbf{d}:=(d_1,\dots, d_{m+1})$ be the multi-index defined by 
		$$d_i=\begin{cases}
			4 \text{ for }i\neq k,l,m+1\\ 1\text{ for }i=k,l \\2n+2\text{ for } i=m+1.\end{cases}$$ Then:

		\begin{enumerate}
			\item Let $A := \iota_k^*(x_{ki}) \, \iota_l^*(x_{lj}) \in H^0(X, \mathcal{O}_X(\mathbf{d}))$. Then $A$ is the restriction of a homogeneous polynomial $H \in S_\mathbf{d}$ on the ambient space, i.e., $A = H|_X$.
			\item $A-H\not\in I_1$ and $A-H\in \overline{I_1}$.
			
		\end{enumerate}  
	\end{proposition}
	
	\begin{proof}
		Fix distinct indices $k,l\in\{1,\dots, m\}$. We write the defining equation of $X$ in two ways 
		$$F=R_{(k,0)}x_{k0}^2+R_{(k,1)}x_{k0}x_{k1}+R_{(k,2)}x_{k1}^2=R_{(l,0)}x_{l0}^2+R_{(l,1)}x_{l0}x_{l1}+R_{(l,2)}x_{l1}^2$$
		We also expand the coefficients $R_{(k,i)}$ with respect to the $l$-th $\mathbb{P}^1$ coordinates, and the coefficients $R_{(l,i)}$ with respect to the $k$-th $\mathbb{P}^1$ coordinates
		$$R_{(k,i)}=R_{(lk,i,0)}x_{l0}^2+R_{(lk,i,1)}x_{l0}x_{l1}+R_{(lk,i,2)}x_{l1}^2,$$ 
		$$R_{(l,i)}=R_{(kl,i,0)}x_{k0}^2+R_{(kl,i,1)}x_{k0}x_{k1}+R_{(kl,i,2)}x_{k1}^2.$$
		The symmetry of $F$ under swapping $k$ and $l$ gives $R_{(kl,b,d)} = R_{(lk,d,b)}$ for all $b,d\in\{0,1,2\}$.
		
		The statement in (1) follows, since the restriction map
		\[
		S_{\mathbf d} \longrightarrow H^0\bigl(X,\mathcal{O}_X(\mathbf{d})\bigr)
		\]
		is a surjection, by Proposition~\ref{surjectivity}. In particular, the section $A\in H^0\bigl(X,\mathcal{O}(\mathbf d)\bigr)$ is the restriction of some homogeneous polynomial $H \in S_{\mathbf d}$. Now, for (2) we derive an explicit expression for $H$. Indeed, we have $$x_{ki}x_{lj}A=x_{ki}x_{lj}H|_{X}=R_{(k,2-2i)}R_{(l,2-2j)}|_X,$$
		by Proposition \ref{multiplication-involution}. Since $A$ is the restriction of a section on the ambient space, this equality holds on the ambient space modulo the defining equation of $X$. Specifically, there exists a $Q\in \mathbb{K}[x_{ki},y_0,\dots,y_n\mid 1 \le k \le m,\ i\in\{0,1\}\,]$ such that the following equation holds:
		\begin{equation}\label{Q-computation}
			x_{ki}x_{lj}H-R_{(k,2-2i)}R_{(l,2-2j)}=-Q F.   
		\end{equation}  
		Note that satisfying equation (\ref{Q-computation}) is equivalent to finding a $Q$ such that the expression $R_{(k,2-2i)}R_{(l,2-2j)}-QF$ is divisible by $x_{ki}x_{lj}$.
		Hence, we obtain: $$ F|_{\{x_{ki}=0\}}=R_{(k,2-2i)}x_{k(1-i)}^2\quad\text{and} \quad R_{(l,2-2j)}\big|_{\{x_{ki}=0\}} = R_{(kl,2-2j,2-2i)}\,x_{k(1-i)}^2.$$
		Substituting into \eqref{Q-computation} at $x_{ki}=0$ and canceling 
		$R_{(k,2-2i)}\,x_{k(1-i)}^2$ gives $Q \equiv R_{(kl,2-2j,2-2i)} \pmod{x_{ki}}$.
		The identical argument specializing at $x_{lj}=0$ gives 
		$Q \equiv R_{(lk,2-2i,2-2j)} \pmod{x_{lj}}$. By symmetry $R_{(kl,2-2j,2-2i)}=R_{(lk,2-2i,2-2j)}$, so we may take
		\[
		Q = R_{(kl,2-2j,2-2i)},
		\]
		On the other hand, from (\ref{Q-computation}) one can compute $H$ as
		\begin{equation}\label{fracH}
			H = \frac{R_{(k,2-2i)}R_{(l,2-2j)} - R_{(kl,2-2j,2-2i)}\,F}{x_{ki}x_{lj}}.
		\end{equation}
		With this representation we prove that $H$ has degree $1$ in the $l$-th component. We expand the numerator of equation \ref{fracH} using $R_{(k,2-2i)} = R_{(lk,2-2i,0)}x_{l0}^2 + R_{(lk,2-2i,1)}x_{l0}x_{l1} + R_{(lk,2-2i,2)}x_{l1}^2$ and $F = R_{(l,0)}x_{l0}^2 + R_{(l,1)}x_{l0}x_{l1} + R_{(l,2)}x_{l1}^2$. Thus,
		\[
		\begin{aligned}
			x_{ki}x_{lj}H = & \left(R_{(lk,2-2i,2j)}R_{(kl,2-2j,2i)}-R_{(kl,2-2j,2-2i)}R_{(lk,2i,2j)}\right)x_{lj}^2x_{ki}^2 \\
			& +\left(R_{(lk,2-2i,1)}R_{(kl,2-2j,2i)}-R_{(kl,2-2j,2-2i)}R_{(lk,2i,1)}\right)x_{l0}x_{l1}x_{ki}^2 \\
			& +\left(R_{(lk,2-2i,2j)}R_{(kl,2-2j,1)}-R_{(kl,2-2j,2-2i)}R_{(lk,1,2j)}\right)x_{lj}^2x_{k0}x_{k1} \\
			& +\left(R_{(lk,2-2i,1)}R_{(kl,2-2j,1)}-R_{(kl,2-2j,2-2i)}R_{(lk,1,1)}\right)x_{l0}x_{l1}x_{k0}x_{k1}.
		\end{aligned}
		\]
		From this expansion, we observe that the numerator of equation~\ref{fracH} has degree $2$ in the $l$-th component, hence  $H$ has degree $1$ in the $l$-th component. 
		
		Observe that the ideal $I_1$ is generated by homogeneous elements of multidegree 
		\[
		(2, \dots, 2,0_k, 2, \dots, 2, n+1).
		\]
		In particular, for any $l \neq k$, every generator of $I_1$ has degree $2$ in the $l$-th component, whereas the section $A-H$ has degree $1$ in that component by definition of $\mathbf{d}$ (since $d_l = 1$). 
		Because we are working in a positively graded ring, an element of degree $1$ cannot be expressed as a combination of elements of degree $2$. Therefore, it follows immediately that $A - H \notin I_1$. However, by \eqref{Q-computation}:
		\[
		x_{ki}x_{lj}(A-H) = R_{(k,2-2i)}R_{(l,2-2j)}|_X - x_{ki}x_{lj}H|_X 
		= \bigl(R_{(k,2-2i)}R_{(l,2-2j)} - x_{ki}x_{lj}H\bigr)\big|_X = QF|_X,
		\]
		which lies in $\langle F\rangle \subseteq I_1$.  Since $x_{ki}x_{lj}$ divides $m_1$, multiplying by
		$m_1/x_{ki}x_{lj}\in\widetilde{\mathcal{R}}_1$ gives that $m_1(A-H)\in I_1$ implying that $A-H\in\overline{I}_1$.
	\end{proof}
	
	\medskip
	We now have all the ingredients for the first main result of this section. Using the notation established above, Theorem \ref{thm:Nef-coordinate} gives a description of $\mathcal R_0$. This corresponds to the $s=1$ graded part of Theorem \ref{thm:cox-limit}.
	
	\medskip
	\begin{theorem}\label{thm:Nef-coordinate}
		Let $X$ be a Wehler type hypersurface, and let $\widetilde{\mathcal{R}}_1$, $I_1$, $J_1$, and $\overline{I}_1$ be defined as above. Define
		\[
		\mathcal{R}_1'
		:=
		\widetilde{\mathcal{R}}_1/\overline{I}_1 .
		\]
		
		Then the $\Nef(X)$-graded section ring $\mathcal{R}_0$ is isomorphic to the nonnegatively graded part of $\mathcal{R}_1'$.
	\end{theorem}
	\begin{proof}
		Regard $\widetilde{\mathcal{R}}_1$ as $\Cl(X)$-graded by $\deg(x_{ki})=H_k$,
		$\deg(\iota_k^*(x_{ki}))=\iota_k^*H_k$, $\deg(y_l)=H_{m+1}$, and let
		$$\Phi:\widetilde{\mathcal{R}}_1\to\Cox(X)$$ be the induced graded morphism. Since $F$ and
		$x_{ki}\iota_k^*(x_{ki})-R_{(k,2-2i)}$ vanish on $X$, we have $I_1\subseteq\ker\Phi$. By
		Propositions~\ref{surjectivity} and \ref{zeta-relation}, the image of $\Phi$
		in each degree $D\in\Nef(X)\cap\Cl(X)$ equals $H^0(X,\mathcal{O}_X(D))$. Thus it remains to prove
		$\ker\Phi=\overline{I}_1$. For notational convenience, we denote $\widetilde{\mathcal{R}}_1$ as $\mathcal{A}$ throughout the proof.
		
		Let $\mathcal{A}_{m_1}$ be the localization of $\mathcal{A}$ at the element $m_1$, and
		$\pi:\mathcal{A}\to\mathcal{A}_{m_1}$ the localization map. For any ideal
		$\mathfrak{a}\subseteq\mathcal{A}$ one has
		\[
		\pi^{-1}(\mathfrak{a}\,\mathcal{A}_{m_1})
		=\{f\mid m_1^{N}f\in\mathfrak{a}\text{ for some }N\ge0\}
		=(\mathfrak{a}:J_1^{\infty}),
		\]
		in particular $\mathfrak{a}=\pi^{-1}(\mathfrak{a}\,\mathcal{A}_{m_1})$ precisely when
		$\mathfrak{a}$ is $J_1$-saturated, i.e., $(\mathfrak{a}:J_1^{\infty})=\mathfrak{a}$.
		
		By construction $\overline{I}_1=\pi^{-1}(I_1\,\mathcal{A}_{m_1})$. The ideal
		$\ker\Phi$ is $J_1$-saturated as well, since $\Cox(X)$ is a domain and $\Phi(m_1)\ne0$, and so every
		relation $m_1^{N}f\in\ker\Phi$ implies $\Phi(m_1)^{N}\Phi(f)=0$ and so $\Phi(f)=0$. Thus
		$\ker\Phi=\pi^{-1}\bigl((\ker\Phi)\,\mathcal{A}_{m_1}\bigr)$. Consequently it suffices to prove
		\[
		(\ker\Phi)\,\mathcal{A}_{m_1}=I_1\,\mathcal{A}_{m_1}.
		\]
		
		Let $\mathcal{B}=\mathbb{K}[x_{ki}^{\pm1},y_0,\dots,y_n\mid 1 \le k \le m,\ i\in\{0,1\}\,]$. After localizing, the relations $x_{ki}\iota_k^*(x_{ki})=R_{(k,2-2i)}$ allow us to eliminate the
		variables $\iota_k^*(x_{ki})$. This yields a homomorphism
		\[
		\bar\phi:\mathcal{A}_{m_1}\longrightarrow  \mathcal{B}/\langle F\rangle,
		\qquad \iota_k^*(x_{ki})\longmapsto \frac{R_{(k,2-2i)}}{x_{ki}},
		\]
		whose kernel is precisely $I_1\,\mathcal{A}_{m_1}$, hence
		$\mathcal{A}_{m_1}/I_1\mathcal{A}_{m_1}\xrightarrow{\ \sim\ }
		\mathcal{B}/\langle F\rangle$.
		
		Let $S=\mathbb{K}[x_{ki},y_0,\dots,y_n\mid 1 \le k \le m,\ i\in\{0,1\}\,]$.
		For every divisor $D\subset \pp$ the ideal-sheaf sequence
		\[
		0\to\mathcal{O}_{\mathbb P}(D-\operatorname{div} F)\xrightarrow{\cdot F}\mathcal{O}_{\mathbb P}(D)
		\to\mathcal{O}_X(D)\to0
		\]
		gives
		\[0\to S_{D-\operatorname{div}F}\xrightarrow{\cdot F}S_D\to H^0(X,\mathcal{O}_X(D)),
		\]
		so the restriction map
		$S\to\Cox(X)$ has kernel $\langle F\rangle$ and $S/\langle F\rangle\hookrightarrow\Cox(X)$,
		localizing yields $$\mathcal{B}/\langle F\rangle\hookrightarrow\Cox(X)_{m_1}.$$
		Since the $\Phi(x_{ki})$ become invertible in $\Cox(X)_{m_1}$, $\Phi$ extends to
		$\Phi_{m_1}:\mathcal{A}_{m_1}\to\Cox(X)_{m_1}$ and factors as
		\[
		\mathcal{A}_{m_1}\twoheadrightarrow \mathcal{A}_{m_1}/I_1\mathcal{A}_{m_1}
		\xrightarrow{\ \sim\ }\mathcal{B}/\langle F\rangle
		\hookrightarrow\Cox(X)_{m_1}.
		\]
		As the last map is injective, $\ker\Phi_{m_1}=I_1\mathcal{A}_{m_1}$ and by exactness of
		localization, then $\ker\Phi_{m_1}=(\ker\Phi)\mathcal{A}_{m_1}$. Hence
		$(\ker\Phi)\mathcal{A}_{m_1}=I_1\mathcal{A}_{m_1}$, which consequently gives
		$\ker\Phi=\overline{I}_1$. Therefore
		$\mathcal{A}/\overline{I}_1\xrightarrow{\sim}\operatorname{im}(\Phi)\subseteq\Cox(X)$, and in each
		degree $D\in\Nef(X)\cap\Cl(X)$ this identifies $(\mathcal{A}/\overline{I}_1)_D
		=H^0(X,\mathcal{O}_X(D))$. Hence the nonnegatively graded part of $\mathcal{R}_1'$ coincides with $\mathcal{R}_0$.
	\end{proof}
	
	\medskip
	\begin{remark}\label{rem:explicit-relations}
		The identity $\ker\Phi=\overline{I}_1$ established in
		Theorem~\ref{thm:Nef-coordinate} shows that the specific
		relations exhibited in Propositions~\ref{zeta-relation} are contained in $\overline{I}_1$. Recall that $\zeta_k=x_{k0}\iota_k^*(x_{k1})$ in
		$\Cox(X)$.
		
		For the quadratic relation of Proposition~\ref{zeta-relation},
		multiplying by $x_{k1}^2$ gives
		\[
		x_{k1}^2\bigl(\zeta_k^2+R_{(k,1)}\zeta_k+R_{(k,0)}R_{(k,2)}\bigr)
		=R_{(k,0)}\bigl(R_{(k,0)}x_{k0}^2+R_{(k,1)}x_{k0}x_{k1}+R_{(k,2)}x_{k1}^2\bigr)
		=R_{(k,0)}F\in I_1.
		\]
		Since $x_{k1}^2$ divides $m_1^{2}$, multiplying by $m_1^{2}/x_{k1}^2\in\widetilde{\mathcal{R}}_1$ yields
		$m_1^{2}\bigl(\zeta_k^2+R_{(k,1)}\zeta_k+R_{(k,0)}R_{(k,2)}\bigr)\in I_1$, hence
		$\zeta_k^2+R_{(k,1)}\zeta_k+R_{(k,0)}R_{(k,2)}\in\overline{I}_1$.
		
		For the Vieta relation \eqref{Vieta-relation}, set
		\[
		G_k\;:=\;x_{k0}\,\iota_k^*(x_{k1})+x_{k1}\,\iota_k^*(x_{k0})+R_{(k,1)}.
		\]
		Multiplying by $x_{k0}x_{k1}$ and reducing modulo the Vieta-type
		generators of $I_1$ via
		$x_{k1}\,\iota_k^*(x_{k1})\equiv R_{(k,0)}$ and
		$x_{k0}\,\iota_k^*(x_{k0})\equiv R_{(k,2)}$,
		\[
		x_{k0}x_{k1}\,G_k
		\;\equiv\;R_{(k,0)}x_{k0}^2+R_{(k,1)}x_{k0}x_{k1}+R_{(k,2)}x_{k1}^2
		\;=\;F\;\pmod{I_1}.
		\]
		Since $F\in I_1$ and $x_{k0}x_{k1}$ divides $m_1$, multiplying by
		$m_1/x_{k0}x_{k1}\in\widetilde{\mathcal{R}}_1$ gives $m_1\,G_k\in I_1$, so
		$G_k\in\overline{I}_1$.
	\end{remark}
	
	\medskip
	Although Theorem~\ref{thm:Nef-coordinate} provides a presentation of the nef coordinate ring $\mathcal{R}_0$, determining generators for the saturated ideal $\overline{I}_1$ is computationally cumbersome. The next theorem replaces this presentation with an explicit generators--relations description of $\mathcal{R}_0$, thereby avoiding saturation altogether.
	
	Let
	\[
	\mathcal{A}_0
	:=\mathbb{K}[\,x_{ki},\,\zeta_k,\,y_j\mid 1\le k\le m,\ i\in\{0,1\},\ 0\le j\le n\,],
	\]
	graded by $\Cl(X)$ with $\deg x_{ki}=H_k$, $\deg y_j=H_{m+1}$, and
	\[
	\deg\zeta_k=\delta_k:=\operatorname{div}F-2H_k=2\Big(\!\!\sum_{j\neq k}H_j\Big)+(n+1)H_{m+1},
	\]
	which in coordinates reads $\delta_k=(2,\dots,2,0_k,2,\dots,2,n+1)$. Let
	$\mathcal{I}_1\subseteq\mathcal{A}_0$ be the ideal generated by the following homogeneous
	elements:
	\begin{itemize}
		\item $\zeta_k^2+R_{(k,1)}\zeta_k+R_{(k,0)}R_{(k,2)}$, for each $k$;
		\item $x_{k1}\zeta_k-R_{(k,0)}x_{k0}$ and $x_{k0}\zeta_k+R_{(k,1)}x_{k0}+R_{(k,2)}x_{k1}$, for each $k$;
		\item $\zeta_k\zeta_l-x_{k0}x_{l0}H_{kl}$, for each $k<l$, where $H_{kl}\in S_{\mathbf{d}}$ is the
		polynomial of Proposition~\ref{norelation} with $i=j=1$, i.e., $H_{kl}|_X=\iota_k^*(x_{k1})\iota_l^*(x_{l1})$.
	\end{itemize}
	Set $\widetilde{R}_0:=\mathcal{A}_0/\mathcal{I}_1$.
	
	\begin{theorem}\label{thm:Nef-no-saturation}
		The natural graded map $\mathcal{A}_0\to\Cox(X)$ sending each generator to its image in
		$\Cox(X)$ descends to an isomorphism of
		$\Cl(X)$-graded $\mathbb{K}$-algebras
		$\overline{\Psi}\colon\widetilde{R}_0\xrightarrow{\ \sim\ }\mathcal{R}_0$.
	\end{theorem}
	
	\begin{proof}
		Each generator of $\mathcal{I}_1$ vanishes in $\Cox(X)$. Indeed $F|_X=0$ by definition of $X$. The quadratic relations 
		Proposition~\ref{multiplication-involution}, the quadratic relations $\zeta_k^2+R_{(k,1)}\zeta_k+R_{(k,0)}R_{(k,2)}$, for each $k$ vanish by Proposition~\ref{zeta-relation}(1). The linear relations follow from $x_{k1}\zeta_k=x_{k0}\,x_{k1}\iota_k^*(x_{k1})=R_{(k,0)}x_{k0}$ and the analogous identity for $x_{k0}$ in Proposition~\ref{multiplication-involution}, and the binomial relations vanish by Proposition~\ref{norelation}(1). 
		Hence, the natural map descends to a graded homomorphism $\overline{\Psi}\colon\widetilde{R}_0\to\mathcal{R}_0$, since
		every generator of $\mathcal{A}_0$ has degree in $\Nef(X)\cap\Cl(X)$.
		
		\medskip
		Fix $D=(a_1,\dots,a_m,b)\in\Nef(X)\cap\Cl(X)$ and write
		$h(D)=\dim_{\mathbb{K}}H^0(X,\mathcal{O}_X(D))$ and $Z(D)=\{k\;|\;a_k=0\}$. Let $r\colon S\to\Cox(X)$ be the restriction map. The ideal-sheaf
		sequence computations in Proposition~\ref{surjectivity}, together with
		$H^1(\mathbb{P},\mathcal{O}(D))=0$ for $D\in\Nef(X)\cap\Cl(X)$, give the exact sequence
		\[
		0\to r(S_D)\to H^0(X,\mathcal{O}_X(D))\to H^1(\mathbb{P},\mathcal{O}(D-\operatorname{div}F))\to0.
		\]
		The kernel of $r$ on $S_D$ is $\langle F\rangle_D\cong S_{D-\operatorname{div}F}$, so
		$\dim_{\mathbb{K}}r(S_D)=\dim S_D-\dim S_{D-\operatorname{div}F}$, while by \eqref{eq:H^1} 
		$H^1(\mathbb{P},\mathcal{O}(D-\operatorname{div}F))$ vanishes unless $Z(D)=\{k\}$ for some $k$, in which
		case it is isomorphic to $S_{D-\delta_k}$. Therefore
		\begin{equation}\label{eq:Cox-hilb}
			h(D)=\dim S_D-\dim S_{D-\operatorname{div}F}+\sum_{k\in Z(D)}\dim S_{D-\delta_k},
		\end{equation}
		and at most one summand is nonzero, as $(D-\delta_k)_{k'}=-2$ for distinct $k,k'\in Z(D)$.
		
		We first show $\overline{\Psi}_D$ is surjective. The composite
		$S\hookrightarrow\mathcal{A}_0\xrightarrow{\Psi}\Cox(X)$ equals $r$, so
		$r(S_D)\subseteq\overline{\Psi}\bigl((\widetilde{R}_0)_D\bigr)$. If
		$H^1(\mathbb{P},\mathcal{O}(D-\operatorname{div}F))=0$, then $r$ is surjective onto
		$H^0(X,\mathcal{O}_X(D))$ and we are done. Otherwise $Z(D)=\{k\}$ with $a_i,b>0$ for $i\neq k$, and
		Proposition~\ref{zeta-relation}(2) gives
		$H^0(X,\mathcal{O}_X(D))=r(S_D)\oplus\zeta_k\cdot r(S_{D-\delta_k})$. Since
		$\zeta_k\cdot r(f)=\overline{\Psi}(\zeta_k f)$ for $f\in S_{D-\delta_k}$, the second summand also lies
		in the image. In both cases $\overline{\Psi}_D$ is surjective, so
		$\dim_{\mathbb{K}}(\widetilde{R}_0)_D\ge h(D)$.
		
		\medskip
		For the reverse inequality set $\overline{S}:=S/\langle F\rangle$.
		For every $k$, we observe that
		\[
		x_{k1}(x_{k0}\zeta_k+R_{(k,1)}x_{k0}+R_{(k,2)}x_{k1})-x_{k0}(x_{k1}\zeta_k-R_{(k,0)}x_{k0})=F,
		\]
		implying that $F\in\mathcal{I}_1$.  Hence, the map
		$S\hookrightarrow\mathcal{A}_0\to\widetilde{R}_0$ factors through a graded ring homomorphism $\sigma\colon\overline{S}\to\widetilde{R}_0$, making
		$\widetilde{R}_0$ a $\overline{S}$-algebra. 
		Its remaining generators are $\zeta_1,\dots,\zeta_m$, so $\widetilde{R}_0$ is spanned as an $\overline{S}$-module by the
		monomials $\zeta^{\mathbf{e}}=\prod_{k=1}^m\zeta_k^{e_k}$ with $\mathbf{e}=(e_1,\dots,e_m)\in\mathbb{Z}_{\ge0}^m$, where $\mathbf{e}_1,\dots,\mathbf{e}_m$ denotes the
		standard basis of $\mathbb{Z}^m$. 
		We claim every $\zeta^{\mathbf{e}}$ lies in the $\overline{S}$-submodule $$N:=\sigma(\overline{S})+\sum_{k=1}^m\sigma(\overline{S})\,\zeta_k,$$ arguing by induction on
		$|\mathbf{e}|=\sum_k e_k$. The cases $|\mathbf{e}|\le1$ hold by definition of $N$. 
		Let $|\mathbf{e}|\ge2$. If $e_k\ge2$ for some $k$, then $\zeta^{\mathbf{e}}=\zeta_k^2\,\zeta^{\mathbf{e}-2\mathbf{e}_k}$, and the quadratic relation gives in $\widetilde{R}_0$
		\[
		\zeta^{\mathbf{e}}
		=-\,\sigma(\overline{R_{(k,1)}})\,\zeta^{\mathbf{e}-\mathbf{e}_k}
		-\,\sigma(\overline{R_{(k,0)}R_{(k,2)}})\,\zeta^{\mathbf{e}-2\mathbf{e}_k}.
		\]
		Otherwise, if $e_i\le1$ for all $i$, it follows that $e_k=e_l=1$ for some $k<l$. Hence
		$\zeta^{\mathbf{e}}=\zeta_k\zeta_l\,\zeta^{\mathbf{e}-\mathbf{e}_k-\mathbf{e}_l}$, and the binomial
		relation gives
		\[
		\zeta^{\mathbf{e}}=\sigma(\overline{x_{k0}x_{l0}H_{kl}})\,\zeta^{\mathbf{e}-\mathbf{e}_k-\mathbf{e}_l}.
		\]
		In either case $\zeta^{\mathbf{e}}$ is an $\overline{S}$-combination of monomials $\zeta^{\mathbf{e}'}$
		with $|\mathbf{e}'|<|\mathbf{e}|$, which lie in $N$ by induction. As $N$ is an $\overline{S}$-module,
		$\zeta^{\mathbf{e}}\in N$, hence $\widetilde{R}_0=N$. Taking degree-$D$ parts,
		\[
		(\widetilde{R}_0)_D
		\subseteq\sigma(\overline{S}_D)+\sum_{k=1}^m\sigma(\overline{S}_{D-\delta_k})\,\zeta_k.
		\]
		By the linear relations, $x_{k1}\zeta_k=\sigma(\overline{R_{(k,0)}x_{k0}})$ and
		$x_{k0}\zeta_k=\sigma(\overline{-R_{(k,1)}x_{k0}-R_{(k,2)}x_{k1}})$, so
		$x_{ki}\zeta_k\in\sigma(\overline{S})$. Thus if a monomial $m$ of degree $D-\delta_k$ is divisible by
		some $x_{ki}$, writing $m=x_{ki}m'$ gives
		$\sigma(\overline{m})\zeta_k=\sigma(\overline{m'})\,x_{ki}\zeta_k\in\sigma(\overline{S}_D)$. Modulo $\sigma(\overline{S}_D)$, the summand $\sigma(\overline{S}_{D-\delta_k})\zeta_k$ is spanned by the classes $\sigma(\overline{m})\zeta_k$, where $m$ ranges over monomials of degree $D-\delta_k$ having $x_k$-degree zero.
		
		Since $(\delta_k)_k=0$, the condition that $m$ has $x_k$-degree zero forces $(D-\delta_k)_k=a_k=0$, hence $k\in Z(D)$. Therefore, whenever such monomials exist, every monomial of degree $D-\delta_k$ automatically has $x_k$-degree zero. Since monomials of degree $D-\delta_k$ form a basis of $S_{D-\delta_k}$, there are exactly $\dim_{\mathbb K}S_{D-\delta_k}$ of them.
		Combining this with
		$\dim_{\mathbb{K}}\sigma(\overline{S}_D)\le\dim_{\mathbb{K}}\overline{S}_D
		=\dim S_D-\dim S_{D-\operatorname{div}F}$ and \eqref{eq:Cox-hilb},
		\[
		\dim_{\mathbb{K}}(\widetilde{R}_0)_D
		\le\dim S_D-\dim S_{D-\operatorname{div}F}+\sum_{k\in Z(D)}\dim S_{D-\delta_k}=h(D).
		\]
		
		Both inequalities give $\dim_{\mathbb{K}}(\widetilde{R}_0)_D=h(D)
		=\dim_{\mathbb{K}}\mathcal{R}_{0,D}$, so the surjection $\overline{\Psi}_D$ is an isomorphism. As both
		$\widetilde{R}_0$ and $\mathcal{R}_0$ are supported on $\Nef(X)\cap\Cl(X)$, the map
		$\overline{\Psi}$ is an isomorphism of $\Cl(X)$-graded $\mathbb{K}$-algebras.
	\end{proof}

	\section{Cox rings of Wehler type hypersurfaces}\label{sec:cox-Wehler}
	
	In this section we describe the Cox ring of a Wehler type hypersurface from two complementary perspectives. First, following the framework of Theorem~\ref{thm:direct-limit-mov}, we realize $\Cox(X)$ as a filtered direct limit of finitely generated $\kk$-algebras arising from chamber section rings. This provides a description of $\Cox(X)$ in terms of an increasing system of finitely generated pieces.
	
	We then turn to explicit presentations. Using wall-crossing relations between chambers, we obtain generators and relations for $\Cox(X)$, establishing the presentation stated in Theorem~\ref{thm:cox-intro}. 
	
	\medskip
	We begin with the direct-limit description. For the reader's convenience, we recall its statement.

	\begin{theorem}\label{thm:cox-limit-section6}
		Let $X$ be a Wehler type hypersurface and let $\{\iota_k\}_{k=1}^{m}$ be the generators of $\Bir(X)$. For each $s\ge 0$, let $W_s$ denote the set of reduced words of length at most $s$ in the generators $\iota_k$, and define
		\[
		\widetilde{\mathcal R}_s
		:=
		\mathbb K\bigl[\,w^{*}(x_{ki}),\,y_0,\dots,y_n \;\bigm|\; w\in W_s,\; 1\le k\le m,\; i\in\{0,1\}\,\bigr].
		\]
		
		For $s\ge 1$, set
		\[
		I_s
		:=
		\Bigl\langle\, w^{*}\!\bigl(x_{ki}\,\iota_k^{*}(x_{ki})-R_{(k,2-2i)}\bigr),\; w^{*}(F)
		\;\Bigm|\; w\in W_{s-1},\; 1\le k\le m,\; i\in\{0,1\}\,\Bigr\rangle,
		\]
		and define
		\[
		m_s
		:=
		\!\!\!\prod_{\substack{w\in W_{s-1}\\ 1\le k\le m,\;i\in\{0,1\}}}
		\!\!\!
		w^{*}(x_{ki}),
		\qquad
		J_s
		:=
		\langle m_s\rangle
		\subseteq
		\widetilde{\mathcal R}_s.
		\]
		
		Let
		\[
		\overline I_s
		:=
		(I_s:J_s^{\,\infty}),
		\qquad
		\mathcal R'_s
		:=
		\widetilde{\mathcal R}_s/\overline I_s.
		\]
		
		With the convention $I_0=\langle F\rangle$, $m_0=1$, and $J_0=\langle 1\rangle$, the Cox ring of $X$ is isomorphic to the direct limit
		\[
		\Cox(X)\;\cong\;\varinjlim_s \mathcal R'_s.
		\]
	\end{theorem}
	
	Let $\varphi_s\colon\widetilde{\mathcal R}_s\to\Cox(X)$ be the evaluation
	homomorphism sending each formal symbol $w^{*}(x_{ki})$ to the corresponding
	pullback section and $y_l$ to itself.
	By Proposition~\ref{multiplication-involution} every generator of $I_s$
	vanishes in $\Cox(X)$, so $I_s\subseteq\ker\varphi_s$. Since $\Cox(X)$ is a
	domain (Lemma~\ref{domain}) and $\varphi_s(m_s)\ne 0$, the kernel
	$\ker\varphi_s$ is $J_s$-saturated; combined with $I_s\subseteq\ker\varphi_s$
	this yields
	\[
	\overline I_s\;\subseteq\;\ker\varphi_s,
	\]
	and $\varphi_s$ descends to $\overline\varphi_s\colon\mathcal R'_s\to\Cox(X)$.
	
	\begin{proof}[Proof of Theorem~\ref{thm:cox-limit-section6}]
		We divide the proof into four steps.
		
		\medskip
		\emph{Step 1. Cox ring as a direct limit:}  In this step we show that $\Cox(X)$ arises as a direct limit of finitely generated rings. Let
		$\operatorname{Bir}(X)$ be the birational automorphism group of \(X\). By Theorem \ref{thm:Wehler-bir}, $\operatorname{Bir}(X)$ is generated by involutions \(\iota_1,\dots,\iota_m\).
		
		For each integer $s\ge 0$ let $W_s\subset \operatorname{Bir}(X)$ be the set of elements
		that can be written as a reduced word of length at most $s$ in the fixed
		generating set $\{\iota_1,\dots,\iota_m\}$. For each $s$, recall the definition
		\[
		\mathcal{K}_s := \bigcup_{w\in W_s} w^*(\Nef(X)).
		\]
		It follows that $\mathcal{K}_s$ is a convex rational polyhedral cone by \cite[Proposition~8]{Vinberg}. Additionally, by the Cone theorem~\ref{Kawamata-Morrison}, the translates $w^*(\Nef(X))$ cover the effective cone. Hence
		\[
		\Eff(X)\
		\;=\; \bigcup_{s\ge 0} \mathcal{K}_s.
		\]
		For each $s$ we define the $\mathcal{K}_s$–graded section ring as
		\[
		\mathcal R_s=
		\bigoplus_{[D]\in \mathcal{K}_s\cap \Cl(X)} H^0\bigl(X,\mathcal{O}_X(D)\bigr)
		\;\subset\; \Cox(X).
		\]
		The inclusions $\mathcal K_s\subset \mathcal K_{s+1}$ induce inclusions
		$\mathcal R_s\subset \mathcal R_{s+1}$, and the union of the $\mathcal K_s$ is all of
		$\Eff(X)$.  Therefore, every homogeneous summand 
		$H^0(X,\mathcal{O}_X(D))\subset \Cox(X)$ with $[D]\in \Eff(X)$ is contained in some
		$\mathcal{R}_s$ with  $s\geq 0$. In particular, we obtain
		\[
		\Cox(X)
		\;=\; \bigoplus_{[D]\in \Eff(X)\cap \Cl(X)}
		H^0\bigl(X,\mathcal{O}_X(D)\bigr)
		\;\cong\; \varinjlim_s \mathcal R_s ,
		\]
		which proves the direct-limit statement.
		
		\emph{Step 2. Generation of $\mathcal{R}_s$:}  Let $[D]\in \mathcal K_{s}\cap \Cl(X)$ be an effective class. Then by Theorem~\ref{Kawamata-Morrison} there exists a nef divisor $L$ with
		$[D]=w^*[L] = (\iota_l v)^*[L]$ for $w\in W_s$ and some $v\in W_{s-1}$. Each involution $\iota_k$ is an isomorphism in codimension one, hence for any Cartier divisor $M$, the pull-back induces an isomorphism on global
		sections
		\[
		\iota_k^*: H^0\bigl(X,\mathcal{O}_X(M)\bigr)
		\xrightarrow{\;\sim\;}
		H^0\bigl(X,\mathcal{O}_X(\iota_k^*M)\bigr).
		\]
		By composing, we also get an isomorphism
		\[
		w^* = (\iota_l v)^*:
		H^0\bigl(X,\mathcal{O}_X(L)\bigr)
		\xrightarrow{\;\sim\;}
		H^0\bigl(X,\mathcal{O}_X(D)\bigr).
		\]
		
		By Propositions~\ref{surjectivity} and ~\ref{zeta-relation}
		the $\mathbb{K}$-vector space $H^0(X,\mathcal{O}_X(L))$ is generated by monomials in the sections $x_{ki}$, $y_j$ and $\zeta_k$.
		Applying $w^*$, we see that $H^0(X,\mathcal{O}_X(D))$ is generated by the monomials in the sections $w^*(x_{ki})$,
		$w^*(\zeta_k)$, and $y_j$.
		Hence $\mathcal R_{s}$ is generated by the set $$\{w^*(x_{ki}),w^*(\zeta_k),y_{0},\dots,y_{n}|\ 1\le k\le m,i\in\{0,1\},\ w\in W_s\}.$$
		
		\medskip
		\emph{Step 3. Injection of $\mathcal{R}^\prime_s$ in $\Cox(X)$:}
		We argue by induction on $s$ that $\ker\varphi_s\subseteq\overline I_s$.
		
		For $s=0$ we have that $I_0=\overline I_0=\langle F\rangle$; the ideal-sheaf sequence used in the proof
		of Theorem~\ref{thm:Nef-coordinate} gives $\ker\varphi_0=\langle F\rangle$. 
		
		For $s=1$, we have $\widetilde{\mathcal R}_1=\mathcal A$, as defined in the previous section.
		Theorem~\ref{thm:Nef-coordinate} yields the equality $\ker\varphi_1=\overline I_1$.
		
		Now, let $s\ge2$ and assume that $\ker\varphi_{s-1}=\overline I_{s-1}$ holds. Write $W_s^{\circ}:=W_s\setminus W_{s-1}$ for the words of length exactly $s$. Every $w\in W_s^{\circ}$ has a reduced expression $w=v\iota_a$ with $v$ of length exactly $s-1$. For an index $a$, the involution $\iota_a$ fixes $x_{ki}$ for $a\neq k$, so $w^{*}(x_{ki})=v^{*}(x_{ki})\in\widetilde{\mathcal R}_{s-1}$.
		Hence, the variables of $\widetilde{\mathcal R}_s$ not already in $\widetilde{\mathcal R}_{s-1}$ are
		\[
		(v\iota_k)^{*}(x_{ki}),\qquad v\iota_k\in W_s^{\circ},\ 1\le k\le m,\ i\in\{0,1\}.
		\]
		For each such variable the Vieta-type generator of $I_s$ indexed by $v\in W_{s-1}$ gives
		\begin{equation}\label{eq:vieta}
			v^{*}(x_{ki})\cdot(v\iota_k)^{*}(x_{ki})\;\equiv\;v^{*}\!\bigl(R_{(k,2-2i)}\bigr)
			\pmod{I_s},
		\end{equation}
		with $v^{*}(R_{(k,2-2i)})\in\widetilde{\mathcal R}_{s-1}$.
		
		Let $P\in\ker\varphi_s$,  written as a polynomial in the new variables with coefficients in $\widetilde{\mathcal R}_{s-1}$, and let $d_{v,k,i}\ge0$ be the exponent of $P$ in $(v\iota_k)^{*}(x_{ki})$.
		Set
		\[
		A\;:=\;\!\!\!\prod_{v\iota_k\in W_s^{\circ}}\!\!\! v^{*}(x_{ki})^{\,d_{v,k,i}}\in\ \widetilde{\mathcal R}_{s-1}.
		\]
		Pairing each occurrence of $(v\iota_k)^{*}(x_{ki})$ in $AP$ with a factor $v^{*}(x_{ki})$ of $A$ and applying \eqref{eq:vieta} replaces every new variable, so
		\begin{equation}\label{eq:descend}
			AP\;\equiv\;H\pmod{I_s},\qquad H\in\widetilde{\mathcal R}_{s-1}.
		\end{equation}
		Applying $\varphi_s$ to \ref{eq:descend} and using $\varphi_s(P)=0$ gives $\varphi_s(H)=0$. Since $\varphi_s|_{\widetilde{\mathcal R}_{s-1}}=\varphi_{s-1}$, it follows that $H\in\ker\varphi_{s-1}=\overline I_{s-1}$ by induction.
		Hence $m_{s-1}^{N}H\in I_{s-1}\subseteq I_s$ for some $N$. Multiplying
		\eqref{eq:descend} by $m_{s-1}^{N}$ gives
		\[
		(m_{s-1}^{N}A)\cdot P
		\;=\; m_{s-1}^{N}H \;+\; m_{s-1}^{N}\!\cdot\!(AP-H)
		\;\in\; I_s.
		\]
		The monomial $m_{s-1}^{N}A$ involves only factors of $m_s$, so it divides $m_s^{M}$ for some $M$. Writing $m_s^{M}=(m_{s-1}^{N}A)\,\tau$ yields
		$m_s^{M}\cdot P=\tau(\!\cdot\!m_{s-1}^{N}A)\!\cdot\!P\in I_s$. Therefore
		$P\in(I_s:J_s^{\,\infty})=\overline I_s$, which completes the induction.
		
		Since $\ker\varphi_s=\overline I_s$ for all $s$, it follows that
		$\overline\varphi_s\colon\mathcal R'_s\to\Cox(X)$ is injective for all $s$.
		
		\medskip
		\emph{Step 4. $\Cox(X)$ as a direct limit of $\mathcal R'_s$:}
		Viewing $\mathcal R'_s$ and $\mathcal R'_{s+1}$ as $\mathbb K$-subalgebras of $\Cox(X)$, we claim that
		\[
		\mathcal R'_s\subset\mathcal R_s\subset\mathcal R'_{s+1}.
		\]
		Each inclusion is verified on a generating set. The subalgebra
		$\operatorname{im}\varphi_s$ is generated by the sections $w^{*}(x_{ki})$ with $w\in W_s$, of degree $w^{*}(H_k)\in\mathcal K_s$, together with $y_0,\dots,y_n$, of degree $H_{m+1}\in\Nef(X)\subseteq\mathcal K_s$.
		Since every generator is homogeneous of degree in $\mathcal K_s$, this yields $\mathcal R'_s\subset\mathcal R_s$. 
		For the second inclusion, Step~2 shows that $\mathcal R_s$ is generated by the sections $w^{*}(x_{ki})$, $w^{*}(\zeta_k)$
		and $y_l$ with $w\in W_s$, each of which lies in $\operatorname{im}\varphi_{s+1}$. Thus, $\mathcal R_s\subset\mathcal R'_{s+1}$. This chain of inclusions combined with Step 1 give the desired isomorphism
		\[
		\Cox(X)\;\cong\;\varinjlim_s\mathcal R'_s.
		\]
	\end{proof}
	
	Now, we proceed to give the presentation of $\Cox(X)$ given in Theorem \ref{thm:cox-intro}. The underlying idea is to pull back the ideal
	$\mathcal{I}_1$ of Theorem~\ref{thm:Nef-no-saturation} along every $w\in \Bir(X)$ and adjoin a set of wall-crossing relations. Each such relation rewrites a product of generators coming from two different chambers as a polynomial in generators whose presenters have strictly smaller length.
	The iteration of this process reduces every relation of $\Cox(X)$ to one of the nef chamber. This process is developed in Proposition~\ref{prop:reduce}.
	
	\medskip
	We work in the polynomial ring
	\[
	\mathcal{A}_{\infty}:=\mathbb{K}\bigl[\,w^{*}(x_{ki}),\ w^{*}(\zeta_k),\ y_j \ \bigm|\
	w\in \Bir(X),\ 1\le k\le m,\ i\in\{0,1\},\ 0\le j\le n\,\bigr],
	\]
	graded by $\Cl(X)$, with evaluation homomorphism
	$\Phi\colon \mathcal{A}_{\infty}\to \Cox(X)$. Since $\iota_j$ fixes
	$x_{ki}$ for $j\ne k$, the section $w^{*}(x_{ki})$ depends only on the coset of $w$ modulo $\langle \iota_j:j\ne k\rangle$. Thus, we index each generator $w^{*}(x_{ki})$ and $w^{*}(\zeta_k)$ by the minimal-length representative of that coset, so that
	$w=\mathrm{id}$ or $w$ ends in $\iota_k$. For $w\in \Bir(X)$, let $\ell(w)$ denote the reduced word length of $w$, i.e., the number of letters in any reduced expression for $w$ (see for example \cite[Section 5.2]{Hum90}).

	\medskip
	By the Cone Theorem~\ref{Kawamata-Morrison} the effective cone is tiled by the
	chambers $K_w=w^*\Nef(X)$, $w\in \Bir(X)$. Their extreme rays are
	$w^*H_1,\dots,w^*H_{m+1}$ (Lemma~\ref{Neron-Severi}), where $H_{m+1}$ is fixed by $\Bir(X)$. We use the basis $H_1,\dots,H_{m+1}$ of $N^1(X)_{\mathbb{R}}$ to
	define coordinates on each chamber. 
	For $1\le k\le m$, set $$\lambda_{0,k}\!\bigl(\sum_j c_j H_j\bigr):=c_k.$$ 
	This is the supporting functional of the wall between $K_0$ and $K_{\iota_k}$, normalized $\ge0$ on $K_0$. 
	The $\Bir(X)$-action transports $\lambda_{0,k}$ to a general chamber. The wall functional between $K_w$ and $K_{\iota_kw}$ is given by
	\begin{equation}\label{eq:dualfunctional}
		\lambda_{w,k}(\xi)\ :=\ \lambda_{0,k}\!\bigl((w^{-1})^{*}\xi\bigr),
		\qquad \xi\in N^1(X)_{\mathbb{R}}.
	\end{equation}
	
	\medskip
	A product $w^*(x_{ki}) w'^*(x_{lj})$ whose degree $w^*H_k + w'^*H_l$ lies in the nef chamber $\Nef(X)$ can be expressed as an element of the nef graded ring $\widetilde{R}_0$ via Theorem~\ref{thm:Nef-no-saturation}, but this expression is not visible inside $\mathcal{A}_\infty$ until a relation is imposed. Similarly, with the equality $\zeta_k-x_{k0}\iota_k^{*}(x_{k1})=0$. In the next definition we record those relations.
	
	\medskip
	\begin{definition}\label{def:wall}
		Let $w,w'\in \Bir(X)$ with $w^{*}H_k+w'^{*}H_l\in\Nef(X)$. By
		Theorem~\ref{thm:Nef-no-saturation} there is a unique coset
		$H_{w,w',ki,lj}+\mathcal{I}_1\in\widetilde{R}_0$ with
		$\Phi(H_{w,w',ki,lj})=\Phi\bigl(w^{*}(x_{ki})\,w'^{*}(x_{lj})\bigr)$. The
		\emph{wall-crossing ideal} $I_{\mathrm{wall}}\subseteq\mathcal{A}_{\infty}$ is
		generated by
		\[
		u^{*}\bigl(w^{*}(x_{ki})\,w'^{*}(x_{lj})\bigr)-u^{*}(H),
		\qquad u\in \Bir(X),\ H\in H_{w,w',ki,lj}+\mathcal{I}_1,
		\]
		and
		\[
		u^*\bigl(\zeta_k-x_{k0}\iota_k^{*}(x_{k1})\bigr),
		\qquad u\in \Bir(X).
		\]
	\end{definition}
	
	\medskip
	To order monomials under reduction modulo $I_{\mathrm{wall}}$, we measure each $x$-type factor by the word length of the element presenting it.
	
	\begin{definition}\label{def:height}
		Let $M=\prod_{m=1}^N w_m^{*}(x_{k_m i_m})$ be a monomial with factors indexed by
		minimal presenters, i.e., each $w_m$ is the minimal-length representative of its coset modulo
		$\langle\iota_j:j\ne k_m\rangle$. Then its \emph{height} $\operatorname{ht}(M)$ is the finite
		multiset $\{\ell(w_1),\dots,\ell(w_N)\}$. We define
		$\operatorname{ht}(M)<\operatorname{ht}(M')$ by arranging both multisets in
		non-increasing order, extending the shorter sequence by zeros until both have the
		same length, and comparing the resulting sequences lexicographically.
	\end{definition}
	
	\begin{example}\label{ex:height}
		Let $m=2$, $n=2$, so $X\subset\mathbb{P}^1\times\mathbb{P}^1\times\mathbb{P}^2$ and $\Bir(X)=\langle \iota_1,\iota_2\rangle$. Consider the monomials
		\[
		M = (\iota_1\iota_2)^*(x_{20})\iota_1^*(x_{10}), \qquad
		M' = \iota_1^*(x_{10})x_{10}.
		\]
		Their minimal presenters have lengths $\{2,1\}$ and $\{1,0\}$, so $\operatorname{ht}(M)=\{2,1\}$ and $\operatorname{ht}(M')=\{1,0\}$.
		Arranging each in non-increasing order gives the sequences $(2,1)$ and $(1,0)$, yielding $\operatorname{ht}(M')<\operatorname{ht}(M)$.
	\end{example}
	
	\begin{remark}\label{rem:height-wellf}
		The height $\operatorname{ht}(M)$ consists only of zeros if and only if every factor of $M$ has a length-$0$ presenter, i.e.\ $M\in\mathcal{A}_0$ as in Theorem \ref{thm:Nef-no-saturation}. The order of Definition~\ref{def:height} is well founded, i.e., no infinite strictly descending
		chain exists. 
	\end{remark}
	
	\medskip
	The reduction in Proposition~\ref{prop:reduce} proceeds by pairing the factor $w_\alpha^{*}(x_{k_\alpha i_\alpha})$ of maximal $\ell(w_\alpha)$ with another factor lying on the opposite side of the wall functional $\lambda_{0,a}$.
	The following lemma establishes the sign behavior of pulled-back extreme rays needed to construct this pairing. For a reduced word $w=\iota_k\cdots \iota_b\neq 1$ we say $\iota_k$ is the initial letter and $\iota_b$ the ending letter of $w$. 
	
	\medskip
	\begin{lemma}\label{lem:pos}
		Let $w\ne1$ be a reduced word ending in $\iota_b$, with first letter $\iota_k$, and write
		$w^{*}H_k=\sum_j c_j H_j$.  Then $c_b\le-1$ and $c_j\ge1-c_b$ for all $j\ne b$.
		In particular $c_j\ge1$ for every $j\ne b$.
	\end{lemma}
	
	\begin{proof}
		We induct on $\ell(w)$.  For $\ell(w)=1$ we have $w=\iota_k$, so $b=k$, and
		Theorem~\ref{involution-action} gives $c_k=-1$, $c_j=2$ for $j\le m$ with $j\ne k$,
		and $c_{m+1}=n+1\ge2$. Therefore, the claim holds.  For $\ell(w)\ge2$, write $w=w'\iota_b$ with
		$w'$ reduced ending in $\iota_{b'}\ne \iota_b$ and first letter $\iota_k$; by induction
		$w'^{*}H_k=\sum_jc'_jH_j$ satisfies $c'_{b'}\le-1$ and $c'_j\ge1-c'_{b'}\ge2$
		for $j\ne b'$, in particular $c'_b\ge2$.  Applying $\iota_b^*$
		(Theorem~\ref{involution-action}),
		\[
		c_b=-c'_b,\qquad
		c_j=c'_j+2c'_b\ (j\le m,\ j\ne b),\qquad
		c_{m+1}=c'_{m+1}+(n+1)c'_b,
		\]
		so $c_b\le-2$, and for each $j\ne b$ one checks $c_j-(1-c_b)=c_j-1-c'_b\ge0$.
		This equals $c'_{b'}+c'_b-1\ge0$ when $j=b'$, equals
		$(c'_j-1)+c'_b>0$ when $j\le m$ is distinct from $b$ and $b'$, and equals $c'_{m+1}+nc'_b-1\ge0$ when $j=m+1$.
	\end{proof}
	
	\begin{lemma}\label{lem:cross}
		Let $w,w'\in \Bir(X)$ with $\ell(w)\ge\ell(w')\ge 1$, and set $L:=\ell(w)$.  Let
		\[
		A=w^{*}H_{k},\qquad B=w'^{*}H_{l}
		\]
		be extreme rays indexed by their minimal presenters; that is, $w=\iota_k w_0$ and $w'=\iota_l w_0'$ are reduced, with first letters $\iota_k,\iota_l$ and $k,l\le m$.
		Write $w=v\iota_a$ in reduced form.
		If $\lambda_{0,a}(B)>0$, then $A+B\in\overline{K_u}$ for some $u\in \Bir(X)$ with $\ell(u)<L$.
	\end{lemma}
	\begin{proof}
		Since $\Bir(X)$ is the free product of the $\langle \iota_k\rangle$ (Theorem~\ref{thm:Wehler-bir}), the chamber graph of $\mathcal{C}:=\Eff(X)$ is a tree in which the graph distance between $K_0$ and $K_w$ equals $\ell(w)$ for every $w\in\Bir(X)$, i.e.\ the number of walls separating $K_0$ from $K_w$ equals the number of generators in any reduced expression for $w$.
		A \emph{gallery} is a sequence of chambers in which consecutive chambers share a wall.
		It is \emph{reduced} if its length equals the graph distance between its endpoints \cite[Section 5.2]{Hum90}.
		Any two chambers are joined by a unique reduced gallery.
		Each wall has the form $\ker\nu\cap\mathcal{C}$ and is the wall of a single edge; deleting that edge splits the tree into the half-spaces $\{\nu\le0\}$ and $\{\nu\ge0\}$, and each chamber closure lies in one of them.
		We call this \emph{sign separation} (cf. \cite[Section 5.13]{Hum90}).
		
		\medskip
		We split the proof into four steps:
		\medskip
		
		\emph{Step 1.} Let $F_a=\ker \lambda_{0,a}\cap\mathcal{C}$ be the wall of the edge $(K_0,K_{\iota_a})$.
		As $w=v\iota_a$ is reduced, the geodesic from $K_0$ to $K_w$ crosses $F_a$ at its first step, so $\overline{K_w}\subseteq\{\lambda_{0,a}\le0\}$.
		Similarly, since $\lambda_{0,a}(B)>0$ and $B\in\overline{K_{w'}}$, sign separation gives $\overline{K_{w'}}\subseteq\{\lambda_{0,a}\ge0\}$.
		Thus $K_w$ and $K_{w'}$ lie in opposite half-spaces of $F_a$.
		
		The gallery from $K_w$ to $K_{w'}$ must pass through $K_0$ and decomposes as the concatenation of the geodesics $[K_w,K_0]$ and $[K_0,K_{w'}]$.
		Lengths decrease from $L$ to $0$ along the first path and increase from $0$ to $\ell(w')$ along the second.
		Hence every gallery chamber, except for $K_w$ and, when $\ell(w')=L$, possibly $K_{w'}$, has length $<L$.
		
		\medskip
		\emph{Step 2.} Let $[w,w']$ denote the set of chambers on the reduced gallery from $K_w$ to $K_{w'}$.
		By standard results in gallery intervals (see for example \cite[Proposition 2.39]{AM17}), it follows that $$\Sigma:=\bigcup_{q\in[w,w']}\overline{K_q}$$ is convex.
		Since $A\in\overline{K_w}$ and $B\in\overline{K_{w'}}$ lie in $\Sigma$, so does $\tfrac12(A+B)$.
		Thus, $A+B\in\overline{K_q}$ for some gallery chamber $q$.
		
		\medskip
		\emph{Step 3.} We exclude the possibility that $A+B\in\overline{K_w}$.
		Let $F=\overline{K_{w_0}}\cap\overline{K_w}$ be the wall adjacent to $K_w$ on the geodesic from $K_w$ to $K_0$, and let $\lambda:=\lambda_{w_0,k}$ be the supporting functional normalized by $\lambda\le0$ on $K_w$ and $\lambda\ge0$ on $K_{w_0}$.
		Since $F$ separates the chamber tree into the half-spaces $\{\lambda\le0\}$ and $\{\lambda\ge0\}$, and the gallery from $K_w$ to $K_{w'}$ crosses $F$ only at its first step, the chamber $K_{w'}$ lies in the $\{\lambda\ge0\}$ half-space.
		Hence $\lambda(B)\ge0$.
		
		\medskip
		We now compute $\lambda(A)$ in order to determine the sign of $\lambda(A+B)$.
		By \eqref{eq:dualfunctional}, $(w_0^{-1})^{*}w^{*}=\iota_k^{*}$, and since $\iota_k$ is the first letter of $w$, we have $\iota_k^{*}H_k=\delta_k-H_k$.
		Therefore
		\[
		\lambda(A)=\lambda_{0,k}\bigl((w_0^{-1})^{*}w^{*}H_k\bigr)
		=\lambda_{0,k}(\iota_k^{*}H_k)
		=\lambda_{0,k}(\delta_k-H_k)
		=-1 .
		\]
		
		If $\lambda(B)=0$, then $B\in F\subseteq\overline{K_w}\subseteq\{\lambda_{0,a}\le0\}$, contradicting $\lambda_{0,a}(B)>0$.
		Since $\lambda$ is integral and $B\in\Cl(X)$, it follows that $\lambda(B)\ge1$.
		Hence
		\[
		\lambda(A+B)\ge -1+1 =0 .
		\]
		Since $\operatorname{int}K_w\subseteq\{\lambda<0\}$ and $\overline{K_w}\subseteq\{\lambda\le0\}$, either $\lambda(A+B)>0$, in which case $A+B\notin\overline{K_w}$, or $\lambda(A+B)=0$.
		In the latter case,
		\[
		A+B\in\overline{K_w}\cap\{\lambda=0\}\subseteq\overline{K_{w_0}},
		\]
		where $\ell(w_0)=L-1$.
		Hence $A+B\in\overline{K_q}$ for a gallery chamber $q\ne w$.
		
		\medskip
		\emph{Step 4.} Assume $\ell(w')=L$.
		We exclude the possibility that $A+B\in\overline{K_{w'}}$.
		So $w'=\iota_l v_0$ with $\ell(v_0)=L-1$.
		Let $\mu:=\lambda_{v_0,l}$ be normalized so that $\mu\le0$ on $K_{w'}$.
		As in Step~3,
		\[
		\mu(B)=\lambda_{0,l}(\delta_l-H_l)=-1 .
		\]
		It therefore suffices to prove $\mu(A)\ge1$, since then $\mu(A+B)\ge0$, excluding $A+B$ from $\operatorname{int}K_{w'}$ exactly as in Step~3.
		
		Put $g:=wv_0^{-1}$.
		We first determine the position of $K_{v_0}$ relative to the wall $F_a$.
		By Step~1, both $K_{w'}$ and $K_0$ lie in $\{\lambda_{0,a}\ge0\}$, so the geodesic from $K_{w'}$ to $K_0$ does not cross $F_a$.
		Since $K_{v_0}$ lies on this geodesic, it follows that
		\[
		\overline{K_{v_0}}\subseteq\{\lambda_{0,a}\ge0\}.
		\]
		On the other hand, Step~1 gives $\overline{K_w}\subseteq\{\lambda_{0,a}\le0\}$.
		
		Hence $K_{v_0}$ and $K_w$ lie in opposite half-spaces of $F_a$, the geodesic $[K_{v_0},K_w]$ passes through $K_0$, and additivity of tree distance gives
		\[
		\ell(g)=d(K_{v_0},K_w)
		=d(K_{v_0},K_0)+d(K_0,K_w)
		=\ell(v_0)+\ell(w)
		=2L-1 .
		\]
		So $g$ is reduced of length $2L-1$ and begins with $\iota_k$.
		
		If $L=1$, then $v_0=\mathrm{id}$ and $g=w=\iota_a=\iota_k$.
		Since $\lambda_{0,a}(B)>0$ and $B=\iota_l^{*}H_l$, Theorem~\ref{involution-action} gives $\lambda_{0,a}(B)=-1$ if $a=l$. Thus $a\ne l$. It follows that the last letter $\iota_k$ of $g$ satisfies $k\ne l$.
		
		If $L\ge2$, writing $v_0=t_1\cdots t_r$ reduced, the last letter of $g$ is $t_1$. Since $w'=\iota_l v_0$ is reduced, $t_1\ne \iota_l$.
		In both cases the last letter of $g$ differs from $\iota_l$.
		
		By \eqref{eq:dualfunctional},
		\[
		\mu(A)=\lambda_{0,l}(g^{*}H_k)=c_l(g^{*}H_k).
		\]
		Since $g$ is reduced, begins with $\iota_k$, and ends in a letter $\iota_b l$ with $b\neq l$, we may apply Lemma~\ref{lem:pos} to obtain $c_l(g^{*}H_k)\ge1$.
		Hence $\mu(A)\ge1$.
		
		By Step~3, $A+B\notin\overline{K_w}$.
		If $\ell(w')=L$, Step~4 shows that whenever $A+B$ lies on the boundary of $K_{w'}$, it lies in the adjacent chamber $\overline{K_{v_0}}$ with $\ell(v_0)=L-1$.
		Since every other chamber along the gallery has length $<L$, we conclude that $A+B\in\overline{K_u}$ for some $u$ with $\ell(u)<L$.
	\end{proof}
	\medskip
	\begin{remark}\label{rem:cross-nef-partner}
		If the partner ray is nef, i.e.\ $w'=\mathrm{id}$ and $B=H_l$ with $\lambda_{0,a}(H_l)>0$, then necessarily $l=a$ and only Step~3 of Lemma~\ref{lem:cross} is used. The conclusion $A+B\in\overline{K_u}$ with $\ell(u)<L$ holds verbatim, the exclusion of $K_{w'}=K_0$ being vacuous since
		$\ell(w')=0<L$.
	\end{remark}
	
	For the reader's convenience, we recall the notation established in Theorem \ref{thm:Nef-no-saturation}. We defined
	$$\mathcal{A}_0 :=\mathbb{K}[\,x_{ki},\,\zeta_k,\,y_j\mid 1\le k\le m,\ i\in\{0,1\},\ 0\le j\le n\,].$$
	The kernel of the natural map $\mathcal{A}_0\to \Cox(X)$ was denoted by $\mathcal{I}_1$.
	
	\medskip
	\begin{proposition}\label{prop:reduce}
		Let $P\in\mathcal{A}_{\infty}$ be homogeneous with $\deg P\in\Nef(X)$.
		Then there exists some $Q\in\mathcal{A}_0$ such that
		$$P\equiv Q\pmod{I_{\mathrm{wall}}}.$$
	\end{proposition}
	
	\begin{proof}
		We first eliminate every factor of the form $w^{*}(\zeta_k)$ with $\ell(w)\ge1$.
		By Definition~\ref{def:wall}, $\zeta_k-x_{k0}\iota_k^{*}(x_{k1})\in I_{\mathrm{wall}}$,
		so for $\ell(w)\ge1$,
		\[
		w^{*}(\zeta_k)\equiv w^{*}(x_{k0})\,(\iota_kw)^{*}(x_{k1})\pmod{I_{\mathrm{wall}}}.
		\]
		Thus, by applying this reduction iteratively, we obtain
		\[
		P\equiv P'\pmod{I_{\mathrm{wall}}},\qquad
		P'\in\mathbb{K}[\,w^{*}(x_{ki}),\,\zeta_k,\,y_0,\dots,y_n \;\bigm|\; w\in \Bir(X),\; 1\le k\le m,\; i\in\{0,1\}\,\bigr].
		\]
		
		\medskip
		By linearity it suffices to treat a single monomial of degree $\deg P$,
		\[
		M=\Bigl(\prod_{\alpha=1}^{N}w_\alpha^{*}(x_{k_\alpha i_\alpha})\Bigr)M_0,
		\qquad
		M_0=\prod_{k}\zeta_k^{\,a_k}\prod_{j}y_j^{\,b_j}\quad(a_k,b_j\in\mathbb{Z}_{\ge0}),
		\]
		where the first product contains every $x$-type factor of $M$ and $M_0$ contains
		the untranslated $\zeta_k$ and the $y_j$.
		With $D:=\deg M_0$,
		\[
		\sum_{\alpha=1}^{N}w_\alpha^{*}H_{k_\alpha}+D=\deg P\in\Nef(X).
		\]
		We prove, by induction on $\operatorname{ht}(M)$ in the order of
		Definition~\ref{def:height}, the statement
		\begin{equation}\label{eq:mon-induction}
			M\equiv Q_M\pmod{I_{\mathrm{wall}}}\ \text{ for some }Q_M\in\mathcal{A}_0,
		\end{equation}
		for every such monomial $M$ of $P^\prime$.
		
		\emph{Base case.}
		If $\ell(w_\alpha)=0$ for all $\alpha$, every factor of $M$ is a length-$0$
		generator, so $M\in\mathcal{A}_0$ and \eqref{eq:mon-induction} holds with $Q_M=M$.
		
		\emph{Inductive step.}
		Assume $\operatorname{ht}(M)$ has a positive entry and that \eqref{eq:mon-induction}
		holds for all monomials of degree in $\Nef(X)$ and strictly smaller height.
		Choose $\alpha$ with $L:=\ell(w_\alpha)\ge1$ maximal, and write
		$w_\alpha=v\iota_a$ in reduced form.
		Set $A:=w_\alpha^{*}H_{k_\alpha}$, so $\lambda_{0,a}(A)\le-1$ by Lemma~\ref{lem:pos}.
		Since $\lambda_{0,a}\ge0$ on $\Nef(X)$ and $\deg M\in\Nef(X)$, with $D:=\deg M_0$,
		\[
		\sum_{\beta\ne\alpha}\lambda_{0,a}\bigl(w_\beta^{*}H_{k_\beta}\bigr)+\lambda_{0,a}(D)\ge1,
		\]
		with all terms integers.
		Hence either $\lambda_{0,a}(w_\beta^{*}H_{k_\beta})\ge1$ for some $\beta\ne\alpha$,
		or $\lambda_{0,a}(D)\ge1$.
		We study both cases separately.
		\begin{itemize}[leftmargin=2.2em]
			\item[(i)] If $\lambda_{0,a}(w_\beta^{*}H_{k_\beta})\ge1$ for some $\beta\ne\alpha$,
			set $w':=w_\beta$, $\eta:=w_\beta^{*}(x_{k_\beta i_\beta})$ and
			$B:=w_\beta^{*}H_{k_\beta}$.
			It follows that $\ell(w')\le L$.
			\item[(ii)] Otherwise $\lambda_{0,a}(D)\ge1$.
			Since $\deg\zeta_k=\delta_k$ and $\deg y_j=H_{m+1}$, we obtain
			\[
			\lambda_{0,a}(D)
			=\sum_{k}a_k\,\lambda_{0,a}(\delta_k)
			+\Bigl(\sum_j b_j\Bigr)\lambda_{0,a}(H_{m+1})
			=2\Big(\!\!\sum_{k\ne a}\!a_k\Big),
			\]
			since $\lambda_{0,a}(\delta_k)=2$ for $k\ne a$, $\lambda_{0,a}(\delta_a)=0$, and
			$\lambda_{0,a}(H_{m+1})=0$.
			Every summand is nonnegative, so $\lambda_{0,a}(D)\ge1$ forces $a_{k'}\ge1$ for
			some $k'\ne a$, meaning $M_0$ contains a factor $\zeta_{k'}$ with $k'\ne a$.
			We expand this factor inside $M$ using the relation
			\[
			\zeta_{k'}\equiv x_{k'0}\,\iota_{k'}^{*}(x_{k'1})\pmod{I_{\mathrm{wall}}},
			\]
			replacing $\zeta_{k'}$ in $M_0$ by $x_{k'0}\cdot\iota_{k'}^{*}(x_{k'1})$.
			Set $w':=\iota_{k'}$, $\eta:=\iota_{k'}^{*}(x_{k'1})$ and $B:=\iota_{k'}^{*}H_{k'}$.
			By Lemma~\ref{lem:pos}, $\lambda_{0,a}(B)=2$ since $a\ne k'$, and $\ell(w')=1\le L$.
		\end{itemize}
		
		\medskip
		In both cases $\lambda_{0,a}(B)\ge1$ and $\ell(w')\le L$.
		By Lemma~\ref{lem:cross} (or Remark~\ref{rem:cross-nef-partner} when $B\in\Nef(X)$)
		there is $u\in\Bir(X)$ with
		\[
		A+B\in\overline{K_u},\qquad \ell(u)<L.
		\]
		The product $w_\alpha^{*}(x_{k_\alpha i_\alpha})\,\eta$ has degree
		$A+B\in\overline{K_u}$, so $(u^{-1})^{*}(A+B)\in\Nef(X)$.
		By Definition~\ref{def:wall}, there exists $H\in\mathcal{A}_0$ with
		$\Phi(H)=\Phi\bigl((u^{-1})^*(w_\alpha^{*}(x_{k_\alpha i_\alpha})\,\eta)\bigr)$
		and the relation
		\[
		w_\alpha^{*}(x_{k_\alpha i_\alpha})\,\eta\;-\;u^{*}(H)\;\in\;I_{\mathrm{wall}}.
		\]
		Setting $G:=u^{*}(H)\in u^{*}\mathcal{A}_0$,
		\[
		w_\alpha^{*}(x_{k_\alpha i_\alpha})\,\eta\equiv G\pmod{I_{\mathrm{wall}}},
		\qquad\deg G=A+B.
		\]
		As in the beginning of the proof, after reducing the translated $\zeta$'s in $G$, we may take
		\[
		G\in\mathbb{K}[\,w^{*}(x_{ki}),\,\zeta_k,\,y_j\,]\quad
		\text{with every $x$-type factor $w^{*}(x_{ki})$ satisfying $\ell(w)\le\ell(u)<L$.}
		\]
		
		Write $M=w_\alpha^{*}(x_{k_\alpha i_\alpha})\,\eta\cdot R$, where $R$ is the product
		of the remaining factors of $M$ (in case~(ii), $R$ contains the remaining $x_{k'0}$).
		Multiplying the congruence above by $R$ and writing $G=\sum_i c_i g_i$,
		\[
		M\equiv\sum_i c_i\,(g_i R)\pmod{I_{\mathrm{wall}}},\qquad M_i:=g_i R,\quad
		\deg M_i=\deg M\in\Nef(X).
		\]
		We compare the factor structure of $M_i$ with that of $M$.
		The monomial $M_i=g_iR$ is obtained from $M$ by replacing the pair
		$w_\alpha^{*}(x_{k_\alpha i_\alpha})\,\eta$ with $g_i$, while all factors of $R$
		remain unchanged.
		By construction every $x$-type factor $w^{*}(x_{ki})$ of $g_i$ satisfies $\ell(w)<L$, so the factors of $M_i$ originating from $g_i$ contribute only entries $<L$ to $\operatorname{ht}(M_i)$.
		
		It remains to account for $\eta$.
		In Case~(i), $\eta=w_\beta^{*}(x_{k_\beta i_\beta})$ was an existing factor of $M$.
		It is part of the replaced pair and does not appear in $M_i$, removing the entry $\ell(w_\beta)\le L$ from the height multiset $\operatorname{ht}(M_i)$.
		In Case~(ii), $\eta=\iota_{k'}^{*}(x_{k'1})$ was introduced by expanding
		$\zeta_{k'}\equiv x_{k'0}\,\iota_{k'}^{*}(x_{k'1})$; it is the second factor of the pair $w_\alpha^{*}(x_{k_\alpha i_\alpha})\,\eta$ replaced by $G$, hence it does not appear in $M_i$.
		The factor $x_{k'0}$, a length-$0$ generator, is the only remaining factor from the expansion of $\zeta_{k'}$, and it appears in $R$.
		
		In both cases the entry $L$ contributed by $w_\alpha^{*}(x_{k_\alpha i_\alpha})$
		is removed from $\operatorname{ht}(M_i)$ and no entry $\ge L$ is introduced, so
		$\operatorname{ht}(M_i)<\operatorname{ht}(M)$.
		
		\medskip
		By the induction hypothesis each $M_i\equiv Q_i\pmod{I_{\mathrm{wall}}}$ with
		$Q_i\in\mathcal{A}_0$; hence $M\equiv\sum_i c_i Q_i\in\mathcal{A}_0$.
		The order is well founded, so the induction terminates, and summing over the
		monomials of $P'$ produces $Q\in\mathcal{A}_0$ with
		$P\equiv Q\pmod{I_{\mathrm{wall}}}$.
	\end{proof}

	We now provide the presentation of $\Cox(X)$. We use the wall-crossing ideal $I_{\mathrm{wall}}\subseteq\mathcal{A}_{\infty}$ of Definition~\ref{def:wall}.
	The defining generators of $I_{\mathrm{wall}}$ are permuted among themselves by the pullback action of $\Bir(X)$ on $\mathcal{A}_{\infty}$, so $I_{\mathrm{wall}}$ is invariant under pullback: $u^{*}(I_{\mathrm{wall}})=I_{\mathrm{wall}}$ for every $u\in\Bir(X)$.
	Moreover, fixing $u,w,w',k,l,i,j$ and subtracting the generators attached to two representatives of the coset $H_{w,w',ki,lj}+\mathcal{I}_1$ produces $u^{*}(h)$ for an arbitrary $h\in\mathcal{I}_1$, it follows that $u^{*}\mathcal{I}_1\subseteq I_{\mathrm{wall}}$ for every $u\in\Bir(X)$. In particular $\mathcal{I}_1\subseteq I_{\mathrm{wall}}$.
	
	\begin{theorem}\label{thm:cox-pres}
		The evaluation homomorphism $\Phi$ induces an isomorphism of $\Cl(X)$--graded $\mathbb{K}$--algebras
		\[
		\mathcal{A}_{\infty}/I_{\mathrm{wall}}\ \xrightarrow{\ \sim\ }\ \Cox(X).
		\]
	\end{theorem}
	
	\begin{proof}
		First, we observe that $\Phi$ is surjective. Indeed, by Step 2 of Theorem \ref{thm:cox-limit-section6}, the section ring
		$\mathcal{R}_s$ is generated by the set of sections $$\{w^*(x_{ki}),w^*(\zeta_k),y_{0},\dots,y_{n}|\ 1\le k\le m,i\in\{0,1\},\ w\in W_s\}.$$
		Since $\Cox(X)\cong \varinjlim_s \mathcal R_s$ the surjectivity of $\Phi$ follows.
		
		\medskip
		We proceed to prove that $I_{\mathrm{wall}}\subseteq\ker\Phi$.
		It suffices to check the generators of Definition~\ref{def:wall}.
		For the wall generators, every representative $H$ of the coset $H_{w,w',ki,lj}+\mathcal{I}_1$ satisfies $\Phi(H)=\Phi\bigl(w^{*}(x_{ki})\,w'^{*}(x_{lj})\bigr)$, since $\mathcal{I}_1\subseteq\ker\Phi$ by Theorem~\ref{thm:Nef-no-saturation}; applying $u^{*}$ and the compatibility $\Phi\circ u^{*}=u^{*}\circ\Phi$ gives $\Phi\bigl(u^{*}(w^{*}(x_{ki})\,w'^{*}(x_{lj}))-u^{*}(H)\bigr)=0$.
		For the relations $u^{*}\bigl(\zeta_k-x_{k0}\iota_k^{*}(x_{k1})\bigr)$, by definition it follows that $\Phi(\zeta_k)=\Phi(x_{k0})\,\Phi\bigl(\iota_k^{*}(x_{k1})\bigr)$, and applying $u^{*}$ preserves the vanishing.
		Hence every generator lies in $\ker\Phi$, so $I_{\mathrm{wall}}\subseteq\ker\Phi$.
		
		\medskip
		We now prove that $\ker\Phi\subseteq I_{\mathrm{wall}}$.
		Since $\Phi$ is $\Cl(X)$--graded, $\ker\Phi$ is homogeneous; fix a homogeneous $P\in\ker\Phi$ of degree $D$.
		Then $D\in\Eff(X)$, so by Theorem~\ref{Kawamata-Morrison} there is $w\in \Bir(X)$ with $D\in K_w=w^{*}\Nef(X)$, i.e.\ $(w^{-1})^{*}D\in\Nef(X)$.
		The lifted automorphism $(w^{-1})^{*}$ of $\mathcal{A}_{\infty}$ sends $P$ to a homogeneous element $(w^{-1})^{*}P$ of degree $(w^{-1})^{*}D\in\Nef(X)$, and
		\[
		\Phi\bigl((w^{-1})^{*}P\bigr)=(w^{-1})^{*}\Phi(P)=0,
		\]
		so $(w^{-1})^{*}P\in\ker\Phi$.
		Since $(w^{-1})^{*}P$ is homogeneous with degree in $\Nef(X)$, by Proposition~\ref{prop:reduce} there exists $Q\in\mathcal{A}_0$ with
		\[
		(w^{-1})^{*}P\equiv Q\pmod{I_{\mathrm{wall}}}.
		\]
		Since $(w^{-1})^{*}P\in\ker\Phi$ and $I_{\mathrm{wall}}\subseteq\ker\Phi$, we have
		\[
		Q=(w^{-1})^{*}P-\bigl((w^{-1})^{*}P-Q\bigr)\in\ker\Phi,
		\]
		so $Q\in\ker\Phi\cap\mathcal{A}_0$.
		By Theorem~\ref{thm:Nef-no-saturation}, $\ker\bigl(\Phi|_{\mathcal{A}_0}\bigr)=\mathcal{I}_1$, so $Q\in\mathcal{I}_1\subseteq I_{\mathrm{wall}}$. Consequently, we obtain $(w^{-1})^{*}P\in I_{\mathrm{wall}}$.
		Since $I_{\mathrm{wall}}$ is $\Bir(X)$--stable,
		\[
		P=w^{*}\bigl((w^{-1})^{*}P\bigr)\in w^{*}(I_{\mathrm{wall}})=I_{\mathrm{wall}}.
		\]
		
		\medskip
		Both inclusions give $\ker\Phi=I_{\mathrm{wall}}$, and consequently $\Phi$ descends to an isomorphism $\mathcal{A}_{\infty}/I_{\mathrm{wall}}\xrightarrow{\ \sim\ }\Cox(X)$ of $\Cl(X)$--graded $\mathbb{K}$--algebras.
	\end{proof}
	
	\begin{corollary}\label{cor:cox-intro}
		Let $\mathcal{I}\subseteq\mathcal{A}_{\infty}$ be the ideal generated by the $\Bir(X)$--translates $w^{*}\mathcal{I}_1$, the relations $w^{*}\bigl(\zeta_k-x_{k0}\iota_k^{*}(x_{k1})\bigr)$, and the single-representative wall relations $u^{*}\bigl(w^{*}(x_{ki})\,w'^{*}(x_{lj})-H_{w,w',ki,lj}\bigr)$, for $u,w,w'\in\Bir(X)$.
		Then $I_{\mathrm{wall}}=\mathcal{I}$ and consequently Theorem~\ref{thm:cox-intro} follows.
	\end{corollary}
	
	\begin{proof}
		Each generator of $\mathcal{I}$ lies in $I_{\mathrm{wall}}$: the single-representative wall relations are the $H=H_{w,w',ki,lj}$ single cases of Definition~\ref{def:wall}, the $\zeta$-relations are common to both, and $w^{*}\mathcal{I}_1\subseteq I_{\mathrm{wall}}$ by the observation preceding Theorem~\ref{thm:cox-pres}.
		Conversely, every coset wall generator $u^{*}\bigl(w^{*}(x_{ki})\,w'^{*}(x_{lj})\bigr)-u^{*}(H)$ with $H=H_{w,w',ki,lj}+h$, $h\in\mathcal{I}_1$, equals $u^{*}\bigl(w^{*}(x_{ki})\,w'^{*}(x_{lj})-H_{w,w',ki,lj}\bigr)-u^{*}(h)\in \mathcal{I}$, and the $\zeta$-relations lie in $\mathcal{I}$.
		Hence $I_{\mathrm{wall}}=\mathcal{I}$, so $\mathcal{A}_{\infty}/\mathcal{I}\xrightarrow{\sim}\Cox(X)$, which is the presentation of Theorem~\ref{thm:cox-intro}.
	\end{proof}

	\section{F-purity of {Cox(X)}}
	In this section we study the singularities of the Cox ring of a Wehler type
	hypersurface $X\subset(\mathbb{P}^1)^m\times\mathbb{P}^n$ via positive characteristic methods. Our main result, Theorem~\ref{F-puretype}, states that
	$\Cox(X)$ is of dense $F$-pure type. Recall that $F$-purity is the characteristic-$p$ analogue of log canonicity.
	
	We start by defining the integral model of a Wehler type hypersurface $X\subset(\mathbb{P}^1)^m\times\mathbb{P}^n$. Let
	$\mathbb{P}_{\mathbb{Z}}:=(\mathbb{P}^1_{\mathbb{Z}})^m\times_{\Spec\mathbb{Z}}\mathbb{P}^n_{\mathbb{Z}}$, and let $\{m_i\}_{i=1}^N$ be a monomial basis of
	$H^0(\mathbb{P}_{\mathbb{Z}},\mathcal{O}_{\mathbb{P}_{\mathbb{Z}}}(2,\dots,2,n+1))$.
	Set
	\[
	S:=\Spec\mathbb{Z}[t_1,\dots,t_N],\qquad
	\mathbb{P}_S:=\mathbb{P}_{\mathbb{Z}}\times_{\Spec\mathbb{Z}}S.
	\]
	Consider the universal section
	\[
	s_{\mathrm{univ}}:=\sum_{i=1}^N t_i m_i
	\in H^0(\mathbb{P}_S,\mathcal{O}_{\mathbb{P}_S}(2,\dots,2,n+1)),
	\]
	and define $X_{\mathrm{univ}}:=Z(s_{\mathrm{univ}})$, the vanishing subscheme of
	$s_{\mathrm{univ}}$. Each tuple $(c_1,\dots,c_N)\in\mathbb{C}^N$ determines a ring homomorphism
	\[
	\phi\colon\mathbb{Z}[t_1,\dots,t_N]\longrightarrow\mathbb{C},\qquad \phi(t_i)=c_i,
	\]
	and the hypersurface $F=\sum_{i=1}^N c_i m_i=0$ is the base change of
	$X_{\mathrm{univ}}$ along $\Spec\phi$:
	\[
	X_{\mathbb{C}}:=X_{\mathrm{univ}}\times_{S}\Spec\mathbb{C},
	\]
	fitting into the commutative diagram
	\[\begin{tikzcd}
		X_{\mathbb{C}} \arrow[r]\arrow[d]
		& X_{\mathrm{univ}} \arrow[d] \\
		\Spec\mathbb{C} \arrow[r] & S.
	\end{tikzcd}\]
	
	To reduce modulo a prime $p$, we perform the base-change of $X_{\mathrm{univ}}$ along
	$\Spec\mathbb{F}_p\to\Spec\mathbb{Z}$. This yields 
	\[
	X_{\mathbb{F}_p,\mathrm{univ}}:=X_{\mathrm{univ}}\times_{\Spec\mathbb{Z}}\Spec\mathbb{F}_p
	\;\longrightarrow\;\Spec\mathbb{F}_p[t_1,\dots,t_N],
	\]
	where a hypersurface over a finite field $\mathbb{F}_q$ of characteristic $p$ is
	the fiber of $X_{\mathbb{F}_p,\mathrm{univ}}$ along the evaluation morphism
	$\Spec\mathbb{F}_q\to\Spec\mathbb{F}_p[t_1,\dots,t_N]$, $t_i\mapsto\bar c_i$. We
	obtain the commutative diagram
	\[\begin{tikzcd}
		X_{\mathbb{F}_p,\mathrm{univ}} \arrow[r]\arrow[d]
		& X_{\mathrm{univ}} \arrow[d] \\
		\Spec\mathbb{F}_p[t_1,\dots,t_N] \arrow[r] & S.
	\end{tikzcd}\]
	We establish the model adapted to the coefficients of $F$. Set
	$A:=\mathbb{Z}[c_1,\dots,c_N]\subset\mathbb{C}$. For a very general Wehler type
	hypersurface the $c_i$ are algebraically independent over $\mathbb{Q}$, so
	$\phi$ induces an isomorphism $A\cong\mathbb{Z}[t_1,\dots,t_N]$; the hypersurface
	$X_A:=Z\bigl(\sum_{i=1}^N c_i m_i\bigr)\subset\mathbb{P}_A$ is then a model of
	$X_{\mathbb{C}}$ over $\Spec A$, i.e.\ $X_{\mathbb{C}}=X_A\times_{\Spec A}\Spec\mathbb{C}$.
	\begin{definition}[{\cite[Definition~2.13]{GOST15}}]\label{def:dense-F-type}
		\hfill
		\begin{enumerate}
			\item A projective variety $X$ over an $F$-finite field of characteristic $p>0$ is \emph{globally $F$-split} if the Frobenius $\mathcal{O}_X\to F_*\mathcal{O}_X$ splits as
			a map of $\mathcal{O}_X$-modules.
			A ring $R$ of characteristic $p>0$ is \emph{$F$-pure}
			if the Frobenius $R\to F_*R$ is pure, i.e.,\ $M\to M\otimes_R F_*R$ is injective for every $R$-module $M$.
			
			\item Let $X$ be a projective variety (resp.\ $R$ a ring) over a field $\mathbb{K}$ of characteristic zero. We say that $X$ is of \emph{dense globally $F$-split type} (resp.\ $R$ is of \emph{dense $F$-pure type}) if there exist a finitely generated $\mathbb{Z}$-subalgebra $A\subset\mathbb{K}$ and a flat $A$-model $X_A$ of $X$
			(resp.\ $R_A$ of $R$), satisfying $X_A\times_{\Spec A}\Spec\mathbb{K}\cong X$
			(resp.\ $R_A\otimes_A\mathbb{K}\cong R$), such that the set of closed points $\mu\in\Spec A$ for which the fiber $X_\mu:=X_A\times_{\Spec A}\Spec\kappa(\mu)$ is
			globally $F$-split (resp.\ $R_\mu:=R_A\otimes_A\kappa(\mu)$ is $F$-pure) is dense in $\Spec A$.
		\end{enumerate}
	\end{definition}
	
	By the following lemma, Theorem \ref{F-puretype} follows from the fact that the Wehler type hypersurface $X$ is of dense globally $F$-split type. 
	
	\begin{lemma}\emph{(\cite[Lemma~4.5]{GOST15})}
		Let $X$ be a normal projective variety over an $F$-finite field of characteristic $p>0$, with $\Pic^0(X)=0$ and $\Cl(X)$ finitely generated. 
		If $X$ is globally $F$-split, then $\Cox(X)$ is $F$-pure.
	\end{lemma}

	We prove the following statement. 
	\begin{proposition}\label{global-Fsplit}
		Let $X$ be a Wehler type hypersurface of multidegree $(2,\dots,2,n+1)$ over a field of characteristic $0$.
		Then $X$ is of dense globally $F$-split type.
	\end{proposition}
	\begin{proof}
		Set $L=\mathcal{O}_{\mathbb{P}}(2,\dots,2,n+1)=\omega_{\mathbb{P}}^{-1}$, and let $\{m_i\}_{i=1}^N$ be the
		monomial basis of $H^0(\mathbb{P},L)$. Recall the affine parameter space
		$S=\Spec\mathbb{Z}[t_1,\dots,t_N]$ and the universal section
		\[
		F_{\mathrm{univ}}=\sum_{i=1}^N t_i\,m_i\ \in\ H^0\big(\mathbb{P}_S,\,L\big)
		\]
		introduced above. For a point $v$ of $S$ with coordinates
		$(t_i(v))$ we write $F_v=\sum_i t_i(v)\,m_i$ for the corresponding section of $L$ and
		$X_v=Z(F_v)$ for its zero scheme.
		The monomial $$M=\big(\prod_{k=1}^m x_{k0}x_{k1}\big)\big(\prod_{i=0}^n y_i\big)$$
		has multidegree $(2,\dots,2,n+1)$, hence belongs to the basis $\{m_i\}_{i=1}^N$. We relabel the basis so that $m_1=M$.
		
		\medskip
		Let $k$ be an algebraically closed field of characteristic $p>0$, and let
		$F\in H^0(\mathbb{P}_k,L)$ be a section with $Z(F)$ smooth.
		Being a product of projective spaces, $\mathbb{P}_k$ is globally $F$-split by
		\cite[Example~1.1.10.3 and Exercise~1.3.E.8]{BrionKumar}.
		Suppose the coefficient of $M^{p-1}$ in $F^{p-1}$ is nonzero. Since $k$ is algebraically
		closed, we may rescale $F$ by a $(p-1)$-st root of this coefficient so that the coefficient
		equals $1$. Note this leaves $Z(F)$ unchanged. Then, by the splitting criterion
		\cite[Theorem~1.3.8]{BrionKumar} evaluated at a torus-fixed point of $\mathbb{P}_k$, the section
		\[
		F^{p-1}\in H^0(\mathbb{P}_k,\omega_{\mathbb{P}}^{\otimes(1-p)})
		\]
		splits $\mathbb{P}_k$.
		Since $F$ is a section of $\omega_{\mathbb{P}}^{-1}$, the converse part of
		\cite[Proposition~1.3.11]{BrionKumar} then shows that $Z(F)$ is compatibly split in $\mathbb{P}_k$.
		Hence $Z(F)$ is globally $F$-split by \cite[Remark~1.1.4(ii)]{BrionKumar}.
		
		\medskip
		Fix a prime $p$. The coefficient of $M^{p-1}$ in $F_{\mathrm{univ}}^{\,p-1}$ is a polynomial
		$C_p\in\mathbb{Z}[t_1,\dots,t_N]$, homogeneous of degree $p-1$.
		Specializing $t_1=1$ and $t_2=\cdots=t_N=0$ gives $F_{\mathrm{univ}}^{\,p-1}=M^{p-1}$, so
		$C_p(1,0,\dots,0)=1$, i.e., the coefficient of $t_1^{p-1}$ in $C_p$ equals $1$.
		Hence the reduction $c_p\in\mathbb{F}_p[t_1,\dots,t_N]$ of $C_p$ is nonzero, and
		\[
		U_p=\big\{\,t\in\mathbb{A}^N_{\mathbb{F}_p}\ \mid\ c_p(t)\neq0\,\big\}
		\]
		is a dense open subset.
		
		Fix $c_1,\dots,c_N\in\mathbb{C}$ algebraically independent over $\mathbb{Q}$, as holds for a
		very general $X$, and put
		\[
		A=\mathbb{Z}[c_1,\dots,c_N]\ \cong\ \mathbb{Z}[t_1,\dots,t_N],
		\]
		with model $X_A=Z\big(\textstyle\sum_i c_i m_i\big)$ over $\Spec A\cong\mathbb{A}^N_{\mathbb{Z}}$.
		Let $W\subset\Spec A$ be the open locus of smooth fibers. Since $L$ is very ample, Bertini's theorem shows that a general member of $|L|$ over $\overline{\mathbb{F}_p}$ is smooth. Hence $W$ meets every fiber $\mathbb{A}^N_{\mathbb{F}_p}$ in a nonempty, dense open subset.
		Set
		\[
		\Sigma=\big\{\,\mu\in W\ \text{closed}\ \mid\ t(\mu)\in U_{\operatorname{char}\kappa(\mu)}\,\big\}.
		\]
		For $\mu\in\Sigma$ the fiber $X_\mu$ is smooth and the coefficient of $M^{p-1}$ is nonzero, so $X_\mu\otimes\overline{\kappa(\mu)}$ is globally $F$-split by the preceding paragraph, and therefore so is $X_\mu$.
		Any closed subset of $\Spec A$ containing $\Sigma$ meets each fiber $\mathbb{A}^N_{\mathbb{F}_p}$ in a closed set containing $W\cap U_p$, hence equals that fiber, and so equals $\Spec A$.
		Thus $\Sigma$ is dense, and $X$ is of dense globally $F$-split type.
	\end{proof}
	
	\begin{corollary}
		Let $X$ be a Wehler type hypersurface defined over $\mathbb{C}$. Then the Cox ring of $X$ is of dense $F$-pure type.
	\end{corollary}
	
	\FloatBarrier
	\bibliographystyle{habbvr}
	\bibliography{bib}
\end{document}